\documentclass[a4paper,10pt]{amsart}
\usepackage[utf8]{inputenc}
\usepackage{amssymb}
\usepackage{amsmath}
\usepackage{bbm}
\usepackage{mathrsfs}
\usepackage[all]{xy}
\usepackage{geometry}
\usepackage{manfnt}
\usepackage{color}
\usepackage[normalem]{ulem}
\usepackage{graphicx}
\usepackage[colorlinks,citecolor=magenta,linkcolor=blue,urlcolor=blue,filecolor=blue]{hyperref}
\title{A Pirashvili-type theorem for functors on non-empty finite sets}
\author{Geoffrey Powell}
\address{LAREMA, UMR 6093 du CNRS et de  l'Universit\'e d’Angers, SFR MathSTIC, 2 Bd Lavoisier 49045 Angers, France
}
\email{Geoffrey.Powell@math.cnrs.fr}
\urladdr{http://math.univ-angers.fr/~powell/}
\author{Christine Vespa}
\address{Universit\'e de Strasbourg, Institut de Recherche Math\'ematique Avanc\'ee, Strasbourg, France.}
\email{vespa@math.unistra.fr}
\urladdr{http://irma.math.unistra.fr/~vespa/}
\thanks{This work was partially supported by the ANR Project {\em ChroK}, {\tt 
ANR-16-CE40-0003}.}
\newtheorem{THM}{Theorem}
\newtheorem{COR}[THM]{Corollary}

\newtheorem{thm}{Theorem}[section]
\newtheorem{prop}[thm]{Proposition}
\newtheorem{cor}[thm]{Corollary}
\newtheorem{lem}[thm]{Lemma}
\theoremstyle{definition}
\newtheorem{defn}[thm]{Definition}
\newtheorem{exam}[thm]{Example}

\theoremstyle{remark}
\newtheorem{rem}[thm]{Remark}
\newtheorem{nota}[thm]{Notation}
\newtheorem{hyp}[thm]{Hypothesis}

\renewcommand{\epsilon}{\varepsilon}
\renewcommand{\theta}{\vartheta}
\newcommand{\kbar}[1][-]{\overline{\kring[#1]}}

\newcommand{\f}{\mathcal{F}}
\newcommand{\cala}{\mathscr{A}}
\newcommand{\calc}{\mathscr{C}}
\newcommand{\cald}{\mathscr{D}}
\newcommand{\fcatk}[1][\calc]{\f (#1; \kring)}
\newcommand{\nat}{\mathbb{N}}
\newcommand{\zed}{\mathbb{Z}}
\newcommand{\ext}{\mathrm{Ext}}
\newcommand{\sym}{\mathfrak{S}}
\newcommand{\rat}{\mathbb{Q}}
\newcommand{\op}{^\mathrm{op}}
\newcommand{\ob}{\mathsf{Ob}\hspace{3pt}}
\newcommand{\kring}{\mathbbm{k}}

\newcommand{\aut}{\mathrm{Aut}}
\newcommand{\fin}{\mathbf{Fin}}
\newcommand{\finne}{{\overline{\fin}}}
\newcommand{\sets}{\mathbf{Set}}

\newcommand{\n}{\mathbf{n}}

\newcommand{\filt}{\mathfrak{f}}
\renewcommand{\hom}{\mathrm{Hom}}

\newcommand{\dash}{\hspace{-3pt}-\hspace{-3pt}}
\newcommand{\modules}{\mathrm{mod}}
\newcommand{\twist}{t^{* \otimes}}

\newcommand{\extbar}{\overline{\ext}}
\newcommand{\cperm}{C_{\sf{perm}}}
\newcommand{\prt}{\mathfrak{Part}}

\newcommand{\sgn}{\mathsf{sign}}

\newcommand{\sgnrep}{\mathrm{sgn}}
\newcommand{\epfn}{\mathrm{Rk}}
\newcommand{\fout}[1][\gr]{\mathcal{F}^{\mathrm{Out}} (#1; \kring)}
\newcommand{\gr}{\mathbf{gr}}
\newcommand{\A}{\mathfrak{a}}
\newcommand{\ak}{\A_\kring}
\newcommand{\orient}{\mathsf{Or}}
\newcommand{\oclass}{\omega}
\newcommand{\fb}{\mathbf{\Sigma}}
\newcommand{\finj}{\mathbf{FI}}

\newcommand{\zext}{\mathbf{Z}}
\newcommand{\id}{\mathrm{Id}}
\newcommand{\gcmnd}{{\perp^{\hspace{-2pt} \scriptscriptstyle \Gamma}}}
\newcommand{\fbcmnd}{{\perp^{\hspace{-2pt} {\scriptscriptstyle \fb}}}}
\newcommand{\ord}{\mathbf{\Delta}}
\newcommand{\aord}{\ord_{\mathrm{aug}}}
\newcommand{\cosimp}{\mathfrak{C}}

\newcommand{\kz}{\mathsf{Kz}\ }
\newcommand{\cbb}{\mathbb{C}}
\newcommand{\krank}{\mathrm{rank}_\kring}
\newcommand{\Oz}{{\mathbf{\Omega}}}
\newcommand{\dk}{\mathfrak{DK}}
\newcommand{\cochain}{\mathrm{coCh}}
\newcommand{\trunc}{\tau}
\numberwithin{equation}{section}
\begin{document}

\begin{abstract}
Pirashvili’s Dold-Kan type theorem for finite pointed sets follows from the identification in terms of surjections of the morphisms between the tensor powers of a functor playing the role of the augmentation ideal; these functors are projective. 

We give an unpointed analogue of this result: namely,  we compute the morphisms between the tensor powers of the corresponding functor in the unpointed context. We also calculate the Ext groups between such objects, in particular showing that these functors are not projective; this is an important difference between the pointed and unpointed contexts.

This work is motivated by our functorial analysis of the higher Hochschild homology of a wedge of circles.
\end{abstract}

\maketitle
\section{Introduction}

 The aim of this paper is to study the category of functors from the category $\finne$ of finite non-empty sets and {\em all} morphisms to the category of modules over a fixed commutative ring $\kring$. We denote this category of functors by $\fcatk[\finne]$; more generally,  $\fcatk[\calc]$ (or the category of $\calc$-modules) denotes the category  of functors from an essentially small category $\calc$ to $\kring$-modules.

 Our main motivation for studying $\fcatk[\finne]$ comes from our functorial approach to higher Hochschild homology to a wedge of circles  \cite{2018arXiv180207574P}, as indicated later in this introduction.
 
In particular, this motivates our introduction in Definition \ref{defn:kbar} of the objects $\kbar$ of $\fcatk[\finne]$: for $X$ a finite non-empty set, $\kbar[X]$ is the kernel of the natural `augmentation' $\kring[X] \rightarrow \kring$. These are fundamental building blocks of $\fcatk[\finne]$.
 
The main result of the paper is the following:
 
 \begin{THM}[Theorem \ref{thm:complex_Cab}]
\label{THM:main}
For $a,b \in \nat$, there are isomorphisms
\begin{eqnarray*}
\hom_{\fcatk[\finne]}(\kbar^{\otimes a}, \kbar^{\otimes b} ) 
& \cong & 
\left\{ 
\begin{array}{ll}
\kring^{ \oplus \epfn (b,a)} 
& b\geq a >0 \\
0 & b<a ;
\end{array}
\right.
\\
\ext^1 _{\fcatk[\finne]}(\kbar^{\otimes a}, \kbar^{\otimes b} )
&\cong & 
\left\{ 
\begin{array}{ll}
\kring & b = a+1 \\
0 & \mbox{otherwise,}
\end{array}
\right.
\end{eqnarray*}
where $\epfn (b,a)$ is a combinatorially-defined function of $b,a$ (cf. Proposition \ref{cor:inductive_rho}).
\end{THM}

This  leads to the following:

\begin{COR} (Corollary \ref{cor:non-proj_inj})
\label{Cor-Intro}
For $a \in \mathbb{N}$, $\kbar^{\otimes a}$ is not projective in $\fcatk[\finne]$.
\end{COR}

These results should be compared with the \textit{pointed} situation studied previously by Pirashvili, where the category $\finne$ should be replaced by the category $\Gamma$  of finite \textit{pointed} sets.  
In \cite{Phh}, Pirashvili gave an equivalence of categories between  the categories $\fcatk[\Gamma]$ and $\fcatk[\Oz]$ where $\Oz$ is the category of finite sets and surjections. The proof of this result is based on using a certain family of \textit{projective} objects $(t^*)^{\otimes n}$ in  $\fcatk[\Gamma]$, for $n\in \nat$. More precisely, the equivalence follows from  the identification, for $a, b \in \nat$ \cite[(1.10.2)]{Phh}
\begin{equation}\label{DK-Pira-Intro}
\hom_{\fcatk[\Gamma]} ((t^*)^{\otimes a}, (t^*)^{\otimes b}) 
\cong 
\kring \hom_{\Oz} (\mathbf{b}, \mathbf{a})
\end{equation}
where, for $m \in \nat$, 
$\mathbf{m} := \{1, \ldots ,m \}$. We refer to this here as Pirashvili's theorem. 

The forgetful functor 
 $
 \theta : \Gamma 
 \rightarrow 
 \finne
 $
 (i.e., forgetting the chosen basepoint)  induces an exact functor $\theta^* : \fcatk[\finne] \rightarrow \fcatk[\Gamma]$ satisfying $\theta^* \kbar \cong t^*$ and, more generally, $\theta^* (\kbar^{\otimes a})  \cong (t^*)^{\otimes a}$, for $a \in \nat$.
 
 Hence the first part of Theorem \ref{THM:main}  can be considered as the unpointed version of the isomorphism (\ref{DK-Pira-Intro}). However,  Corollary \ref{Cor-Intro} exhibits a major difference between  $(t^*)^{\otimes a}$ and  the $\finne$-modules $\kbar^{\otimes a}$, for $a \in \nat$: the former are projective whereas the latter are not.
 
We make the link with Pirashvili's result more explicit in Corollary \ref{cor:Rba}, where we identify $\hom_{\fcatk[\finne]} (\kbar^{\otimes a},\kbar^{\otimes b}) $ as a submodule of $\kring \hom_\Oz (\mathbf{b}, \mathbf{a})$, more precisely as the kernel of an explicit map given by further structure on $\kring \hom_{\Oz} (-, \mathbf{a})$. Theorem \ref{THM:main} determines the rank of this kernel.

The results are refined by taking into account the natural action of the group $\sym_a \times \sym_b\op$  on $\hom_{\fcatk[\finne]}(\kbar^{\otimes a}, \kbar^{\otimes b} )$. For instance:

\begin{THM}
(Proposition \ref{prop:hom_equivariance})
For $a<b \in \nat^*$ and $\kring = \rat$ there is an isomorphism of (virtual) $\sym_a \times \sym_b\op$-representations
\begin{eqnarray*}
\hom_{\fcatk[\finne]}(\kbar^{\otimes a}, \kbar^{\otimes b} )
&\cong &
(-1)^{b-a+1} [\sgnrep_{\sym_a} \otimes \sgnrep_{\sym_b}] 
\ + 
\\
&&
\sum_{t=0}^{b-a}(-1)^t
\big[\big(
\kring \hom_\Oz (\mathbf{b - t}, \mathbf{a}) 
\otimes 
\sgnrep_{\sym_t} 
\big)\uparrow _{\sym_{b-t}\op \times \sym_t \op}^{\sym_b\op}
\big].
\end{eqnarray*}
\end{THM}

As indicated above, this project was motivated by our work on higher Hochschild homology of a wedge of circles  \cite{2018arXiv180207574P}. For coefficients taken  in  $\fcatk[\Gamma]$, Hochschild homology gives a close relationship with $\fcatk[\gr]$, functors on the category $\gr$ of finitely-generated free groups; in particular, it leads  to interesting representations of the automorphism groups of free groups. For coefficients taken  in $\fcatk[\finne]$, this Hochschild homology is related to representations of \textit{outer} automorphism groups of free groups \cite{MR3982870} and corresponds, in the functorial context, to \textit{outer functors} \cite{ 2018arXiv180207574P}, which are functors on which inner automorphisms of free groups act trivially.

The study of  the category of outer functors $\fout$ and its r\^ole in the study of the  non-pointed version of higher Hochschild homology is of significant interest, as witnessed by the relationship with hairy graph homology \cite{MR3982870}. To understand higher Hochschild homology in the unpointed context, one of our goals is to understand the functor cohomology groups $\ext^i _{\fout}(\ak^{\otimes a},\ak^{\otimes b})$, where $\ak$ is induced by abelianization. There are  natural morphisms
\[
\ext^i _{\fout}(\ak^{\otimes a},\ak^{\otimes b})
\rightarrow 
\ext^i_{\fcatk[\gr]}(\ak^{\otimes a}, \ak^{\otimes b})
\]
 where the codomain is given by the main result of  \cite{V_ext}:
 \[
\ext^i_{\fcatk[\gr]}(\ak^{\otimes a}, \ak^{\otimes b})
\cong 
\left\{ 
\begin{array}{ll}
\kring \hom_\Oz (\mathbf{b}, \mathbf{a}) & i = b-a \\
0 & \mbox{otherwise,}
\end{array}
\right.
\]
which should be viewed  as a cohomological avatar of the relationship between $\Gamma$-modules and functors on the category $\gr$ explained before.
 
We expect that the image of the above morphism, for $i=b-a$, should be  closely related to $\hom _{\fcatk[\finne]} (\kbar^{\otimes a}, \kbar^{\otimes b})$. As evidence for this, in \cite{PV_2019} we prove that, if $\kring $ is a field,
 $$
 \ext^1 _{\fout}(\ak^{\otimes a},\ak^{\otimes a+1}) \cong  \hom_{\fcatk[\finne]}(\kbar^{\otimes a}, \kbar^{\otimes a+1} ).
 $$  

This result explains why we focus here on the computation of $\hom$ and $\ext^1$; the case of higher $\ext$ groups, which involves further material, will be presented elsewhere.

\bigskip
{\bf Strategy of  proof and organization of the paper.} We denote by $\finj$ the category of finite sets and injections and by  $\fb$ the category of finite sets and bijections.
\begin{enumerate}
\item 
We begin by giving a comonadic description of morphisms in $\fcatk[\finne]$ using the Barr-Beck theorem. This  identifies $ \hom_{\fcatk[\finne]}(\kbar^{\otimes a}, \kbar^{\otimes b} ) $ as the equalizer of a pair of maps,  or equivalently, as the kernel of an explicit map (see Remark \ref{Rem-noyau}). 

Unfortunately, this is inaccessible to direct study. To get around this difficulty, we introduce, in Section \ref{sect:new},  a cosimplicial object, extending the previous equalizer diagram. This arises via the general cosimplicial object that is constructed from a comonad in Appendix \ref{sect:homextseq}.

The $0$th and $1$st cohomology groups of this cosimplicial object correspond respectively to $\hom_{\fcatk[\finne]}(\kbar^{\otimes a}, \kbar^{\otimes b} ) $ and $\ext^1 _{\fcatk[\finne]}(\kbar^{\otimes a}, \kbar^{\otimes b} )$. The rest of the paper is largely devoted to computing the cohomology of this cosimplicial object.
\item 
The main result of Section \ref{sect:new} is an isomorphism between the above cosimplicial object and another one, that is defined in terms of the category $\Oz$. The latter is given by  a cobar-type cosimplicial construction (see Section \ref{subsect:cobar}). 

More precisely, in Section \ref{sect:koszul}, for $F$ an $\finj \op$-module  we introduce the cosimplicial $\fb\op$-module $\cosimp^\bullet F$. For $F=\kring \hom_\Oz (-, \mathbf{a})$  (which is an  $\finj \op$-module by Proposition \ref{prop:hom_oz_functor_finjop}) we obtain the cosimplicial object $\cosimp^\bullet \kring \hom_\Oz (-, \mathbf{a})$, which is the second cosimplicial object above. 

Section \ref{sect:new} provides the technical underpinnings of the paper.
\item 
In Section \ref{sect:koszul}, for $F$ an $\finj \op$-module we introduce the Koszul complex $\kz F$ 
in $\fb\op$-modules and we prove that this complex is quasi-isomorphic to the normalized cochain complex associated to the opposite of the cosimplicial object $\cosimp^\bullet F$. 
\item 
The heart of the paper is Section \ref{sect:cohomology}, in which we compute the cohomology of the Koszul complex  of $\kring \hom_\Oz (-, \mathbf{a})$, for $a \in \nat^*$. The key result is that its homology is concentrated in its top and bottom degrees (see Theorem \ref{thm:cohomology_NC}), so that the explicit calculation follows by considering the Euler-Poincar\'e characteristic of the complex. 
\item 
These results are combined in Section \ref{sect:homext} to prove Theorem \ref{THM:main}.
\end{enumerate}

Section \ref{sect:koszul} relates our calculations to  $\finj\op$-cohomology, which is introduced in 
Section \ref{FIop-cohom}. This is the dual notion of the $\finj$-homology  introduced in \cite{MR3654111}. The Koszul complex $\kz F$ is the $\finj\op$-version of the Koszul complex considered in 
\cite[Theorem 2]{MR3603074} for $\finj$-modules and Theorem \ref{thm:norm_vs_kz} is the analogue of \cite[Theorem 2]{MR3603074}. Thus Theorem \ref{thm:cohomology_NC}  can be interpreted as the computation of the $\finj\op$-cohomology of the $\finj\op$-module  $\kring \hom_\Oz (-, \mathbf{a})$ as follows:

\begin{THM}\label{FI-op-cohom}
For $a, b, n \in \nat$, 
$$H^n_{\finj\op} (\kring \hom_\Oz (-, \mathbf{a}))(\mathbf{b})=\left\{ 
\begin{array}{ll}
 \kring^{\oplus \epfn (b,a)}  & \text{\ if\  }  b\geq a>0  \text{\ and\ } n=0 \\
 \kring & \text{\ if\  } (b> a\geq 0  \text{\ and\ } n=b-a) \text{\ or\ } (a=b=n=0) \\
0 & \mbox{otherwise.}
\end{array}
\right.
$$
\end{THM}

\tableofcontents

\subsection{Notation}
\begin{itemize}
\item $\kring$ denotes a commutative ring with unit.
\item
$\nat$ denotes the non-negative integers and $\nat^*$ the positive integers.
\item 
\label{nota:sgnrep}
For $n \in \nat$, $\sym_n$ denotes the symmetric group on $n$ letters and $\sgnrep_{\sym_n}$  the sign representation, with underlying module $\kring$ and $\sigma \in \sym_n$ acting via the signature $\sgn (\sigma)$. 
\item 
For $n \in \nat$, $\mathbf{n}$ denotes the set $\{1, 2, \ldots, n\}$ (by convention $\mathbf{0}=\emptyset$); we draw the reader's attention to the fact that $\mathbf{n + m}$ denotes the set $\{1, \ldots, n+m\}$ (respectively $\mathbf{n - m}$ denotes $\{1, \ldots, n-m\}$ if $n>m$).
\item 
The notation $F : \calc \rightleftarrows \cald : G$ for an adjunction always indicates that $F$ is the left adjoint to $G$ (i.e., $F \dashv G$ in category-theoretic notation).
\end{itemize}

\begin{rem}
The set $\mathbf{b}$ is equipped with the canonical  total order. Arbitrary subsets of $\mathbf{b}$ occur in the text, denoted variously by $\mathbf{b}'$, $\mathbf{b}''$ , $\mathbf{b}^{(t)}_i$\ldots ; these inherit a total order from $\mathbf{b}$.  (We never consider the finite set associated to a natural number $b'$, for example, so this notation should not lead to confusion.)
\end{rem}

\begin{nota}
\label{nota:ordinals}
Let $\aord$ denote the skeleton of the category of finite ordinals with objects $[n]$ indexed by integers $n \geq -1$, so that the object indexed by $-1$ corresponds to $\emptyset$ and the usual category of non-empty finite ordinals $\ord$ is the full subcategory on objects indexed by $\nat$.  

Denote by $\delta_i$ and $\sigma_j$ the face and degeneracy morphisms of $\aord\op$ and $d_i$, $s_j$ the coface and codegeneracy morphisms of $\aord$.
\end{nota}

Recall that there is an  involution of $\aord$ given by reversing the order on finite ordinals, which restricts to an involution of $\ord$ (cf. \cite[8.2.10]{MR1269324}).

\begin{nota}
\label{nota:opp_simp}
For $\calc$ a category, let 
\begin{eqnarray*}
\op &:& \ord\op \calc \rightarrow \ord \op\calc 
\\
\op &:& \ord \calc \rightarrow \ord \calc
\end{eqnarray*}
denote the  {\em opposite} structure functors on simplicial (respectively cosimplicial) objects, induced by the composition with the involution of $\ord$. 

Explicitly, if $A_\bullet$ is a simplicial object in $\calc$ with face operators $\delta_i: A_n \to A_{n-1}$ and degeneracy operators $\sigma_i: A_n \to A_{n+1}$ for $i \in \{0, 1, \ldots, n\}$, $(A_\bullet)^{op}$ is the simplicial object in $\calc$ with face operators $\tilde{\delta_i}: A_n \to A_{n-1}$ and degeneracy operators $\tilde{\sigma_i}: A_n \to A_{n+1}$, where  $\tilde{\delta_i}=\delta_{n-i}$ and $\tilde{\sigma_i}=\sigma_{n-i}$.

This notation is also  applied for augmented objects.
\end{nota}

\begin{nota}
\label{nota:shift}
For $C_*$ a chain complex (respectively $C^*$ a cochain complex), we use the convention for shifting given in \cite[1.2.8]{MR1269324}, i.e., 
$$C[p]_n=C_{n+p} \qquad (\mbox{respectively } \ C[p]^n=C^{n-p}).$$
\end{nota}

\section{Categories of sets}
\label{sect:background}

Let $\sets$ denote the category of sets and $\sets_*$ that of pointed sets. Consider the commutative diagram of categories and forgetful functors 
\[
\xymatrix{
 & \fb\ar@{^(->}[dl]\ar@{^(->}[dr]
 \\
\Oz
\ar@{^(->}[dr]
&&
\finj 
\ar@{^(->}[dl]
\\
\finne \ar@{^(->}[r]
&
\fin
\ar@{^(->}[r]
&
\sets
\\
\Gamma 
\ar[u]^\theta 
\ar@{^(->}[rr]
&
&
\sets_*
\ar[u]_{\mathrm{forget}}
&
}
\]
where
\begin{enumerate}
\item  $\fin \subset \sets$ is the full subcategory of  finite  sets;
\item 
 $\finne \subset \fin$ the full subcategory of non-empty finite sets;
 \item  
  $\Gamma \subset \sets_*$  the full subcategory of finite pointed 
sets;
\item 
$\Oz$ the category of  finite sets and surjections;
\item
 $\finj$ the category of finite sets and injections;
 \item 
  $\fb \subset \fin$ the subcategory of finite sets and bijections;
  \item 
   the functor $\theta$ forgets the marking of the basepoint.
  \end{enumerate}

\begin{rem}
The functor $\theta : \Gamma \rightarrow \finne$ is essentially surjective and faithful but is not full.
\end{rem}

\begin{rem}
\label{rem:Oz}
By definition,  
\begin{enumerate}
\item 
$\hom_\Oz (\mathbf{0}, \mathbf{0} ) = \{ * \}$, 
$\hom _\Oz (\mathbf{a}, \mathbf{0} ) = \emptyset =  \hom _\Oz (\mathbf{0}, \mathbf{a} )$ for $a>0$ and
$$\hom_{\Oz} (\mathbf{b}, \mathbf{1}) = \left\{ \begin{array}{ll}
\{ * \} & b>0 \\
\emptyset & b=0,
\end{array}
\right.
$$ where, for $b>0$, the morphism is given by the unique surjective map $\mathbf{b} \twoheadrightarrow \mathbf{1}$;
\item 
the symmetric monoidal structure on $\fin$ given by disjoint union of finite sets induces a symmetric monoidal structure on $\Oz$.
\end{enumerate}
\end{rem}

\begin{rem}
In \cite{Pdk}, the category of \textit{non-empty} finite sets and surjections is denoted by $\Omega$ whereas, in \cite{Phh}, $\Omega$ denotes the category of finite sets and surjections. Here we adopt the second convention, denoting this category $\Oz$.
\end{rem}

 \begin{lem}
 \label{lem:theta_finne_adjunction}
There is an adjunction $ (-)_+ : \fin \rightleftarrows \Gamma : \mathrm{forget} $, that restricts to an adjunction 
$
 (-)_+ : \finne \rightleftarrows \Gamma : \theta$, 
 where  the functor $(-)_+$ sends  a finite set $X$ to $X_+$, with $+$ as basepoint.
 
The adjunction unit $\mathrm{Id}_{\finne} \rightarrow \theta \big((-)_+\big)$   for  a finite set $X \in \ob \finne$ is the natural inclusion 
$X \hookrightarrow X_+.$

 For a finite pointed set $(Z,z)\in \ob \Gamma$, the adjunction counit $(\theta)_+ \rightarrow \mathrm{id}_\Gamma$ is the natural morphism 
$
(\theta Z)_+ 
\rightarrow 
Z 
$ in $\Gamma$ that extends the identity on $Z$ by sending $+$ to $z$. 
 \end{lem}

 \begin{rem}
The categories $\fin$ and $\Gamma$ have small 
skeleta with objects $\n \in \ob \fin$ and $\n_+ \in \ob 
\Gamma$, for $n \in \nat$.
 \end{rem}

The following underlines the fact that $\hom_{\Oz} (\mathbf{b}, \mathbf{a})$ can be understood in terms of the fibres of maps:

\begin{lem}
\label{lem:surjections_fibres}
For $a>0$ and any non-empty subset $\mathbf{b}' \subset \mathbf{b}$, 
there is a bijection between $\hom_{\Oz} (\mathbf{b}', \mathbf{a})$ and the set of ordered partitions of $\mathbf{b}'$ into $|\mathbf{a}|=a$ disjoint, non-empty subsets, via 
\[
( f : \mathbf{b}'\twoheadrightarrow \mathbf{a}) 
\mapsto 
(f^{-1} (1), f^{-1} (2), \ldots , f^{-1} (a) ). 
\]

In particular, the ordered sets:
\begin{eqnarray}
\label{eqn:order_bprime}
f^{-1} (1), f^{-1} (2), \ldots , f^{-1} (a)
\\
\label{eqn:order_b}
f^{-1} (1), f^{-1} (2), \ldots , f^{-1} (a) , \mathbf {b}\backslash \mathbf{b}'
\end{eqnarray} 
define permutations of $\mathbf{b}'$ and $\mathbf{b}$ respectively, 
where each fibre $f^{-1}(i)$, $i \in \mathbf{a}$, and $\mathbf{b}\backslash \mathbf{b}'$  inherit their order  from $\mathbf{b}$.
\end{lem}

\section{Functor categories on categories of sets}
\label{sect:gamma_finne_modules}
 
In this section, we first review functors on $\Gamma$, before considering the corresponding structures for $\finne$. Then we compare the functor categories on $\Gamma$ and on $\finne$, obtaining the comonad 
$\gcmnd : \fcatk[\Gamma] \rightarrow \fcatk[\Gamma]$ in Notation \ref{nota:gcmnd}. Finally, we recall the relationship between (contravariant) functors on $\finj$ and on $\fb$, obtaining the comonad $\fbcmnd : \fcatk[\fb\op]\rightarrow \fcatk[\fb\op]$ introduced in Notation \ref{nota:fbcmnd}. The comonads $\gcmnd$ and $\fbcmnd$ play a crucial role in the paper.

 We begin  with some recollections concerning functor categories in order to fix notation. 
 
For $\calc$ an (essentially) small category, $\fcatk[\calc]$ denotes the category of functors from $\calc$ to $\kring$-modules. (This category  may also be  referred to as the category of $\calc$-modules when the codomain is clear from the context.) This is a Grothendieck abelian category with structure inherited from $\kring$-modules; moreover the tensor product over $\kring$  induces a symmetric monoidal structure $\otimes$ on $\fcatk[\calc]$ with unit the constant functor $\kring$.

If $\psi : \calc \rightarrow \cald$ is a functor between small categories, the functor given by precomposition with $\psi$ is denoted $\psi ^*  : \fcatk[\cald] \rightarrow \fcatk [\calc]$. This is an exact functor, which is symmetric monoidal. 

\begin{nota} \label{proj}
For $X$ an object of a small category $\calc$, $P^\calc_X$ denotes the  projective of $\fcatk[\calc]$ given by $\kring \hom_\calc (X, -)$, the composite of the corepresentable functor $\hom_\calc (X, -)$ with the free $\kring$-module functor. 
\end{nota}

The projectives $P^\calc_X$, as $X$ ranges over a set of representatives of the isomorphism classes of objects of $\calc$, form a family of projective generators of $\fcatk[\calc]$. 

\subsection{The functor category $ \fcatk[\Gamma] $ }

The functor category $ \fcatk[\Gamma] $ has been studied by Pirashvili and his co-authors (see for example \cite{Phh,Pdk}). We begin this section by some basic facts concerning this category.

 Since $\Gamma$ is a pointed category, one has the splitting:

\begin{lem}
\label{lem:Gamma_mod_reduced}
There is an equivalence of categories 
$
 \fcatk[\Gamma] \cong \overline{\fcatk[\Gamma]} \times \kring \dash \modules,
$ 
where $\overline{\fcatk[\Gamma]}$ is the full subcategory of reduced objects (i.e., those vanishing on $\mathbf{0}_+$).
 \end{lem}

The projectives given in Notation \ref{proj} have the following property:

\begin{prop}
\label{prop:coprod_PGamma}
\ 
For $m,n \in \nat$, there is an isomorphism
$
P^\Gamma_{\mathbf{m+n}_+}
\cong 
P^\Gamma_{\mathbf{m}_+}
\otimes 
P^\Gamma_{\mathbf{n}_+}
$. 
In particular, for $n \in \nat$, one has $P^\Gamma_{\mathbf{n}_+}
\cong (P^\Gamma_{\mathbf{1}_+})^{\otimes n}$.
\end{prop}
\begin{proof}
Use that the  object $\mathbf{m+n}_+$ of $\Gamma$ is isomorphic to the coproduct $\mathbf{m}_+ \bigvee \mathbf{n}_+$ in $\Gamma$.
\end{proof} 

One of the most important results concerning the category $\fcatk[\Gamma]$ is Pirashvili's  Dold-Kan type theorem for $\Gamma$-modules (see Theorem \ref{thm:DK_equiv}).
To recall this, we introduce the following:

\begin{defn}
\label{defn:t*}
(Cf. \cite{Phh}.) 
Let $t^* \in \ob \fcatk[\Gamma]$ be the quotient of $P^\Gamma _{\mathbf{1}_+}$ given by $t^* (\mathbf{n}_+) := \kring [\mathbf{n}_+]/ \kring [\mathbf{0}_+]$.
\end{defn}

One has the following, in which the image of the generator $[x] \in \kring[\mathbf{n}_+]$ (corresponding to $x \in \mathbf{n}_+$) in $t^* (\mathbf{n}_+)$ is written $[[x]]$, which is zero if and only if $x=+$.

\begin{lem}
\cite{Phh}
\label{lem:t*_projective}
The functor $t^*$ is projective in $\fcatk[\Gamma]$ and there is a canonical splitting $P^\Gamma_{\mathbf{1}_+} \cong t^* \oplus \kring$.

Explicitly, for $(Z, z)$ a finite pointed set, the inclusion $t^* (Z) \cong \kring [Z]/ \kring [z] \hookrightarrow  P^\Gamma_{\mathbf{1}_+} (Z) \cong \kring [Z] $ is induced on generators by $[[y]] \mapsto [y] - [z]$, for $y \in Z \backslash \{z \}$, and the projection $P^\Gamma_{\mathbf{1}_+} (Z) \cong \kring [Z] \rightarrow \kring \cong \kring [\mathbf{0}_+]$ is induced by the morphism $Z \rightarrow \mathbf{0}_+$ to the terminal object of $\Gamma$. 
\end{lem}

\begin{proof}
The canonical splitting is provided by Lemma \ref{lem:Gamma_mod_reduced} and the remaining statements follow by explicit identification of the functors. 
\end{proof}

Lemma \ref{lem:t*_projective} together with Proposition \ref{prop:coprod_PGamma} lead directly to the following Proposition,  in which $(t^* )^{\otimes 0}=\kring$: 

\begin{prop}
\label{prop:proj_gen_Gamma}
 \cite{Phh}
 The set $\{(t^* )^{\otimes a}  | \ a \in \nat \}$ is a set of projective generators of $\fcatk[\Gamma]$. Restricting to $a \in \nat^*$ gives projective generators of $\overline{\fcatk[\Gamma]}$. 
\end{prop}

The category $\fcatk[\Gamma]$ is symmetric monoidal for the structure induced by $\otimes$ on $\kring$-modules and $\Oz$ has a symmetric monoidal structure (see Remark \ref{rem:Oz}).

In the following result, $\kring \Oz $ denotes the $\kring$-linearization of the category $\Oz$. 

\begin{thm}
\label{thm:DK_equiv}
\cite{Phh}
There is a $\kring$-linear, symmetric monoidal embedding 
\[
(\kring \Oz)\op
\rightarrow 
\fcatk[\Gamma]
\]
that is induced by $\mathbf{1} \mapsto t^*$. This is fully faithful and induces an equivalence of categories between ${\fcatk[\Gamma]}$ and ${\fcatk[\Oz]}$. 

 In particular,  for $a, b\in \nat$,  there is a natural isomorphism
\[
\hom_{\fcatk[\Gamma]}((t^*)^{\otimes a}, (t^*)^{\otimes b} ) 
\cong 
\kring \hom_\Oz (\mathbf{b}, \mathbf{a}). 
\]

\end{thm}


In Section \ref{sect:new} it is necessary to have an explicit description of the natural isomorphism given in the previous Theorem. So, the rest of this section is devoted to make explicit the isomorphism 
\[
\kring \hom_{\Oz} (\mathbf{b}, \mathbf{a}) 
\stackrel{\cong}{\rightarrow} 
\hom_{\fcatk[\Gamma]}((t^*)^{\otimes a}, (t^*)^{\otimes b} ).
\]

For $a=0$, we have:
\[
\hom_{\fcatk[\Gamma]}  (\kring , (t^*)^{\otimes b} )
\cong 
\left\{ 
\begin{array}{ll}
\kring & b=0\\
0 & b>0, 
\end{array}
\right.
\]
where, for $b=0$, a generator is given by the identity $\mathrm{Id}_\kring: \kring \to \kring$.

For $a>0$, understanding the morphisms $ \hom_{\fcatk[\Gamma]}  (t^* , (t^*)^{\otimes b} ) $ is a key ingredient.  By Lemma \ref{lem:t*_projective}, for $(Z,z)$ a finite pointed set, $ \{[y] - [z]\  |\ y \in Z \backslash \{z \} \}$ is a basis $t^*(Z)$. We use this basis in the following:

\begin{lem}
\label{lem:hom_Oz_a=1}
For $b \in \nat$,
\[
\hom_{\fcatk[\Gamma]}  (t^* , (t^*)^{\otimes b} )
\cong 
\left\{ 
\begin{array}{ll}
\kring & b>0\\
0 & b=0, 
\end{array}
\right.
\]
where, for $b>0$, a generator is given by the morphism $\xi_b: t^* \rightarrow (t^*)^{\otimes b}$ given upon evaluation on $(Z,z)$ by 
$ [y]-[z] \mapsto ([y]-[z])^{\otimes b}$, 
 where $y \in Z \backslash \{z \}$.  
 
 In particular, the action of $\sym_b$ on $\hom_{\fcatk[\Gamma]}  (t^* , (t^*)^{\otimes b} )$ induced by the place permutation action on $(t^*)^{\otimes b}$ is trivial. 
\end{lem}

\begin{proof}
If $b=0$, since $t^*$ is a reduced $\Gamma$-module, the statement follows from Lemma \ref{lem:Gamma_mod_reduced}. 
 For $b>0$,  
 the inclusion $t^* \subset P^{\Gamma}_{\mathbf{1}_+} \cong t^* \oplus \kring$ (where the isomorphism is given by Lemma \ref{lem:t*_projective}) induces an isomorphism 
$$
\hom_{\fcatk[\Gamma]} ( P^{\Gamma}_{\mathbf{1}_+}, (t^*)^{\otimes b} )  \cong 
\hom_{\fcatk[\Gamma]}  (t^* , (t^*)^{\otimes b} ) $$
 by Lemma \ref{lem:Gamma_mod_reduced}.  Yoneda's Lemma gives that the left hand side is $\kring$, with trivial $\sym_b$-action, where the given generator corresponds to the element $([1]-[+])^{\otimes b}  \in t^*(\mathbf{1}_+)^{\otimes b}$.
\end{proof}

The symmetric monoidal structure of $\Oz$ (see Remark \ref{rem:Oz}) gives the following:

\begin{lem}
\label{lem:mor_Oz}
For $a \in \nat^*$ and $ b  \in \nat$ the symmetric monoidal structure of $\Oz$ induces a $\sym_b\op$-equivariant isomorphism 
\[
\coprod_{\substack { (b_i)\in (\nat^*)^{a} \\ \sum_{i=1}^a b_i = b }}
\big( \prod_{i=1}^a  \hom_{\Oz} (\mathbf{b_i}, \mathbf{1})\big)
 \uparrow_{\prod_{i=1}^a \sym_{b_i}}^{\sym_b} 
\stackrel{\cong}{\rightarrow }
\hom_{\Oz} (\mathbf{b}, \mathbf{a} ),
\]
in which $\hom_{\Oz} (\mathbf{b_i}, \mathbf{1}) \cong\{ * \}$.
\end{lem}

\begin{proof}
For $b=0$, the two sides of the bijection are the empty set.
For $b \not = 0$, the component of the morphism indexed by $(b_i) \in (\nat^*)^a$ is induced by the set map
\begin{eqnarray*}
\prod_{i=1}^a 
\hom_{\Oz} (\mathbf{b_i}, \mathbf{1})
&\rightarrow& 
\hom_{\Oz} (\mathbf{b}, \mathbf{a})
\\
(\xi_i) &
\mapsto 
& 
\mathbf{b}\cong \amalg_{i=1}^a \mathbf{b_i} 
\stackrel{\amalg_{i=1} ^a \xi_i}{\longrightarrow} 
\amalg_{i=1}^a \mathbf{1} 
\cong
\mathbf{a}.
\end{eqnarray*}

This induces the given isomorphism, since a surjection $f : \mathbf{b} \twoheadrightarrow \mathbf{a}$ is uniquely determined by the partition of $\mathbf{b}$ given by the fibres $\mathbf{b}_i:= f^{-1}(i)$ (cf. Lemma \ref{lem:surjections_fibres}).
\end{proof}

This allows the following definition to be given. 

\begin{defn}
\label{defn:dk_map}
Let $\dk_{\mathbf{a}, \mathbf{b}}: \hom_\Oz (\mathbf{b}, \mathbf{a}) \rightarrow \hom_{\fcatk[\Gamma]} ((t^*)^{\otimes a} , (t^*)^{\otimes b})$, for $a, b \in \nat$, be the $\sym_a \times \sym_b\op$-equivariant map determined by the following:
\begin{enumerate}
\item
$\dk_{\mathbf{0}, \mathbf{0}}(\id_{\mathbf{0}}) := \id_{{\kring}}$; 
\item 
for $a>0$, $b>0$ and a  sequence $(b_i) \in (\nat^*)^{a}$ such that $\sum_{i=1}^a b_i = b$,   the restriction of $\dk_{\mathbf{a}, \mathbf{b}}$ to $\prod_{i=1}^a  \hom_{\Oz} (\mathbf{b_i}, \mathbf{1}) \subset \hom_{\Oz} (\mathbf{b}, \mathbf{a})$ (using the inclusion given by Lemma \ref{lem:mor_Oz}), 
\[
\prod_{i=1}^a  \hom_{\Oz} (\mathbf{b_i}, \mathbf{1})
\rightarrow 
\hom _{\fcatk[\Gamma]}((t^*)^{\otimes a}, (t^*)^{\otimes b} )
\]
sends the one point set on the left to the morphism $\bigotimes_{i=1}^a (t^* \rightarrow (t^*)^{\otimes b_i})$, the tensor product of the generators  $\xi_{b_i}$ given by Lemma \ref{lem:hom_Oz_a=1}.
\end{enumerate}
\end{defn}

\begin{prop}
\label{prop:explicit_DK_morphism}
For $a, b\in \nat$, the $\kring$-linear extension  of $\dk_{\mathbf{a}, \mathbf{b}}$ induces a $\sym_a \times \sym_b\op$-equivariant isomorphism
\[
\kring \dk_{\mathbf{a}, \mathbf{b}} : 
\kring \hom_{\Oz} (\mathbf{b}, \mathbf{a}) 
\stackrel{\cong}{\rightarrow} 
\hom_{\fcatk[\Gamma]}((t^*)^{\otimes a}, (t^*)^{\otimes b} ).
\]
\end{prop}

\begin{proof}
The construction of the map $\dk_{\mathbf{a}, \mathbf{b}}$ ensures that the $\kring$-linearization is $\sym_a \times \sym_b\op$-equivariant. The fact that it is an isomorphism follows from Pirashvili's result, Theorem \ref{thm:DK_equiv}.
\end{proof}

\begin{rem}
We will see in Proposition \ref{prop:iso-FB-mod} that the maps $\dk_{\mathbf{a}, \mathbf{b}}$ define an isomorphism of $\fb \times \fb\op$-modules.
\end{rem}

\subsection{The functor category $ \fcatk [\finne]$ }
\label{section:F(Fin)}
 
 The aim of this section is to give an analysis of the category $ \fcatk [\finne]$  parallel to that of $\fcatk[\Gamma]$ given in the previous section.

Whilst $\emptyset$ is the initial object of $\fin$, it is not final, so that $\fin$ is not pointed and there is no splitting analogous to that of Lemma \ref{lem:Gamma_mod_reduced}. Instead one has:

\begin{prop}
\label{prop:restrict_finne}
The restriction functor $\fcatk[\fin] \rightarrow \fcatk [\finne]$ induced by 
$\finne \subset \fin$ is exact and symmetric monoidal and has exact left adjoint 
 given by extension by zero (i.e., for $F \in \ob \fcatk [\finne]$, defining $F 
(\emptyset) : = 0$). 

Via these functors, $\fcatk[\finne]$ is equivalent to the full subcategory of 
 reduced objects of $\fcatk[\fin]$ (functors $G$ such that $G (\emptyset) 
=0$). 
\end{prop}

\begin{nota}
Let $\overline{\kring}$ denote the constant functor of $\fcatk[\finne]$ taking value $\kring$. (This notation is introduced so as to avoid potential confusion with the constant functor of $\fcatk[\fin]$, which is denoted by $\kring$.)
\end{nota}

\begin{prop}
\label{prop:overline_k_injective}
The following conditions are equivalent:
\begin{enumerate}
\item 
$\kring$ is injective as a $\kring$-module; 
\item 
the constant functor $\overline{\kring}$ is injective in $\fcatk[\finne]$;
\item 
the constant functor $\kring$ is injective in $\fcatk[\fin]$. 
\end{enumerate}
\end{prop}

\begin{proof}
Since $\mathbf{1}$  is the terminal object of both  $\fin$ and $\finne$, in both cases, the constant functor is isomorphic to the maps from $\hom (-, \mathbf{1})$ to $\kring$, $\kring^{\hom (-, \mathbf{1})} $, where $\hom$ is calculated respectively in $\finne$ and $\fin$. The equivalence of the statements then follows by using Yoneda's Lemma.
\end{proof}
 
 Proposition \ref{prop:coprod_PGamma} has the following analogue for $ \fcatk [\finne]$; the proof is similar.
 
 \begin{prop}
\label{prop:coprod_Pfinne}
For $m,n \in \nat^*$, there is an isomorphism
$
P^\finne_{\mathbf{m+n}}
\cong 
P^\finne_{\mathbf{m}}
\otimes 
P^\finne_{\mathbf{n}}
$.  
In particular, for $n \in \nat^*$, one has $P^\finne_{\mathbf{n}}
\cong (P^\finne_{\mathbf{1}})^{\otimes n}$.
\end{prop}

We introduce the functor $\kbar \in \ob \fcatk[\finne]$, which is the analogue of the functor $t^* \in \ob \fcatk[\Gamma]$ introduced in Definition \ref{defn:t*}. 

\begin{defn}
\label{defn:kbar}
Let $\kbar$ be the object of $\fcatk [\finne]$ defined as the kernel of the surjection $P^\finne _{\mathbf{1}} \twoheadrightarrow \overline{\kring}$ in $\fcatk[\finne]$ which is induced by the map to the terminal object of $\finne$.
 (Explicitly, $P^\finne _{\mathbf{1}}(\mathbf{n}) \cong \kring [\mathbf{n}]$ and the canonical map $\mathbf{n} \rightarrow \mathbf{1}$ induces $\kring [\mathbf{n}] \twoheadrightarrow \kring [\mathbf{1}] = \kring$.)
\end{defn}

This gives the defining short exact sequence:
\begin{eqnarray}
\label{eqn:def_kbar}
0
\rightarrow 
\kbar
\rightarrow 
P^\finne_{\mathbf{1}}
\rightarrow 
\overline{\kring}
\rightarrow 0.
\end{eqnarray}

The functor $\kbar$ plays an essential r\^ole in analysing the category of $\finne$-modules. The principal difference between $t^* \in \ob \fcatk[\Gamma]$ and $\kbar \in \ob \fcatk [\finne]$ is that the functor $\kbar$ is not projective (see  Corollary \ref{cor:non-proj_inj}). We record immediately the following important observation:

\begin{prop}
\label{prop:ses_kbar_not_split}
The short exact sequence (\ref{eqn:def_kbar}) does not split, hence 
\begin{enumerate}
\item 
$\overline{\kring}$ is not projective in $\fcatk[\finne]$; 
\item 
the short exact sequence represents a non-trivial class in $\ext^1 _{\fcatk[\finne]} (\overline{\kring}, \kbar)$. 
\end{enumerate}
\end{prop}

\begin{proof}
The morphism $P^\finne_{\mathbf{1}} \twoheadrightarrow \overline{\kring}$ is an isomorphism when evaluated on $\mathbf{1}$ hence, were a section to exist, it would send  $1 \in \overline{\kring}$ to the generator $[\mathbf{1}]$. However, this morphism cannot extend to a natural transformation from the constant functor $\overline{\kring}$ to  $P^\finne_{\mathbf{1}}$, since $[\mathbf{1}]$ is  not invariant under morphisms of $\finne$. The remaining statements follow immediately.
\end{proof}

\subsection{Comparison of the functor categories $ \fcatk[\Gamma] $ and $ \fcatk [\finne]$}
\label{section:Comparison-FGamma-FFin}
In this section, we relate the categories $ \fcatk[\Gamma] $ and $ \fcatk [\finne]$ using the functors $\theta$ and $(-)_+$ introduced in Section \ref{sect:background}.

Recall that a functor $\psi$ is conservative if a morphism $f$ is an isomorphism if and only if $\psi (f)$ is.

\begin{prop}
\label{prop:adjoints_theta}
 The adjunction $(-)_+ : \finne \rightleftarrows \Gamma : \theta$ 
induces an adjunction
 \[
  \theta^* :  \fcatk[\finne ] \rightleftarrows \fcatk[\Gamma] : (-)_+^* 
 \]
of exact functors that are symmetric monoidal. 
Moreover:
\begin{enumerate}
\item 
$\theta^* : \fcatk[\finne]\rightarrow \fcatk[\Gamma]$ is faithful and conservative; 
\item 
$\theta^*$ sends projectives to projectives. (Explicitly, for $n \in \nat^*$, $\theta^* P^\finne_{\mathbf{n}}  \cong P^\Gamma_{\mathbf{n}_+}$.)
\end{enumerate}
\end{prop}

\begin{proof} The first statement is general: 
precomposition with a functor induces an exact, symmetric monoidal functor;  precomposition with an adjunction yields an adjunction, with the r\^ole of the adjoints reversed. 

The fact that $\theta^*$ is faithful and conservative follows from the fact that $\theta$ is essentially surjective and faithful. (This is the reason for restricting attention to $\finne$.)

Finally, it follows formally from the adjunction that $\theta^*$ sends projectives to projectives. The explicit identification follows from the adjunction isomorphism $\hom_{\finne} (\mathbf{n}, \theta(X)) \cong \hom_{\Gamma} (\mathbf{n}_+, X)$, for $X$ a finite pointed set.
\end{proof} 

\begin{prop}
\label{prop:tstar_kbar}
\ 
\begin{enumerate}
\item 
Applying $\theta^*$ to the short exact sequence (\ref{eqn:def_kbar}) yields the canonical (split) short exact sequence  in $\fcatk[\Gamma]$:
\[
0
\rightarrow 
t^* 
\rightarrow 
 P^\Gamma_{\mathbf{1}_+}
 \rightarrow 
 \kring 
 \rightarrow 
 0;
 \]
\item
 for $a \in \nat$, there is an isomorphism $\theta^* (\kbar^{\otimes a}) \cong (t^* )^{\otimes a}$.   
 \end{enumerate}
\end{prop}

\begin{proof}
It is clear that, with respect to the identification of the action of $\theta^*$ on projectives given in Proposition \ref{prop:adjoints_theta}, $\theta^*(P^\finne _{\mathbf{1}} \twoheadrightarrow \overline{\kring})$ is the projection $P^\Gamma _{\mathbf{1}_+} \twoheadrightarrow \kring$. Since $\theta^*$ is exact, this gives the first statement,  together with the isomorphism $\theta^* (\kbar) \cong t^*$.  

The second statement then follows from the fact that $\theta^*$ is symmetric monoidal.
\end{proof}

In the rest of the section we  establish the comonadic description of $\fcatk[\finne]$ using the Barr-Beck theorem (see Theorem \ref{thm:barr_beck}).

\begin{nota}
\label{nota:gcmnd}
Let $\gcmnd : \fcatk[\Gamma] \rightarrow \fcatk[\Gamma]$ be the comonad given by  $\gcmnd =\theta^* \circ (-)_+^* $.   
\end{nota}

One has the following  immediate consequence of Proposition \ref{prop:adjoints_theta}: 

\begin{cor}
\label{cor:perp_sym}
The functor $\gcmnd : \fcatk[\Gamma]\rightarrow \fcatk[\Gamma]$ is exact and symmetric monoidal. 
\end{cor}

The structure maps of the comonad $\gcmnd$ are described in the following Proposition.

\begin{prop}
\label{prop:comonad-gamma-explicit}
Let $F \in \ob \fcatk[\Gamma]$ and $(Y,y)$ be a finite pointed set.
\begin{enumerate}
\item 
The counit $\epsilon^\Gamma_{F}: \gcmnd F(Y)=F(Y_+) \rightarrow  F(Y)$ is given by $F(\alpha)$ where $\alpha$ extends the identity on $Y$ by sending $+$ to $y$.
\item  
Iterating $\gcmnd$, labelling the functors $\gcmnd$ and their associated basepoint, one has 
\[
\gcmnd_1 \gcmnd_2 F (Y) 
= 
F ((Y_{+_1}) _{+_2}).
\]
\item 
The diagonal $\Delta: \gcmnd F (Y) \rightarrow \gcmnd_1 \gcmnd_2 F (Y)$ is induced by the  pointed map $Y_+ \rightarrow (Y_{+_1}) _{+_2}$ given by the identity on $Y$ and sending $+$ to $+_2$. 
\item 
The counit $\gcmnd_1 \epsilon^\Gamma_{F} : \gcmnd_1 \gcmnd_2 F (Y) \rightarrow \gcmnd F(Y)$ is given by sending $+_2$ to $+$ and the counit $\epsilon^{\Gamma}_{\gcmnd_2 F} : \gcmnd_1 \gcmnd_2 F (Y)\rightarrow \gcmnd F(Y)$  by sending $+_1$ to $y$.  
\end{enumerate}
\end{prop}

\begin{proof}
By Proposition \ref{prop:adjoints_theta}, the adjunction giving rise to $\gcmnd$ is induced by the adjunction  $(-)_+ : \finne \rightleftarrows \Gamma : \theta$, hence the structure morphisms of the comonad $\gcmnd$ can be deduced from Lemma \ref{lem:theta_finne_adjunction}.

Note that, by definition of $\gcmnd$, we have:
$$\gcmnd_1 \gcmnd_2 F (Y) =F \circ (-)_{+_2} \circ \theta \circ   (-)_{+_1} \circ   \theta(Y)=F ((Y_{+_1}) _{+_2}).$$
\end{proof}

The Barr-Beck theorem (see, for example, \cite[Theorem 4.3.8]{KS} for the dual monadic statement) then gives the 
following:

\begin{thm}
\label{thm:barr_beck}
The category $\fcatk[\finne]$ is equivalent to the category $\fcatk[\Gamma]_\gcmnd$ of 
$\gcmnd$-comodules in $\fcatk[\Gamma]$. 
\end{thm}

\begin{rem}
\label{rem:barr_beck}
The Barr-Beck correspondence can be understood explicitly as sketched below. 

Suppose that $G \in \ob \fcatk[\finne]$. Let $\psi_{\theta^*G}:=\theta^*(\eta_G): \theta^*G \to \gcmnd \theta^*G$ where $\eta: Id \to (-)_+^* \circ \theta^*$ is the unit of the adjunction given in Proposition \ref{prop:adjoints_theta}. Explicitly, for $X \in \finne$, $(\eta_G)_X=G(i_X): G(X) \to G(X \amalg \{ + \})$ where $i_X: X \to X \amalg \{ + \}$ is the canonical inclusion.
Then $(\theta^*G, \psi_{\theta^*G})$ is a $\gcmnd$-comodule and the correspondence $G \mapsto (\theta^*G, \psi_{\theta^*G})$ defines a functor $\fcatk[\finne] \to \fcatk[\Gamma]_\gcmnd$.

Suppose that $F \in \ob \fcatk[\Gamma]_\gcmnd$; the associated functor on $\finne$ is given for $n \in \nat^*$ by 
$ 
\mathbf{n} 
\mapsto 
F ((\mathbf{n-1})_+)
$. 
This corresponds to choosing a basepoint  for each object of the skeleton of $\finne$. 

The $\gcmnd$-comodule structure allows the full functoriality to be recovered. For $X \in \ob \Gamma$ a  finite pointed set, one has the morphisms
\[
F(X) 
\rightarrow 
\gcmnd F(X) \cong F(X_+) 
\rightarrow 
F(X),
\]
where the first is the comodule structure map and the second is induced by sending the added basepoint $+$ to the basepoint of $X$. The middle term is independent of the basepoint of $X$ and these morphisms induce a {\em canonical} isomorphism between $F(X,x)$ and $F(X,y)$ for any basepoints $x, y \in X$. 
\end{rem}


\subsection{Comparison of the functor categories $ \fcatk[\finj\op] $ and $ \fcatk [\fb\op]$}
\label{sect:fbfi}

In this section we relate the categories  $ \fcatk[\finj\op] $ and $ \fcatk [\fb\op]$. 

We begin by exhibiting the right adjoint of the restriction functor $\downarrow : \fcatk[\finj\op] \rightarrow \fcatk[\fb\op]$ induced by $\fb \op \subset \finj\op$.
Consider the functor in $\fcatk[\fb\op \times \finj]$ given by $(\mathbf{b}, \mathbf{n}) \mapsto \kring \hom_{\finj} (\mathbf{b}, \mathbf{n})$. (Observe that $\kring \hom_{\finj} (\mathbf{b}, \mathbf{n})= 0$ for $b >n$.) Hence, for $G$ a $\fb\op$-module, 
\[
\mathbf{n} \mapsto 
\hom _{\fcatk[\fb\op]} (\kring \hom_{\finj} ( \cdot , \mathbf{n}) , G (\cdot))
\]
is an $\finj\op$-module and the right hand side identifies as 
\begin{eqnarray}
\label{eqn:uparrow}
\bigoplus _{b\leq n} 
\hom _{\sym_b\op} (\kring \hom_{\finj} ( \mathbf{b} , \mathbf{n}) , G (\mathbf{b}))
\subset 
\bigoplus _{b\leq n} 
\hom _{\kring} (\kring \hom_{\finj} ( \mathbf{b} , \mathbf{n}) , G (\mathbf{b}))
.
\end{eqnarray}

\begin{defn}
\label{defn:uparrow}
Let $\uparrow: \fcatk[\fb\op] \to \fcatk[\finj\op]$ be the functor given by 
\[
\uparrow G := \Big (\mathbf{n} \mapsto 
\hom_{\fcatk[\fb\op]}  (\kring \hom_{\finj} ( \cdot , \mathbf{n}) , G (\cdot))
\Big)
\]
for $G \in \ob \fcatk[\fb\op]$.
\end{defn}

\begin{prop}
\label{prop:right_adjoint_fbfi}
The functor $\uparrow: \fcatk[\fb\op] \to \fcatk[\finj\op]$ is the right adjoint of the restriction functor $\downarrow : \fcatk[\finj\op] \rightarrow \fcatk[\fb\op]$. In other words, for $F \in \ob  \fcatk[\finj\op]$ and  $G \in \ob  \fcatk[\fb\op]$
there is a natural isomorphism:
\[
\hom _{\fcatk[\fb\op]} (\downarrow F, G) 
\cong 
\hom_{ \fcatk[\finj\op]} (F, 
\uparrow G).
\]
\end{prop}

\begin{proof}
We analyse $\hom_{ \fcatk[\finj\op] } (F (-),\hom _{\fcatk[\fb\op]} (\kring \hom_{\finj} ( \cdot , -) , G (\cdot)))$, based upon the identification in equation (\ref{eqn:uparrow}), which gives:
\[
\hom_{ \fcatk[\finj\op] } (F (-),\hom _{\fcatk[\fb\op]} (\kring \hom_{\finj} ( \cdot , -) , G (\cdot)))
\]
\[
\cong
\bigoplus _{b\leq n} 
\hom_{ \fcatk[\finj\op] } (F (-),
\hom _{\sym_b\op} (\kring \hom_{\finj} ( \mathbf{b} , -) , G (\mathbf{b})) 
).
\]

Fixing $b\in \nat$ and neglecting the $\sym_b$-action, we consider:
\[
\hom_{ \fcatk[\finj\op] } (F (-),\hom _{\kring} (\kring \hom_{\finj} (\mathbf{b}, - ) , G (\mathbf{b})))
.
\]

Fixing $n \in \nat$ and neglecting the functoriality with respect to $\finj$, one has the standard {\em natural} adjunction isomorphisms for $\kring$-modules:
\begin{eqnarray*}
 \hom_{\kring} (F (\mathbf{n}),\hom _{\kring} (\kring \hom_{\finj} (\mathbf{b}, \mathbf{n} ) , G (\mathbf{b})))
&\cong & 
\hom_{\kring} (F (\mathbf{n}) \otimes \kring \hom_{\finj} (\mathbf{b}, \mathbf{n} ) , G (\mathbf{b}))
\\
& \cong &
\hom_{\kring} ( \kring \hom_{\finj} (\mathbf{b}, \mathbf{n}) ,\hom_{\kring} (F (\mathbf{n}),  G (\mathbf{b}))).
\end{eqnarray*}

Then, taking into account the functoriality with respect to $\finj$ gives the first isomorphism below:
\begin{eqnarray*}
\hom_{\fcatk[\finj\op]} (F (-),\hom _{\kring} (\kring \hom_{\finj} (\mathbf{b}, - ) , G (\mathbf{b})))
&\cong& 
\hom_{\fcatk[\finj]} (\kring \hom_{\finj} (\mathbf{b}, - ), \hom _{\kring} ( F(-) , G (\mathbf{b})))
\\
&\cong &
\hom _{\kring} ( F(\mathbf{b}) , G (\mathbf{b})),
\end{eqnarray*}
the second isomorphism being given by Yoneda's lemma. 

Finally, taking into account the naturality with respect to $\fb\op$, one arrives at the required isomorphism.
\end{proof}

One can identify the values taken by the functor $\uparrow: \fcatk[\fb\op] \to \fcatk[\finj\op]$ by the following:

\begin{prop}
\label{prop:right_adjoint_fbfi-explicit}
For  $n\in \nat$ and $G$ a $\fb\op$-module,
\begin{enumerate}
\item
 there is a natural isomorphism:
 \[
\uparrow G(\mathbf{n})
 \cong 
 \bigoplus_{\mathbf{n}' \subseteq \mathbf{n}} G (\mathbf{n}');
 \]
\item 
\label{item:finj_op_structure_uparrow}
with respect to these isomorphisms, the morphism $\uparrow G (\mathbf{n}) \rightarrow \uparrow G (\mathbf{l})$ induced by  $f : \mathbf{l} \to \mathbf{n}$ in $\finj $ has restriction to the factor  $G(\mathbf{n}') \subset \uparrow G(\mathbf{n})$ indexed by  $\mathbf{n}' \subseteq \mathbf{n}$ given by 
 \begin{enumerate}
 \item 
 zero if $\mathbf{n'} \not \subseteq f(\mathbf{l})$; 
 \item 
if $\mathbf{n}' \subseteq f (\mathbf{l})$,  the composite
 \[
 \xymatrix{
 G(\mathbf{n}')
 \ar[rr]^{G(f | _{f^{-1} (\mathbf{n}')})}
 &&
 G(f^{-1} (\mathbf{n}'))
 \ar@{^(->}[r] 
 &
\uparrow G (\mathbf{l}),  
 }
 \]
 where $f|_{f^{-1} (\mathbf{n}')}: f^{-1} (\mathbf{n}') \stackrel{\cong}{\rightarrow} \mathbf{n}'$. 
 \end{enumerate}
\end{enumerate}
 \end{prop}

\begin{proof}
Consider the first statement. We have 
$$\uparrow G(\mathbf{n}) =\hom_{\fcatk[\fb\op]}  (\kring \hom_{\finj} ( - , \mathbf{n}) , G (-))
\Big)
\cong 
\bigoplus _{b\leq n} 
\hom _{\sym_b\op} (\kring \hom_{\finj} ( \mathbf{b} , \mathbf{n}) , G (\mathbf{b})).
$$

For $b, n\in \nat$, $\hom_{\finj} (\mathbf{b}, \mathbf{n})$ is a free $\sym_b\op$-set on the set of subsets $\mathbf{n}'\subseteq \mathbf{n}$ of cardinal $b$. More precisely, 
\[
\hom_{\finj} (\mathbf{b}, \mathbf{n}) 
\cong 
\coprod _{\substack{\mathbf{n}' \subseteq \mathbf{n}\\
|\mathbf{n}'|=b}}
\hom_{\finj} (\mathbf{b}, \mathbf{n}') 
\]
where $\hom_{\finj} (\mathbf{b}, \mathbf{n}') \cong \sym_b$ as a right $\sym_b\op$-set. 
Hence: 
$$\uparrow G(\mathbf{n})
\cong 
 \bigoplus _{b\leq n}  \bigoplus_{\substack{\mathbf{n}' \subseteq \mathbf{n}\\
|\mathbf{n}'|=b}} G(\mathbf{b})
\cong  
\bigoplus _{b\leq n}  \bigoplus_{\substack{\mathbf{n}' \subseteq \mathbf{n}\\
|\mathbf{n}'|=b}} G(\mathbf{n}')
\cong 
\bigoplus_{\mathbf{n}' \subseteq \mathbf{n}} G(\mathbf{n}').$$

The second statement follows by considering the inclusion $\hom_{\finj} (-, \mathbf{l}) \hookrightarrow \hom_{\finj} (-, \mathbf{n})$ induced by $f$. Evaluated on $\mathbf{b}$, this corresponds to the inclusion 
\[
\coprod _{\substack{\mathbf{l}' \subseteq \mathbf{l}\\
|\mathbf{l}'|=b}}
\hom_{\finj} (\mathbf{b}, \mathbf{l}') 
\hookrightarrow 
\coprod _{\substack{\mathbf{n}' \subseteq f(\mathbf{l})\\
|\mathbf{n}'|=b}}
\hom_{\finj} (\mathbf{b}, \mathbf{n}')
\subseteq 
\coprod_{\substack{\mathbf{n}' \subseteq \mathbf{n}\\
|\mathbf{n}'|=b}}
\hom_{\finj} (\mathbf{b}, \mathbf{n}'),
\]
where the first map sends $\hom_{\finj} (\mathbf{b}, \mathbf{l}')$ to $\hom_{\finj} (\mathbf{b}, f(\mathbf{l}'))$
by postcomposition with  $f|_{\mathbf{l}'}$.

The result follows.
\end{proof}

\begin{rem}
For simplicity of notation, in the rest of the paper, for $F \in \fcatk[\finj\op]$, the restriction of $F$ to $\fb\op$ will be denoted by $F$ instead of $\downarrow F$.
\end{rem}
\begin{nota}
\label{nota:fbcmnd}
Let $\fbcmnd : \fcatk[\fb\op]\rightarrow \fcatk[\fb\op]$ be the comonad 
associated to the adjunction of Proposition \ref{prop:right_adjoint_fbfi} (i.e., $\fbcmnd=\downarrow \circ \uparrow)$, 
with structure morphisms $\Delta : \fbcmnd \rightarrow \fbcmnd \fbcmnd $ and counit $\epsilon : \fbcmnd \rightarrow \mathrm{Id}$. 
\end{nota}

\begin{prop}
\label{prop:identify_fbcmnd}
Let  $G \in \ob \fcatk[\fb\op]$ and $b \in \nat$.
\begin{enumerate}
\item
\label{item:p1}
There is a natural isomorphism:
\begin{eqnarray}
\label{eqn:fbcmnd}
\fbcmnd G(\mathbf{b}) \cong 
\bigoplus_{\mathbf{b}' \subseteq \mathbf{b} } G (\mathbf{b}').
\end{eqnarray}
\item 
\label{prop:identify_fbcmnd2}
The counit $\epsilon^\Sigma_G : \fbcmnd G (\mathbf{b}) \rightarrow G (\mathbf{b}) $ identifies as the projection 
 $ 
 \bigoplus_{\mathbf{b}' \subseteq \mathbf{b}} G(\mathbf{b}')
 \twoheadrightarrow 
 G(\mathbf{b}) 
 $
onto the summand indexed by $\mathbf{b}$. 
\item 
\label{prop:identify_fbcmnd3}
There is a natural  isomorphism
$
\fbcmnd \fbcmnd G (\mathbf{b}) 
\cong 
\bigoplus _{\mathbf{b}'' \subseteq \mathbf{b}' \subseteq \mathbf{b}}
G(\mathbf{b}'')$.
\item 
\label{prop:identify_fbcmnd4}
$\Delta : \fbcmnd G(\mathbf{b}) \rightarrow \fbcmnd \fbcmnd G(\mathbf{b})$ identifies as the morphism 
\[
 \bigoplus_{\mathbf{l}' \subseteq \mathbf{b}} G(\mathbf{l}')
\rightarrow
\bigoplus _{\mathbf{b}'' \subseteq \mathbf{b}' \subseteq \mathbf{b}}
G(\mathbf{b}'')
\]
with component 
$
G(\mathbf{l}')
\rightarrow
G(\mathbf{b}''),
$
the identity if $\mathbf{l}' = \mathbf{b}''$ and zero otherwise.
\end{enumerate}
\end{prop}

\begin{proof}
The first two statements follow from Proposition \ref{prop:right_adjoint_fbfi-explicit}.

Since $\fbcmnd G \in \ob \fcatk[\fb\op]$, we can apply (\ref{eqn:fbcmnd}) to this functor to obtain 
$$\fbcmnd (\fbcmnd G) (\mathbf{b}) \cong \bigoplus _{\mathbf{b}' \subseteq \mathbf{b}} (\fbcmnd G)(\mathbf{b}').$$ 
Applying  (\ref{eqn:fbcmnd}) to each term $\fbcmnd G (\mathbf{b}') $ gives the expression for $\fbcmnd \fbcmnd G (\mathbf{b}) $.

The natural morphism $\Delta : \fbcmnd G \rightarrow \fbcmnd \fbcmnd G$ is induced by the $\finj\op$-module structure of $\uparrow G$, which is given by Proposition \ref{prop:right_adjoint_fbfi-explicit} (\ref{item:finj_op_structure_uparrow}). This leads to the stated identification.
\end{proof}

\begin{rem}
\label{rem:reindexation2}
The decomposition of $ \fbcmnd_0 \fbcmnd_1 G$ (we label the functors for clarity) given in Proposition \ref{prop:identify_fbcmnd} (\ref{prop:identify_fbcmnd3}) can be reindexed as follows:
$$  \fbcmnd_0 (\fbcmnd_1 G) (\mathbf{b}) 
\cong 
\bigoplus_{\mathbf{b}= \mathbf{b}^{(1)}_\Sigma \amalg \mathbf{b}^{\Sigma}_0 }
\fbcmnd_1 G (\mathbf{b}^{(1)}_\Sigma)
\cong
\bigoplus_{\mathbf{b}= (\mathbf{b}^{(2)} _\Sigma\amalg \mathbf{b}^{\Sigma}_1) \amalg \mathbf{b}^{\Sigma}_0 }
G (\mathbf{b}^{(2)}_\Sigma).  
$$
The sum is indexed by ordered decompositions of $\mathbf{b}$ into three subsets (possibly empty). With respect to the decomposition given in Proposition \ref{prop:identify_fbcmnd} (\ref{prop:identify_fbcmnd3}), we have  $\mathbf{b}^{(2)}_\Sigma := \mathbf{b}''$, $\mathbf{b}^{\Sigma}_1:=\mathbf{b}'\backslash \mathbf{b}''$ and $\mathbf{b}^{\Sigma}_0:=\mathbf{b}\backslash \mathbf{b}'$. The indices in $\mathbf{b}^{\Sigma}_0, \mathbf{b}^{\Sigma}_1$ record from which application of $\fbcmnd$ the subset $\mathbf{b}^{\Sigma}_i$ arises. More generally, we have:
\begin{equation}
\label{iteration-Sigma}
 \fbcmnd_0 \fbcmnd_1 \ldots \fbcmnd_\ell G (\mathbf{b}) 
\cong
\bigoplus_{\mathbf{b}=(\ldots ((\mathbf{b}^{(\ell+1)}_\Sigma \amalg \mathbf{b}^{\Sigma}_\ell) \amalg \mathbf{b}^{\Sigma}_{\ell-1})) \ldots \amalg \mathbf{b}^{\Sigma}_1) \amalg \mathbf{b}^{\Sigma}_0 }
G (\mathbf{b}^{(\ell+1)}_\Sigma).  
\end{equation}
\end{rem}

The general theory of comonads associated to an adjunction implies that, if $F \in \ob \fcatk[\finj\op]$, the underlying  object $F \in \ob \fcatk[\fb\op]$ has a canonical $\fbcmnd$-comodule structure.

\begin{prop}
\label{prop:fi_comod}
For $F \in \ob \fcatk[\finj\op]$, the  $\fbcmnd$-comodule structure  of the underlying  object $\downarrow F \in \ob \fcatk[\fb\op]$ (denoted below simply by $F$) identifies   with respect to the isomorphism (\ref{eqn:fbcmnd}) of Proposition \ref{prop:identify_fbcmnd} as follows. For $b \in \nat$, the structure morphism $\psi_F : F  \rightarrow \fbcmnd F $ is given by 
\[
\psi _F : F (\mathbf{b}) 
\rightarrow 
\bigoplus 
_{\mathbf{b}' \subseteq \mathbf{b}} 
F (\mathbf{b}')
\]
with component $F (\mathbf{b}) \rightarrow F(\mathbf{b}') $ given by the $\finj\op$ structure of $F$ for $\mathbf{b}' \subset \mathbf{b}$. 
\end{prop}

The Barr-Beck theorem implies:

\begin{thm}
\label{thm:barr_beck_finjop}
The category  $\fcatk[\finj\op]$ is equivalent to the category $\fcatk[\fb\op]_\fbcmnd$ of $\fbcmnd$-comodules in $\fcatk[\fb\op]$. 
\end{thm}

Proposition \ref{prop:fbcmnd_coaug} below shows that the comonad $\fbcmnd$ is equipped with a natural transformation $\eta : \id \rightarrow \fbcmnd $,  which can be considered as a form of coaugmentation. This form  of additional structure is considered in Appendix \ref{subsect:cobar}; it is important that this satisfies the Hypothesis \ref{hyp:coaugmentation}.

\begin{prop}
\label{prop:fbcmnd_coaug}
There is a natural transformation $\eta : \id \rightarrow \fbcmnd $  of endofunctors of $\fcatk[\fb\op]$, where the morphism $\eta_G : G(\mathbf{b})  \rightarrow \fbcmnd G(\mathbf{b}) $, for $G \in \ob \fcatk[\fb\op]$ and $b \in \nat$, is given by
\[
G(\mathbf{b}) 
\rightarrow 
\bigoplus_{\mathbf{b}' \subseteq \mathbf{b}}
G(\mathbf{b}'),
\]
the inclusion of the factor indexed by $\mathbf{b}'=\mathbf{b}$. 

Hence the following  diagrams commute:
\[
\xymatrix{
G 
\ar[r]^{\eta_G}
\ar[d]_{\eta_G}
&
\fbcmnd G
\ar[d]^{\Delta_G}
&
 G
\ar[r]^{\eta_{G}}
\ar[rd]_{1_{G}}
&
 \fbcmnd G 
\ar[d]^{\epsilon^\Sigma_{ G}}
\\
\fbcmnd G
\ar[r]_{\eta_{(\fbcmnd G)}}
&
\fbcmnd \fbcmnd G 
&
&
G.
}
\]

Thus, the structure $(\fcatk[\fb\op]; \fbcmnd, \Delta, \epsilon, \eta)$ satisfies Hypothesis \ref{hyp:coaugmentation}. 
\end{prop}

\begin{proof}
This is proved by direct verification, using Proposition \ref{prop:identify_fbcmnd}. 
\end{proof}

\begin{rem}
The constructions of this Section apply verbatim to functors from $\fb\op$ to any abelian category. 
\end{rem}

\section{The Koszul complex}
\label{sect:koszul}

The purpose of this section is to introduce the Koszul complex $\kz F$ 
in $\fb\op$-modules,  where $F$ is an $\finj \op$-module, and to identify its cohomology.
For this, we apply the general construction given in Appendix \ref{cosimplicial} in our specific situation. More precisely, Proposition \ref{prop:fbcmnd_coaug} shows that the comonad $(\fbcmnd, \Delta, \epsilon)$ on $\fcatk[\fb\op]$ is equipped with a coaugmentation $\eta : \id \rightarrow \fbcmnd $  satisfying Hypothesis \ref{hyp:coaugmentation} and Theorem \ref{thm:barr_beck_finjop} shows that the $\finj\op$-module $F$ can be considered equivalently as a $\fbcmnd$-comodule in $\fcatk[\fb\op]$. Hence, one can form the cosimplicial $\fb\op$-module $\cosimp^\bullet F$, as in Proposition \ref{prop:cosimp_comod}, which has the following form:
\begin{equation}
\label{cosimplicial-obj-1}
\xymatrix{
F
\ar@<.75ex>[rr]|{d^0=\eta_F}
\ar@<-.75ex>[rr]| {d^1=\psi_F}
&&
\fbcmnd F
\ar@/_1pc/@<-.5ex>[ll]|{ s^0=\epsilon_F}
\ar@<1.5ex>[rr]|{ d^0=\eta_{\perp F}}
\ar[rr]|{ d^1=\Delta_F}
\ar@<-1.5ex>[rr]| {d^2=\perp \psi_F
}
&&
\fbcmnd \fbcmnd F
\ar@/_1pc/@<-2.5ex>[ll]|{s^0=\epsilon_{\perp F}}
\ar@/_1pc/@<-1ex>[ll]|{s^1=\perp \epsilon_F}
\ar@{.>}[rr]
&&
\fbcmnd \fbcmnd \fbcmnd F
\ldots  \ .
\ar@{.>}@<-1ex>@/_1pc/[ll]
}
\end{equation}
In order to compare this cosimplicial object to the Koszul complex, we consider its opposite $(\cosimp^\bullet F)\op $.
Considering the normalized subcomplex associated to a cosimplicial object (see Appendix  \ref{sect:app_norm}), the main result of the Section is the following Theorem, in which $\kz F$ denotes the Koszul complex introduced in Section \ref{subsect:kzF}:

\begin{thm}
\label{thm:norm_vs_kz}
For $F \in \ob \fcatk[\finj\op]$, there is a natural surjection of complexes of $\fb\op$-modules
\[
N (\cosimp^\bullet F)\op 
\twoheadrightarrow 
\kz F
\]
which is a quasi-isomorphism, where $(\cosimp^\bullet F)\op $ is the {\em opposite}  cosimplicial $\fb\op$-module to $\cosimp^\bullet F$ (cf. Notation \ref{nota:opp_simp}). 
\end{thm}

\begin{rem}
The complex $\kz F$ is the $\finj\op$-version of the Koszul complex considered in 
\cite[Theorem 2]{MR3603074} for $\finj$-modules, where it is shown that this has the same cohomology as 
$N (\cosimp^\bullet F)\op$ by using  \cite{MR3654111}.

To keep this paper self-contained, we give a direct proof of Theorem \ref{thm:norm_vs_kz}, rather than deriving it from \cite[Theorem 2]{MR3603074}. The relation between the Koszul complex and $\finj\op$-cohomology is explained in Section \ref{FIop-cohom}.
\end{rem}

\subsection{The Koszul complex $\kz F$}
\label{subsect:kzF}

In this section we construct, for an $\finj \op$-module $F$, the complex $\kz F$ in $\fb\op$-modules.
To keep track of the signs in the construction, we introduce the following notation:

\begin{nota}
\label{nota:orient}
For $S$ a finite set,
\begin{enumerate}
\item
denote by $\orient (S)$ the orientation module associated to $S$, i.e.,   the free, rank one $\kring$-module $\Lambda^{|S|} (\kring S)$ (the top exterior power), where $\kring S$ denotes the free $\kring$-module generated by $S$;
\item 
$\orient (S)$ is considered as having cohomological degree $|S|$; 
\item 
if $S$ is ordered with elements $x_1 < \ldots < x_{|S|}$, let $\oclass (S) \in \orient (S)$ denote the element:
\[
\oclass (S):= x_1 \wedge \ldots \wedge x_{|S|}.
\]
\end{enumerate} 
\end{nota}

\begin{rem}\ 
\begin{enumerate}
\item 
By construction, $\orient(S)$ is  a $\sym_{|S|}$-module; it is isomorphic to the signature representation $\sgnrep_{\sym_{|S|}}$. In particular, we may consider it either as  a left or a right  $\sym_{|S|}$-module without ambiguity.
\item 
The orientation modules $\orient (\mathbf{b})$, $b \in \nat$, assemble to give a $\fb\op$-module. 
\item
 If $S$ is ordered, then $\oclass (S)$ gives a canonical generator of $\orient (S)$.
 \end{enumerate}
\end{rem}

\begin{defn}
\label{defn:Cbb}
For $t \in \nat$ and $b\in \nat$,
\begin{enumerate}
\item 
let $(\kz F)^t(\mathbf{b}) $ be the $\sym_b\op$-module
\[
\bigoplus_{\substack{ \mathbf{b}^{(t)} \subset \mathbf{b}\\
|\mathbf{b}^{(t)}|=b-t }}
F(\mathbf{b}^{(t)})
\otimes 
\orient (\mathbf{b} \backslash \mathbf{b}^{(t)}),
\]
where $g \in \sym_b$ sends the summand indexed by $\mathbf{b}^{(t)}$ to that indexed by $g^{-1}(\mathbf{b}^{(t)}) $ via the morphism
\[
F(\mathbf{b}^{(t)})
\otimes 
\orient (\mathbf{b} \backslash \mathbf{b}^{(t)})
\rightarrow 
F (g^{-1}(\mathbf{b}^{(t)}))
\otimes 
\orient (g^{-1}(\mathbf{b} \backslash \mathbf{b}^{(t)}))
\]
given by the action of $g|_{g^{-1}(\mathbf{b}^{(t)}) }$ on the first tensor factor and that of $g |_{g^{-1}(\mathbf{b} \backslash \mathbf{b}^{(t)}) }$ on the second factor.
\item 
Let $d : (\kz  F) ^t (\mathbf{b})
\rightarrow  (\kz  F) ^{t +1} (\mathbf{b})$ be the morphism of $\kring$-modules given on  $Y \otimes \alpha$, where $Y \in F(\mathbf{b}^{(t)})$  and $\alpha \in \orient (\mathbf{b}\backslash \mathbf{b}^{(t)})$, by 
\[
d \big (Y \otimes \alpha  \big) 
= 
\sum_{x \in \mathbf{b}^{(t)}} 
\iota_x^* Y 
\otimes (x \wedge \alpha), 
\]
where $\iota_x : \mathbf{b}^{(t)} \backslash \{ x\}
\hookrightarrow \mathbf{b}^{(t)}$ is the inclusion. 
\end{enumerate}
\end{defn}

This construction is compatible with the action of the group $\sym_b$:

\begin{lem}
\label{lem:Cbb_action}
For $t \in \nat$ and $b \in \nat$
\begin{enumerate}
\item 
there is an isomorphism of $\sym_b\op$-modules:
 \[
(\kz F)^t (\mathbf{b}) 
 \cong 
\big(
F(\mathbf{b - t})
\otimes 
\orient (\mathbf{t}') 
\big)\uparrow _{\sym_{b-t}\op \times \sym_t \op}^{\sym_b\op},
\]
for $\mathbf{b-t} \subset \mathbf{b}$ the canonical inclusion with complement $\mathbf{t}'$;
\item 
the morphism $d : (\kz F)^t  (\mathbf{b})
\rightarrow (\kz F)^{t+1} (\mathbf{b}) $ 
is a morphism of $\sym_b\op$-modules.
\end{enumerate}
\end{lem}

\begin{proof}
The inclusion 
$
F (\mathbf{b - t})
\otimes 
\orient (\mathbf{t}') 
\hookrightarrow 
(\kz F) ^t(\mathbf{b}) 
$ 
corresponding to the summand indexed by $\mathbf{b-t}$ is $(\sym_{b-t} \times \sym_t)\op$-equivariant, thus induces up to a morphism of $\sym_b\op$-modules. It is straightforward to check that this is an isomorphism, as required.

It remains to check that the morphism $d$ is equivariant with respect to the $\sym_b\op$ action. This follows  from the explicit formula for $d$ given in Definition \ref{defn:Cbb}.
\end{proof}

\begin{prop}
\label{prop:kz_ch_cx}
The construction $(\kz F, d)$ is a cochain complex in $\fcatk[\fb\op]$. Hence, $H^* (\kz F)$ takes values in $\nat$-graded $\fb\op$-modules.
\end{prop}

\begin{proof}
It suffices to show that $d$ is a differential. Using the notation of Definition \ref{defn:Cbb}, 
\[
d^2 \big (Y \otimes \alpha  \big) 
= 
\sum_{\{x, y \} \subset  \mathbf{b}^{(t)}} 
\iota^*_{x,y} Y
\otimes \big ( (y \wedge x \wedge \alpha) + (x \wedge y \wedge \alpha) \big) , 
\]
where $\iota_{x,y}$ denotes the inclusion  $\mathbf{b}^{(t)} \backslash \{x, y \} \subset \mathbf{b}^{(t)} $. This is zero, as required. 
\end{proof}

The following Proposition will be used in Section \ref{sect:cohomology}.

\begin{prop} 
\label{Koszul-exact}
The Koszul construction defines an exact functor
$$\kz: \fcatk[\finj \op] \to \cochain(\fcatk[\fb\op]),$$
 where $\cochain(\fcatk[\fb\op])$ is the category of cochain complexes in $\fcatk[\fb\op]$.
\end{prop}

\begin{proof}
 The functoriality is clear from the construction. Moreover, the functor $F \mapsto (\kz F)^t (\mathbf{b})$ as in Definition \ref{defn:Cbb} is exact, since  the underlying $\kring$-module of each $\orient (\mathbf{b} \backslash \mathbf{b}^{(t)})$ is flat (more precisely, free of rank one). This implies that $\kz$ is exact considered as a functor to cochain complexes. 
\end{proof}

\subsection{The cosimplicial $\fb\op$-module $(\cosimp^\bullet F)\op$ and its normalized cochain complex}
In this section, we identify the cosimplicial $\fb\op$-module $(\cosimp^\bullet F)\op$ and its associated normalized cochain complex $N (\cosimp^\bullet F)\op$.

In order to describe $(\cosimp^\bullet F)\op$, we will use the following decompositions, given by Remark \ref{rem:reindexation2} (\ref{iteration-Sigma}), distinguishing the indices of $(\fbcmnd)^{\ell} F (\mathbf{b}) $ by a $\tilde{\ }$:
$$
 (\fbcmnd)^{\ell+1} F (\mathbf{b}) = \fbcmnd_0 \fbcmnd_1 \ldots \fbcmnd_\ell F (\mathbf{b}) 
\cong
\bigoplus_{\mathbf{b}=(\ldots ((\mathbf{b}^{(\ell+1)}_\Sigma \amalg \mathbf{b}^{\Sigma}_\ell) \amalg \mathbf{b}^{\Sigma}_{\ell-1})) \ldots \amalg \mathbf{b}^{\Sigma}_1) \amalg \mathbf{b}^{\Sigma}_0 }
F (\mathbf{b}^{(\ell+1)}_\Sigma)$$
$$ (\fbcmnd)^{\ell} F (\mathbf{b}) = \fbcmnd_0 \fbcmnd_1 \ldots \fbcmnd_{\ell-1} F (\mathbf{b}) 
\cong
\bigoplus_{\mathbf{b}=(\ldots ((\tilde{\mathbf{b}}^{(\ell)}_\Sigma \amalg \tilde{\mathbf{b}}^{\Sigma}_{\ell-1}) \amalg \tilde{\mathbf{b}}^{\Sigma}_{\ell-2})) \ldots \amalg \tilde{\mathbf{b}}^{\Sigma}_1) \amalg \tilde{\mathbf{b}}^{\Sigma}_0 }
F (\tilde{\mathbf{b}}^{(\ell)}_\Sigma)$$

The structure of $\cosimp^\bullet F$ is described in (\ref{cosimplicial-obj-1}) and, by Notation \ref{nota:opp_simp}, $(\cosimp^\bullet F)\op$  is the cosimplicial object with coface operators $\tilde{d}^i: (\fbcmnd)^{\ell} F \rightarrow (\fbcmnd)^{\ell+1} F$ and codegeneracy operators $\tilde{s}^j: (\fbcmnd)^{\ell+1} F \rightarrow (\fbcmnd)^{\ell} F$   satisfying $\tilde{s}^j=s^{\ell-j}$ and $\tilde{d}^j=d^{\ell+1-j}$.
\begin{prop}
\label{prop:cosimplicial-structure-CF}
The structure morphisms of the cosimplicial  $\fb\op$-module $(\cosimp^\bullet F)\op$ are given by the following:
\begin{enumerate}
\item 
for $0 \leq j \leq \ell$, the restriction of $\tilde{s}^j: (\fbcmnd)^{\ell+1} F \rightarrow (\fbcmnd)^{\ell} F$ to the factor indexed by $\mathbf{b}= \mathbf{b}_\Sigma^{(\ell+1)} \amalg \coprod_{i=0}^{\ell} \mathbf{b}^{\Sigma}_{i}$ is:
\begin{enumerate}
\item 
$0$  if $\mathbf{b}^\Sigma_{j} \neq \emptyset$;
\item
if $\mathbf{b}^\Sigma_{j} = \emptyset$, the projection onto the direct summand of $(\fbcmnd)^{\ell} F$ such that
\begin{eqnarray*}
\tilde{\mathbf{b}}^{(\ell)}_\Sigma&= & \mathbf{b}^{(\ell +1)}_\Sigma\\
\tilde{\mathbf{b}}^{\Sigma}_i&=& 
\left\{
\begin{array}{ll}
\mathbf{b}^{\Sigma}_i & i < j \\
\mathbf{b}^{\Sigma}_{i+1} & j \leq  i \leq \ell;
\end{array}
\right. 
\end{eqnarray*}
\end{enumerate}
\item 
for $0 \leq i \leq \ell +1$,  the restriction of $\tilde{d}^i: (\fbcmnd)^{\ell} F \rightarrow (\fbcmnd)^{\ell+1} F$ to the factor indexed by $\mathbf{b}= \tilde{\mathbf{b}}_\Sigma^{(\ell)} \amalg \coprod_{i=0}^{\ell-1} \tilde{\mathbf{b}}^{\Sigma}_{i}$  is given as follows:

\begin{enumerate}
\item 
for $i=0$, it maps to the direct summand of $(\fbcmnd)^{\ell+1} F$ such that 
$\tilde{\mathbf{b}}_{\Sigma}^{(\ell)}=\mathbf{b}_{\Sigma}^{(\ell+1)}  \amalg \mathbf{b}^{\Sigma}_{0}$
and 
$\tilde{\mathbf{b}}^{\Sigma}_{i}=\mathbf{b}^{\Sigma}_{i+1}$ ($0 \leq i \leq \ell-1$) 
via the injection $\mathbf{b}_{\Sigma}^{(\ell+1)} \hookrightarrow \tilde{\mathbf{b}}_{\Sigma}^{(\ell)}$.
\item 
for $0 < i \leq \ell$, is the identity map to each direct summand of $(\fbcmnd)^{\ell+1} F$  such that $\tilde{\mathbf{b}}_{\Sigma}^{(\ell)}=\mathbf{b}_{\Sigma}^{(\ell+1)}$ and 
\[
\tilde{\mathbf{b}}^{\Sigma}_j = 
\left\{
\begin{array}{ll}
\mathbf{b}_j^{\Sigma}   &   j <i \\
\mathbf{b}^{\Sigma}_i \amalg \mathbf{b}^{\Sigma}_{i+1} & j=i \\
\mathbf{b}^{\Sigma}_{j+1} & i<j \leq \ell-1 ;
\end{array}
\right. 
\]
it is zero to the other summands of  $(\fbcmnd)^{\ell+1} F$;
\item 
\label{prop:cosimplicial-structure-CF-2c}
for $i=\ell +1$, is the identity map to the direct summand of $(\fbcmnd)^{\ell+1} F$ corresponding to  $\mathbf{b}_{\Sigma}^{(\ell+1)}=\tilde{\mathbf{b}}_{\Sigma}^{(\ell)}$ and  $\mathbf{b}^{\Sigma}_i= \tilde{\mathbf{b}}^{\Sigma}_i$ for $0 \leq i < \ell$ and $\mathbf{b}_{\ell}^{\Sigma}= \emptyset$; it is zero to the other summands of  $(\fbcmnd)^{\ell+1} F$.
\end{enumerate}
\end{enumerate}
\end{prop} 

\begin{proof}
The general forms of  $s^j: (\fbcmnd)^{\ell+1} F \rightarrow (\fbcmnd)^{\ell} F$ and $d^i: (\fbcmnd)^{\ell} F \rightarrow (\fbcmnd)^{\ell+1} F$ are given in Proposition \ref{prop:cosimp_comod}. Using Notation \ref{nota:opp_simp} (recalled before the statement), we deduce the explicit forms given of $\tilde{s}^j$ and $\tilde{d}^i$ using Proposition \ref{prop:identify_fbcmnd} (\ref{prop:identify_fbcmnd2}) and (\ref{prop:identify_fbcmnd4}), Proposition \ref{prop:fi_comod} and Proposition \ref{prop:fbcmnd_coaug}.
\end{proof}

For the version of the normalized cochain complex used in the following, see Definition \ref{defn:normalized}.
\begin{prop}
\label{prop:norm_cosimp_F}
For $F \in \ob \fcatk[\finj\op]$, the associated normalized cochain complex $ 
N (\cosimp^\bullet F)\op
$ evaluated on $\mathbf{b}$, for $b \in \nat$, 
 has terms:
\[
(N(\cosimp^\bullet F)\op) ^t (\mathbf{b}) = 
\bigoplus_{\substack{\mathbf{b} = \mathbf{b}^{(t)} \amalg \coprod_{i=0}^{t -1} \mathbf{b}_i^{(t)}
\\
{\mathbf{b}^{(t)}_i \neq \emptyset}}}
F (\mathbf{b}^{(t)}). 
\]
In particular, $(N(\cosimp^\bullet F)\op) ^t (\mathbf{b}) =0$ if $t <0$ or if $t >|\mathbf{b}|$.

The differential $d: ( N (\cosimp^\bullet F)\op)^t \rightarrow (N(\cosimp^\bullet F)\op) ^{t +1}$ is given by $\sum_{i=0}^{t }(-1)^i  \tilde{d}^i=(-1)^{t+1}\sum_{i=1}^{t+1}(-1)^i {d}^i$.
\end{prop}

\begin{proof}
The identification of the normalized subcomplex follows  from the form of the codegeneracies $\tilde{s}^j$ given in Proposition \ref{prop:cosimplicial-structure-CF}.
 The vanishing statement follows immediately. 

By Definition \ref{defn:normalized},  the differential on $N (\cosimp^\bullet F)\op$ is the restriction of that on the cochain complex associated to $(\cosimp^\bullet F)\op$, i.e., $\sum_{i=0}^{t + 1}(-1)^i  \tilde{d}^i$. By Proposition \ref{prop:cosimplicial-structure-CF} (\ref{prop:cosimplicial-structure-CF-2c}) the restriction of $\tilde{d}^{t+1}$ to $N (\cosimp^\bullet F)\op$ is zero.
\end{proof}

\subsection{Relating the Koszul complex and the normalized cochain complex} 

In this section, we define a natural surjection from $N (\cosimp^\bullet F)\op$ to the Koszul complex $\kz F$. The proof that this surjection is a quasi-isomorphism is given in the subsequent sections.

To define this natural surjection, we use the following direct summand of the underlying $\nat$-graded object of the normalized complex $N(\cosimp^\bullet F)\op$; we stress that it is not in general a direct summand as a complex.

\begin{nota}
For $F \in \ob \fcatk[\finj\op]$ and $t, b\in \nat$, let $\overline{(N(\cosimp^\bullet F)\op) ^t (\mathbf{b})}$ be the direct summand of $(N(\cosimp^\bullet F)\op) ^t (\mathbf{b})$ consisting of terms such that $|\mathbf{b}_i ^{(t)} |=1$ for each $0 \leq i \leq t-1$ (so that $|\mathbf{b}^{(t)}|= b-t$). 
\end{nota}

The $\sym_b\op$-action on $\overline{(N(\cosimp^\bullet F)\op) ^t (\mathbf{b})}$  is described as follows, in a form suitable for comparison with the Koszul complex:

\begin{prop}
\label{prop:NcosimpF_b_action}
For $F \in \ob \fcatk[\finj\op]$ and $b,t \in \nat$,
\begin{enumerate}
\item 
$\overline{(N(\cosimp^\bullet F)\op) ^t (\mathbf{b})}$ is a direct summand of $(N(\cosimp^\bullet F)\op) ^t (\mathbf{b})$ as $\sym_b\op$-modules; 
\item 
 there is an isomorphism of $\sym_b\op$-modules:
$$
\big(
F (\mathbf{b - t}) 
\otimes 
\kring \aut (\mathbf{t}') 
\big)\uparrow _{\sym_{b-t}\op \times \sym_t \op}^{\sym_b\op}
\stackrel{\cong}{\rightarrow}
\overline{(N(\cosimp^\bullet F)\op)^t (\mathbf{b})}
,$$
 for $\mathbf{b-t} \subset \mathbf{b}$ the canonical inclusion with complement $\mathbf{t}'$.
\end{enumerate}
\end{prop}

\begin{proof}
The first statement is clear. 

The inclusion $\mathbf{b-t} \subset \mathbf{b}$ induces a $\sym_{b-t}\op$-equivariant inclusion 
$
F (\mathbf{b-t} ) \hookrightarrow \overline{(N(\cosimp^\bullet F)\op)^t (\mathbf{b})}
$
corresponding to the summand indexed by taking  $\mathbf{b}_i ^{(t)} = \{b-t+1+i\}$ for $0 \leq i \leq  t-1$.  This induces a morphism of $\sym_b\op$-modules, as in the statement. By inspection, this is an isomorphism. 
\end{proof}

\begin{prop}
\label{prop:kz_quotient}
There is a  surjection $
N (\cosimp^\bullet F)\op
\twoheadrightarrow 
\kz F $ of cochain complexes in $\fcatk[\fb\op]$ 
given, for $t,b \in \nat$, as the composite 
\[
(N(\cosimp^\bullet F)\op) ^t (\mathbf{b}) 
\twoheadrightarrow 
\overline{(N(\cosimp^\bullet F)\op) ^t (\mathbf{b})}
\twoheadrightarrow 
(\kz F)^t(\mathbf{b}) , 
\]
where the first morphism is the projection of Proposition \ref{prop:NcosimpF_b_action} and the second is 
\[
\overline{(N(\cosimp^\bullet F)\op) ^t (\mathbf{b})}
\cong 
\big(
F (\mathbf{b - t}) 
\otimes 
\kring \aut (\mathbf{t}') 
\big)\uparrow _{\sym_{b-t}\op \times \sym_t \op}^{\sym_b\op}
\twoheadrightarrow
\big(
F(\mathbf{b - t})
\otimes 
\orient (\mathbf{t}') 
\big)\uparrow _{\sym_{b-t}\op \times \sym_t \op}^{\sym_b\op}
\cong 
(\kz F)^t (\mathbf{b}), 
\]
induced by the surjection  $
\kring \aut (\mathbf{t}')
\twoheadrightarrow 
\orient (\mathbf{t}') 
$
 of $\sym_t\op$-modules
that sends the generator $[\mathrm{id}_{\mathbf{t}'}]$ to $\oclass (\mathbf{t}')$.

Explicitly,  an element $Y \in F (\mathbf{b}^{(t)})$  that represents a class of $(N (\cosimp^\bullet  F)\op)^t (\mathbf{b})$ in the summand indexed by $(\mathbf{b}^{(t)}; \mathbf{b}_0^{(t)} , \ldots , \mathbf{b}_{t-1} ^{(t)})$ in the decomposition given in Proposition \ref{prop:norm_cosimp_F}
 is sent to zero, unless $|\mathbf{b}_i ^{(t)}|=1$ for all $i$,  $0 \leq i \leq t-1$, when it is sent to 
\[
Y \otimes (\mathbf{b}_0^{(t)} \wedge \ldots \wedge 
\mathbf{b}_{t-1}^{(t)} ) 
\in 
F(\mathbf{b}^{(t)})  \otimes \orient (\mathbf{b}\backslash \mathbf{b}^{(t)}) \subset 
(\kz F)^t (\mathbf{b}).
\]
\end{prop} 
 
\begin{proof}
Lemma \ref{lem:Cbb_action} together with 
Proposition \ref{prop:NcosimpF_b_action} imply that the morphism $N (\cosimp^\bullet  F)\op \twoheadrightarrow \kz F$ is a surjection of $\nat$-graded $\fb\op$-modules. 

It remains to prove that this surjection is compatible with the respective differentials. First observe that, in general, the kernel of the projection 
\[
(N (\cosimp^\bullet F)\op) ^* (\mathbf{b}) \twoheadrightarrow 
\overline{(N (\cosimp^\bullet F)\op) ^* (\mathbf{b})} 
\]
is {\em not} stable under the differential of $N (\cosimp^\bullet F)\op$. This is due to the fact that a coface map of $(\cosimp^\bullet  F)\op (\mathbf{b}) $  can split a set $\mathbf{b}_i^{(t)}$ with two elements $\{u, v\}$ into the 
ordered decompositions as singletons $(\{u\}, \{v \})$ and $(\{v\}, \{u\})$. However, these contributions will vanish on passing to $\kz F$, due to the relation $u \wedge v = - v \wedge u$. 
Using the previous observation, one verifies that the kernel (denoted here simply by $\ker$) of the morphism $(N (\cosimp^\bullet F)\op)^* (\mathbf{b}) \twoheadrightarrow (\kz F)^* (\mathbf{b})$ is stable under the differential. 

 The quotient differential on $ (N (\cosimp^\bullet F)\op)^* (\mathbf{b}) /\ker $ is induced by the coface operator $\tilde{d}^0$ on $(\cosimp^\bullet F)\op(\mathbf{b})$.

By inspection, this quotient complex is isomorphic to $(\kz F)^* (\mathbf{b})$ via the given morphism. 
\end{proof}

\subsection{The complex of ordered partitions}
\label{subsect:permtohedra}

The key to proving Theorem \ref{thm:norm_vs_kz}, i.e., that $N(\cosimp^\bullet F)\op  \twoheadrightarrow \kz F$  is a quasi-isomorphism,
 is the case $F= \mathbbm{1} \in \ob \fcatk[\finj\op]$, where $\mathbbm{1}$ is the following $\finj\op$-module:
 
 \begin{nota}
 \label{notation-1}
 Denote by $\mathbbm{1}$ the $\finj\op$-module given by  $\mathbbm{1}(\mathbf{0})= \kring$ and $\mathbbm{1}(\mathbf{b})=0$ for $b>0$, so that the only morphism that acts non-trivially on $\mathbbm{1}$ is the identity on $\mathbf{0}$. 
 \end{nota}

The aim of this section is to give an explicit description of  the complex $N (\cosimp^\bullet \mathbbm{1})\op$ in terms of the complex given by ordered partitions (see Proposition \ref{prop:compare_Cperm_Norm}) described below. 

Henceforth in this section, $X$ denotes a non-empty finite set.

\begin{nota}
For  $t \in \nat^*$, let 
$\prt_{t} X$ be the set of ordered partitions of $X$ into $t$ non-empty subsets. 
\end{nota}

Consider the refinement of ordered partitions as follows: 

\begin{defn}
\label{def-refinement}
Let $\mathfrak{p}
= (\mathfrak{p}_1, \ldots , \mathfrak{p}_t) 
\in \prt_t (X)$ and  $\mathfrak{q} = (\mathfrak{q}_1, \ldots , \mathfrak{q}_s) \in  \prt_s (X)$, then $\mathfrak{q}$ is an ordered refinement of $\mathfrak{p}$, denoted 
 $\mathfrak{q} \leq \mathfrak{p}$,  if there exists an order-preserving surjection $\alpha : \mathbf{s} \twoheadrightarrow \mathbf{t}$ such that 
\[
\mathfrak{p}_i = \bigcup_{j \in \alpha^{-1} (i)} \mathfrak{q}_j.
\]
(In particular, this requires that $s \geq t$,  with equality if and only if $\mathfrak{p} = \mathfrak{q}$.)
\end{defn}

The following is clear:

\begin{lem}
\label{lem:prt_poset}
The set $\prt_* (X):= \amalg _{t \in \nat^*} \prt_t (X)$ is a poset under refinement of ordered partitions. 
\end{lem}

\begin{exam}
\label{exam:P3}
For $X=\{1, 2, 3\}$ the poset $\prt_* (X)$ is given by:

{\scalebox{0.7}{
$
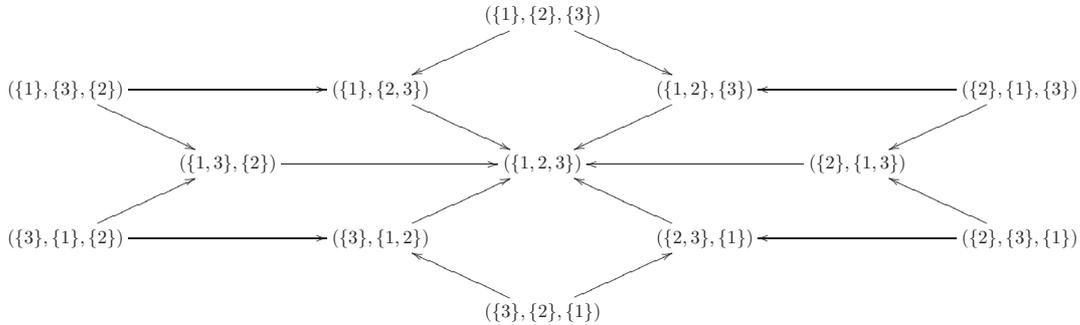
\xymatrix{
&&&(\{1\},\{2\},\{3\})\ar[dr]\ar[dl]&&&\\
(\{1\},\{3\},\{2\})\ar[rr]\ar[rd]&&(\{1\},\{2,3\})\ar[rd]&&(\{1,2\},\{3\})\ar[dl]&&(\{2\},\{1\},\{3\})\ar[ll]\ar[ld]\\
&(\{1,3\},\{2\})\ar[rr]&&(\{1,2,3\})&&(\{2\},\{1,3\})\ar[ll]\\
(\{3\},\{1\},\{2\})\ar[rr]\ar[ru]&&(\{3\},\{1,2\})\ar[ru]&&(\{2,3\},\{1\})\ar[lu]&&(\{2\},\{3\},\{1\})\ar[lu]\ar[ll]\\
&&&(\{3\},\{2\},\{1\}) \ar[lu] \ar[ru]&&&
}
$
}}
\end{exam}

The following Lemma highlights  properties of the poset $(\prt_* (X), \leq )$ that are exploited in the proof of Proposition \ref{prop:acyclicity}. 

\begin{lem}
\label{lem:permutohedron_facet_inclusion_codim2}
\ 
\begin{enumerate}
\item 
For $\mathfrak{p}=(\mathfrak{p}_1, \ldots , \mathfrak{p}_{|X|}),  
\mathfrak{p}'=(\mathfrak{p}'_1, \ldots , \mathfrak{p}'_{|X|}) \in \prt_{|X|}(X)$, there exists $\mathfrak{q} \in \prt_{|X|-1}(X)$ such that $\mathfrak{p}\leq \mathfrak{q}$  and $\mathfrak{p}'\leq \mathfrak{q}$ if and only if  there exists $ i\in \{1, \ldots, |X|-1\}$ such that $\mathfrak{p}_i=\mathfrak{p}'_{i+1}$, $\mathfrak{p}_{i+1}=\mathfrak{p}'_i$ and $\mathfrak{p}_j=\mathfrak{p}'_j$ for $j \not\in \{i,i+1\}$.
\item
Let $\mathfrak{p}=(\mathfrak{p}_1, \ldots , \mathfrak{p}_{t}) \in \prt_t(X)$, for $1\leq t \leq |X|-2$, and $\mathfrak{r} \in \prt_{t+2}(X)$ such that $\mathfrak{r} \leq \mathfrak{p}$. 
 Then 
 \[
 | \{ \mathfrak{q} \in \prt_{t+1} (X)| \mathfrak{r} \leq \mathfrak{q} \leq \mathfrak{p}
 \}
 | = 2.
 \]
\end{enumerate}

\end{lem}

\begin{proof}
For the forward implication of the  first point, 
suppose that there exists $\mathfrak{q} \in \prt_{|X|-1}(X)$ such that  $\mathfrak{p}\leq \mathfrak{q}$  and $\mathfrak{p}'\leq \mathfrak{q}$. By Definition \ref{def-refinement}, we have order-preserving surjections $\alpha: X\to |X|-1$ and $\alpha': X\to |X|-1$ such that 
\begin{equation}
\label{eqn-preuve-lm}
\mathfrak{q}_k = \coprod _{j \in \alpha^{-1} (k)} \mathfrak{p}_j \qquad \textrm{and} \qquad \mathfrak{q}_k = \coprod_{j \in \alpha'^{-1} (k)} \mathfrak{p}'_j.
\end{equation}
Since $\alpha$ and $\alpha'$ are order preserving surjections there is $i, j\in \{1, \ldots, |X|-1\}$ such that $\alpha^{-1}(i) = \{i , i+1 \}$ and $\alpha'^{-1}(j) = \{j , j+1 \}$ with all other fibres of cardinal one. We deduce from (\ref{eqn-preuve-lm}) that $i=j$ and $\mathfrak{q}_i = \mathfrak{p}_i \amalg \mathfrak{p}_{i+1}= \mathfrak{p}'_i \amalg \mathfrak{p}'_{i+1}$ so $\mathfrak{p}_i=\mathfrak{p}'_{i+1}$, $\mathfrak{p}_{i+1}=\mathfrak{p}'_i$ and $\mathfrak{p}_j=\mathfrak{p}'_j$ for $j \not\in \{i,i+1\}$.

For the converse, pick $\mathfrak{p}$ and $\mathfrak{p}'$ as in the statement. The following partition $\mathfrak{q} \in \prt_{|X|-1}(X)$ satisfies the inequalities of the statement:
$$\mathfrak{q}_j=\mathfrak{p}_j\ \mathrm{ for }\  1\leq j \leq i-1,\  \mathfrak{q}_i=\mathfrak{p}_i \amalg \mathfrak{p}_{i+1},\  \mathfrak{q}_j=\mathfrak{p}_{j-1}\ \mathrm{ for } \ i+1\leq j \leq |X|-1.$$

For the second point, by hypothesis, $\mathfrak{r} \leq \mathfrak{p}$. By Definition \ref{def-refinement}, this   corresponds to giving $\mathfrak{r}$ together with an order-preserving surjection $\alpha : \mathbf{t+2} \twoheadrightarrow \mathbf{t}$.
 To prove the result, it suffices to show that there are precisely two possible factorizations of $\alpha$  via order-preserving surjections:
\[
\mathbf{t+2} \stackrel{\alpha'}{\twoheadrightarrow} \mathbf{t+1} \stackrel{\alpha''}{\twoheadrightarrow} 
\mathbf{t}.
\]
The factorization condition implies that $\alpha''$ is uniquely determined by $\alpha'$. 

 There are two possibilities:
\begin{enumerate}
\item 
there exists $i \in \mathbf{t}$ such that $|\alpha^{-1}(i)|=3$, all other fibres having cardinal one; 
\item 
there exists $i<j \in \mathbf{t}$ such that  $|\alpha^{-1}(i)|=|\alpha^{-1}(j)|=2$, with all other fibres of cardinal one. 
\end{enumerate}

In the first case, one has either $\alpha'^{-1}(i) = \{i , i+1 \}$ or $\alpha'^{-1}(i+1) = \{i+1 , i+2 \}$, with all other fibres of cardinal one. In the second case, either $\alpha'^{-1}(i) = \{i , i+1 \}$ or $\alpha'^{-1}(j) = \{j , j+1 \}$, again with all other fibres of cardinal one. 
\end{proof}

\begin{defn}
\label{defn:cperm}
Let $\cperm (X)$ denote the cochain complex with 
\[
\cperm (X)^t := 
\left\{ 
\begin{array}{ll}
\kring [\prt_{t} X] & 1 \leq t \leq |X|\\
0 & \mbox{otherwise}.
\end{array}
\right.
\]

The differential $d : \cperm (X)^t \rightarrow \cperm (X)^{t+1}$ is the alternating sum $\sum_{i=1}^t (-1)^i \delta^i$, where $\delta^i$ sends a generator $[\mathfrak{p}]$ 
corresponding to the ordered partition $\mathfrak{p}= (\mathfrak{p}_1, \ldots, \mathfrak{p}_i, \ldots , \mathfrak{p}_{t} ) $ to the sum of ordered partitions of the form $(\mathfrak{p}_1, \ldots, \mathfrak{p}'_i, \mathfrak{p}''_i, \ldots , \mathfrak{p}_{t} )$ where $\mathfrak{p}_i = \mathfrak{p}'_i \amalg \mathfrak{p}''_i$. 
\end{defn}

\begin{rem}
\label{rem:diff_pm1_or_0}
The element $d[\mathfrak{p}]$, for $\mathfrak{p}\in \prt_t (X)$, is a signed sum (i.e., a linear combination with coefficients in $\{ \pm 1 \}$) of the 
generators $[\mathfrak{q}]$ where $\mathfrak{q} \in \prt_{t+1} (X)$ and $\mathfrak{q} \leq \mathfrak{p}$. 
\end{rem}

\begin{exam}
\label{exam:xperm}
\ 
\begin{enumerate}
\item 
$\cperm(\emptyset)= \kring$ in cohomological degree zero, by convention; 
\item 
$\cperm(\{1 \}) \cong \kring$ in cohomological degree one; 
\item 
$\cperm(\{1, 2\})$ is the complex $\kring \stackrel{\mathrm{diag}}{\rightarrow} \kring ^{\oplus 2}$ concentrated in cohomological degrees one and two; this has cohomology $\kring$ concentrated in degree $2$;
\item 
$\cperm (\{1,2,3 \})$ has the form $\kring \stackrel{\mathrm{diag}}{\rightarrow} \kring ^{\oplus 6} \rightarrow \kring ^{\oplus 6}$ in cohomological degrees $1, 2, 3$; this has cohomology $\kring$ concentrated in degree $3$. Here, the generators are given by the ordered partitions:
\[
\begin{tabular}{|l|l|l|}
\hline
\mbox{degree 1} &
\mbox{degree 2} &
\mbox{degree 3} 
\\
\hline
(\{1,2,3\}) & 
(\{1\},\{2,3\})&
(\{1\},\{2\},\{3\})
\\
&
(\{2,3\},\{1\})&
(\{2\},\{1\},\{3\})
 \\
&
(\{2\},\{1,3\})&
 (\{2\},\{3\},\{1\})
 \\
 &
(\{1,3\},\{2\})&
 (\{3\},\{2\},\{1\})
 \\
&
(\{3\},\{1,2\}) &
(\{3\},\{1\},\{2\})
\\
&
(\{1,2\},\{3\})&
(\{1\},\{3\},\{2\})
\\
\hline
\end{tabular}
\]
and, for instance, the differential on $(\{1 \}, \{2,3\})$ is given by the operator $\delta^2$ which sends the generator to the sum of 
$(\{1 \}, \{2\}, \{3\})$ and $(\{1\}, \{3\}, \{2\})$.
\end{enumerate}
\end{exam}

\begin{prop}
\label{prop:compare_Cperm_Norm}
For $X$ a non-empty finite set, there is an isomorphism of complexes
\[
(N (\cosimp^\bullet \mathbbm{1})\op) (X) 
\cong 
\cperm (X).
\]
\end{prop}

\begin{proof}
By Proposition \ref{prop:norm_cosimp_F} and Notation  \ref{notation-1}, since $\mathbbm{1}$ vanishes on non-empty sets,
 we have:
\[
(N (\cosimp^\bullet \mathbbm{1})\op)^t (X) 
= 
\bigoplus_{\substack{X= \mathbf{b}^{(t)} \amalg \coprod_{i=0}^{t -1} \mathbf{b}_i^{(t)}
\\
{\mathbf{b}^{(t)}_i \neq \emptyset}}}
\mathbbm{1} (\mathbf{b}^{(t)})
=
\bigoplus_{\substack{X= \coprod_{i=0}^{t -1} \mathbf{b}_i^{(t)}
\\
{\mathbf{b}^{(t)}_i \neq \emptyset}}}
\kring
\cong 
\kring [\prt_{t} X]
\]
and the coface $\tilde{d}^0: (\fbcmnd)^{\ell} \mathbbm{1} \rightarrow (\fbcmnd)^{\ell+1}  \mathbbm{1} $ is zero. So, using Proposition \ref{prop:norm_cosimp_F}, the differential of $(N (\cosimp^\bullet \mathbbm{1})\op)$ is given by $\sum_{i=1}^{t }(-1)^i  \tilde{d}^i$.

For $1\leq i\leq t$, comparing the definition of $\delta^i$ given in Definition \ref{defn:cperm} and the explicit description of $\tilde{d}^i: (\fbcmnd)^{\ell} \mathbbm{1} \rightarrow (\fbcmnd)^{\ell+1}  \mathbbm{1}$ given in Proposition \ref{prop:cosimplicial-structure-CF}, we obtain that $\tilde{d}^i$  corresponds, via the above isomorphism, to $\delta^i$.
\end{proof}

\subsection{Cohomology of the complex of ordered partitions}
Let $X$ be a non-empty finite set.
To calculate the cohomology of $\cperm (X)$, we relate it to a cellular complex as follows. The permutohedron $\Pi_X$ is an (abstract) $(|X|-1)$-dimensional polytope, with $k$-faces in bijection with $\prt_{|X|-k} X$ (see \cite[Example 0.10]{MR1311028}).  In particular, the vertices (i.e., $0$-faces) are indexed by elements of $\prt_{|X|} (X)$ (i.e.,  by permutations of the set $X$).

We record the necessary information on the face inclusions of $\Pi_X$:

\begin{lem}
\label{lem:permutohedron_face_inclusions}
Let $\mathfrak{p}=(\mathfrak{p}_1, \ldots , \mathfrak{p}_{t}),\mathfrak{q}=  (\mathfrak{q}_1, \ldots , \mathfrak{q}_{t+1})  \in \prt_* (X)$, for $1 \leq t < |X|$,  be  ordered partitions of the finite non-empty set $X$.  Then $\mathfrak{q} $ is a face of  $\mathfrak {p}$ if and only if $\mathfrak{q}\leq \mathfrak{p}$. 
\end{lem}

 Lemma \ref{lem:permutohedron_facet_inclusion_codim2} can be rephrased as follows:
 
 \begin{lem}
 \label{lem:permutohedron_facet_inclusion_codim2-reformulation}
 \ 
 \begin{enumerate}
\item  
\label{item:premutohedra_part1}
Two vertices of $\Pi_X$ are linked by a $1$-face if and only if they differ by a transposition of adjacent elements in the corresponding ordered lists of elements of $X$. 
\item 
\label{item:premutohedra_part2}
Let $\mathfrak{p}$ be a $n$-face of $\Pi_X$ with $2 \leq n\leq |X|-1$; a $(n-2)$-face $\mathfrak{r}$ of $\mathfrak{p}$ lies in precisely two $(n-1)$-faces of $\mathfrak{p}$. 
\end{enumerate}
   \end{lem}

The polytope $\Pi_X$ has a geometric realization $|\Pi_X|$ as the convex hull in $\mathbb{R}^{|X|}$ of the vectors with pairwise distinct coordinates from $\{1, 2, \ldots, |X| \}$; in particular, $|\Pi_X|$ is contractible. Moreover, by the above, the geometric polytope $|\Pi_X|$ is equipped with a cellular structure with the cells of dimension $k$ indexed by the elements of $\prt_{|X|-k} (X)$, for $0\leq k \leq |X|-1$. The cellular structure is understood by using Lemma \ref{lem:permutohedron_face_inclusions}.

\begin{exam}
\label{exam:P3-2}
The geometric realization of the  permutohedron $\Pi_{\{1, 2, 3\}}$ is a hexagon with vertices indexed by the elements of $\sym_3$:
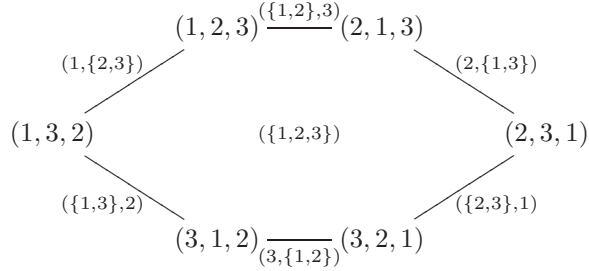
\[
\xymatrix{
&
(1,2,3)
\ar@{-}[r]^{(\{1,2\},3)}
\ar@{-}[ld]_{(1,\{2,3\})}
&
(2,1,3)
\ar@{-}[rd]^{(2,\{1,3\})}
\\
(1,3,2)
\ar@{-}[rd]_{(\{1,3\},2)}
&
\ar@{}[r]|{(\{1,2,3\})}
&&
(2,3,1)
\ar@{-}[ld]^{(\{2,3\},1)}
\\
&
(3,1,2)
\ar@{-}[r]_{(3,\{1,2\})}
&
(3,2,1)
}
\]
where, for notational clarity, the braces have been omitted from singletons.

In this case, the $0$-faces are the vertices, the $1$-faces are the edges and the $2$-face is the interior of the hexagon. Note that a $0$-face lies in precisely two $1$-faces.
\end{exam}

In the following Proposition, $\cperm (X)$ is considered with  {\em homological} degree and $[-|X|]$ denotes the shift of homological degree (i.e., $\cperm (X)[-|X|]_t=\cperm (X)_{t-|X|} $ by Notation \ref{nota:shift}).

\begin{prop}
\label{prop:acyclicity}
Let $X$ be a non-empty finite set, then there is an isomorphism of chain complexes
\[
\gamma_X :
\cperm (X)[-|X|]_*
\stackrel{\cong}{\rightarrow}
C^{\mathrm{cell}}_*(|\Pi_X|; \kring),
\]
where  $C^{\mathrm{cell}}_*(|\Pi_X|; \kring)$ is the cellular complex of $|\Pi_X|$ with coefficients in $\kring$.

In particular, $\cperm(X)$ has cohomology $\kring$  concentrated in degree $|X|$. 
\end{prop}

\begin{proof}
The case $|X|=1$ is clear, hence we suppose that $|X| \geq 2$. 

The  cellular complex $C^{\mathrm{cell}}_*(|\Pi_X|; \kring)$ has underlying graded $\kring$-module which is free on   $\prt_* (X)$, so that there is an isomorphism of $\kring$-modules
\begin{equation}
\label{preuve-Prop4.26}
\cperm (X)^{|X|-t} 
\cong 
C^{\mathrm{cell}}_t(|\Pi_X|; \kring).
\end{equation}

Moreover, for $[\mathfrak{p}]_{\mathrm{cell}}$ the generator corresponding to $\mathfrak{p} \in \prt_t (X)$, $d [\mathfrak{p}]_{\mathrm{cell}}$ is a signed sum of the 
generators $[\mathfrak{q}]_{\mathrm{cell}}$, where $\mathfrak{q} \in \prt_{t+1} (X)$ and $\mathfrak{q} \leq \mathfrak{p}$. (Here and in the following, the subscripts indicate in which  complex the generators live.) With the regrading used here, this is analogous to the behaviour observed in Remark \ref{rem:diff_pm1_or_0}; the only difference being in the signs which may occur.

In homological degree zero, the isomorphism (\ref{preuve-Prop4.26}) is given by 
\begin{eqnarray*}
\cperm (X) ^{|X|} &\rightarrow & C^{\mathrm{cell}}_0(|\Pi_X|; \kring) 
\\
\ [\sigma]_{\mathrm{perm}} & \mapsto & \sgn (\sigma)[\sigma]_{\mathrm{cell}},
\end{eqnarray*}
 using the identification of ordered partitions of $X$ into singletons with the symmetric group $\mathrm{Aut} (X)$ to define $\sgn(\sigma)$.

To prove the statement, we consider the cellular filtration of $|\Pi_X|$. More precisely, we prove, by induction on $n$, that we have an isomorphism of chain complexes
\[
\xymatrix{
\trunc_{\leq n}\cperm (X)[-|X|]_*
\ar[r]^{\gamma_X}_\cong &
\trunc_{\leq n}C^{\mathrm{cell}}_*(|\Pi_X|; \kring)
}
\]
such that, at each degree, $\gamma_X [\mathfrak{p}]_{\mathrm{perm}} = \pm [\mathfrak{p}]_{\mathrm{cell}}$ and
where $\trunc_{\leq n}C_*$ is the truncation of the complex $C_*$ defined by $(\trunc_{\leq n}C_*)_i=0 $ if $i>n$ and $(\trunc_{\leq n}C_*)_i=C_i $ if $i\leq n$.
 
To prove the initial step, note that the $1$-skeleton of $|\Pi_X|$ can be given the structure of an oriented graph, where each edge is oriented from the vertex with negative signature to that of positive signature; this relies crucially upon the fact that edges are indexed by   transpositions (see Lemma  \ref{lem:permutohedron_facet_inclusion_codim2-reformulation} (\ref{item:premutohedra_part1})). This serves to make explicit the signs appearing in the differential of the associated chain complex. 

It follows that the identity on $\prt_{|X|-1} (X)$  induces an isomorphism of chain complexes 
$\gamma_X: \trunc_{\leq 1}\cperm (X)[-|X|]_*
\to \trunc_{\leq 1}C^{\mathrm{cell}}_*(|\Pi_X|; \kring)
$:
\[
\xymatrix{
\cperm (X)^{|X|-1} 
\ar[d]_{\cong}^{[\mathfrak{p}]_{\mathrm{perm}} \mapsto [\mathfrak{p}]_{\mathrm{cell}}}
\ar[r]^d
&
\cperm (X)^{|X|}
\ar[d]_{\cong}^{[\sigma]_{\mathrm{perm}} \mapsto  \sgn (\sigma)[\sigma]_{\mathrm{cell}}} 
\\
C^{\mathrm{cell}}_1(|\Pi_X|; \kring) 
\ar[r]_d
&
C^{\mathrm{cell}}_0(|\Pi_X|; \kring) .
}
\] 

For $n\geq 2$ assume that $\gamma_X:  \trunc_{\leq n-1}\cperm (X)[-|X|]_*
\to \trunc_{\leq n-1}C^{\mathrm{cell}}_*(|\Pi_X|; \kring)
$ is defined and is an isomorphism. To prove the inductive step, choose $\mathfrak{p} \in \prt_{|X|-n} (X)$; as observed above, the differential of the generator  $[\mathfrak{p}]_{\mathrm{cell}} \in C^{\mathrm{cell}}_n(|\Pi_X|; \kring)$  is a signed sum of the generators $[\mathfrak{q}]_{\mathrm{cell}}$  corresponding to its $(n-1)$-faces:
\[
d  [\mathfrak{p}]_{\mathrm{cell}} = \sum_{\substack{\mathfrak{q} \in \prt_{|X|-n+1} (X) \\ \mathfrak{q}\leq \mathfrak{p}}} \eta_{\mathfrak{q}}  [\mathfrak{q}]_{\mathrm{cell}},
\]
 with $\eta_{\mathfrak{q}}\in \{\pm 1 \}.
$

Now, by the inductive hypothesis, we have an isomorphism $\gamma_X : \cperm (X)^{|X|-n+1} \stackrel{\cong}{\rightarrow} C_{n-1}^{\mathrm{cell}}(|\Pi_X|; \kring) $ that is compatible with the differential and which has the form $[\mathfrak{q}]_{\mathrm{perm}} \mapsto \pm [\mathfrak{q}] _{\mathrm{cell}}$.
 Hence, by Remark \ref{rem:diff_pm1_or_0}:
 
 \[
\gamma_Xd[\mathfrak{p}]_{\mathrm{perm}}= \sum_{\substack{\mathfrak{q} \in \prt_{|X|-n+1} (X) \\ \mathfrak{q}\leq \mathfrak{p}}} \alpha_{\mathfrak{q}}  [\mathfrak{q}]_{\mathrm{cell}},
\]
 with $\alpha_{\mathfrak{q}}\in \{\pm 1 \}.
$; moreover, $\gamma_Xd[\mathfrak{p}]_{\mathrm{perm}}$ is a cycle. 
 
We claim that $ \gamma_X d[\mathfrak{p}]_{\mathrm{perm}} =\epsilon _{\mathfrak{p}} d[\mathfrak{p}]_{\mathrm{cell}}$, for some  $\epsilon_{\mathfrak{p}} \in \{ \pm 1 \}$. If $ \gamma_X d[\mathfrak{p}]_{\mathrm{perm}} = d[\mathfrak{p}]_{\mathrm{cell}}$, take $\epsilon _{\mathfrak{p}}=1$. Otherwise, consider 
$$
 \gamma_X d[\mathfrak{p}]_{\mathrm{perm}} +  d[\mathfrak{p}]_{\mathrm{cell}}=\sum_{\substack{\mathfrak{q} \in \prt_{|X|-n+1} (X) \\ \mathfrak{q}\leq \mathfrak{p}}} \kappa_{\mathfrak{q}}[\mathfrak{q}]_{\mathrm{cell}},
 $$
 where $\kappa_{\mathfrak{q}} = \alpha_{\mathfrak{q}}+\eta_ {\mathfrak{q}}$.

By the hypothesis, there is some $\mathfrak{q} \in \prt_{|X|-n+1} (X)$ such that
 $\kappa_{\mathfrak{q}} =0$.
 
If $ \gamma_X d[\mathfrak{p}]_{\mathrm{perm}} +  d[\mathfrak{p}]_{\mathrm{cell}} \not =0$, there is some $\mathfrak{q}'$ such that $\kappa_{\mathfrak{q}'}\not =0$.
 Using  Lemma \ref{lem:permutohedron_facet_inclusion_codim2-reformulation} (\ref{item:premutohedra_part2}), we can find such a pair $(\mathfrak{q}, \mathfrak{q}')$ such that $\mathfrak{q}$ and $\mathfrak{q}'$ have a $(n-2)$-face in common, denoted by $\mathfrak{q} \cap \mathfrak{q}'$; moreover, $\mathfrak{q}, \mathfrak{q}'$ are the only  $(n-1)$-faces of which $\mathfrak{q} \cap \mathfrak{q}'$ is a face. We deduce that $d(\mathfrak{q} \cap \mathfrak{q}') \not = 0$: this contradicts the fact that $ \gamma_X d[\mathfrak{p}]_{\mathrm{perm}} +  d[\mathfrak{p}]_{\mathrm{cell}}$ is a cycle, therefore   $ \gamma_X d[\mathfrak{p}]_{\mathrm{perm}} +  d[\mathfrak{p}]_{\mathrm{cell}}=0$. Taking $\epsilon _{\mathfrak{p}}=-1$ in this case establishes the claim.

  By the uniqueness property of the cycle $d[\mathfrak{p}]_{\mathrm{cell}}$ explained above, it follows that $ \gamma_X d[\mathfrak{p}]_{\mathrm{perm}} =\epsilon _{\mathfrak{p}} d[\mathfrak{p}]_{\mathrm{cell}}$, where $\epsilon_{\mathfrak{p}} \in \{ \pm 1 \}$. Thus, an extension of $\gamma_X$ compatible with the differential and of the required form is given  by defining $\gamma_X [\mathfrak{p}]_{\mathrm{perm}}:= \epsilon _{\mathfrak{p}} [\mathfrak{p}]_{\mathrm{cell}}$. Applying this argument for all cells $\mathfrak{p}$ of dimension $n$ completes the inductive step.    

The isomorphism of chain complexes $\gamma_X$ induces an isomorphism in homology. Now  $|\Pi_X|$ is contractible, hence has homology $\kring$ concentrated in degree zero. Returning to the cohomological, unshifted grading, this implies that 
 $\cperm (X)$ has cohomology isomorphic to $\kring$ concentrated in cohomological degree $|X|$. 
\end{proof} 

\begin{cor}
\label{cor:cohom_cperm}
There is an isomorphism of cohomologically-graded $\fb\op$-modules 
\[
H^* (N (\cosimp^\bullet \mathbbm{1})\op) 
\cong 
\orient.
\]
\end{cor}

\begin{proof}
This follows from Propositions \ref{prop:compare_Cperm_Norm} and \ref{prop:acyclicity}, noting that the result holds evaluated on $\mathbf{0}$ by inspection. For $X$ a non-empty finite set,  the isomorphism in cohomology is induced by the morphism of $\kring$-modules 
$\cperm (X)^{|X|}
\rightarrow  \kring $ given by $[\sigma]  \mapsto  \sgn (\sigma)$.
\end{proof}

\subsection{Proof of Theorem \ref{thm:norm_vs_kz} }
\label{subsect:cohom_iso}

In this section we pass from the case $\mathbbm{1}$ treated in Corollary \ref{cor:cohom_cperm} to that of an arbitrary $\finj\op$-module $F$. 

The complex $(N (\cosimp^\bullet F)\op)^* (\mathbf{b}) $  has a finite filtration given as follows:

\begin{lem}
\label{lem:NC_filter}
Let $F$ be an $\finj\op$-module and $b \in\nat$. 
For $s \in \nat$, let $\filt^s (N (\cosimp^\bullet F)\op)^* (\mathbf{b}) \subset(N (\cosimp^\bullet F)\op)^* (\mathbf{b}) $ be the graded $\kring$-submodule
:
\[
\filt^s (N(\cosimp^\bullet F)\op) ^t (\mathbf{b}) = 
\bigoplus_{\substack{\mathbf{b} = \mathbf{b}^{(t)} \amalg \coprod_{i=0}^{t -1} \mathbf{b}_i^{(t)}
\\
{\mathbf{b}^{(t)}_i \neq \emptyset}
\\
|\mathbf{b}^{(t)} |\leq b-s 
}}
F (\mathbf{b}^{(t)}). 
\]

Then   $\filt^s (N (\cosimp^\bullet F)\op)^* (\mathbf{b}) $ is a subcomplex and  there is a finite filtration 
\[
0 
\subset 
\filt^b (N (\cosimp^\bullet F)\op)^* (\mathbf{b})
\subset 
\filt^{b-1} (N (\cosimp^\bullet F)\op)^* (\mathbf{b})
\subset \ldots
 \subset
\filt^0 (N (\cosimp^\bullet F)\op)^* (\mathbf{b})
=(N (\cosimp^\bullet F)\op)^* (\mathbf{b}).
\]

Moreover, the quotient complex 
$\filt^s (N (\cosimp^\bullet F)\op)^* (\mathbf{b})/
\filt^{s+1} (N (\cosimp^\bullet F)\op)^* (\mathbf{b})$
 decomposes as the direct sum of complexes
\[
\bigoplus 
_{\substack{
\mathbf{b}' \subset \mathbf{b} \\
|\mathbf{b}'|=b-s}}
F(\mathbf{b}') 
\otimes
(N(\cosimp^\bullet \mathbbm{1})\op)^*  (\mathbf{b}\backslash \mathbf{b}').
\]
\end{lem}

\begin{proof}
From the construction of $\cosimp^\bullet F$, it is clear that $\filt^s (N (\cosimp^\bullet F)\op)^* (\mathbf{b}) $ is stable under the differential and that one has a finite filtration as given. 

The subquotient $\filt^s /\filt^{s+1}$ only has contributions from terms  with $|\mathbf{b}' | = b-s$.  Moreover, by construction, the differential behaves as though all non-isomorphisms of $\finj\op$ act as zero. By inspection, the subquotient identifies as stated.
\end{proof}

\begin{proof}[Proof of Theorem \ref{thm:norm_vs_kz}]
The natural surjection $
N (\cosimp^\bullet F)\op
\twoheadrightarrow 
\kz F $ of cochain complexes in $\fcatk[\fb\op]$  is given by Proposition \ref{prop:kz_quotient}.
It suffices to show that $(N (\cosimp ^\bullet F)\op)^* (\mathbf{b}) 
\twoheadrightarrow 
(\kz F)^*(\mathbf{b})$
is a quasi-isomorphism, for each $b \in \nat$.
This follows from the spectral sequence associated to the filtration $\filt^\bullet (N (\cosimp^\bullet  F)\op)^* (\mathbf{b})$ that calculates the cohomology of $(N(\cosimp^\bullet F)\op)^*(\mathbf{b})$.  By Lemma \ref{lem:NC_filter}, the $E_0$-page is given by
$$E_0^{s,t}=\filt^s (N (\cosimp^\bullet F)\op)^{s+t} (\mathbf{b})/
\filt^{s+1} (N (\cosimp^\bullet F)\op)^{s+t} (\mathbf{b})
\cong
 \bigoplus 
_{\substack{
\mathbf{b}' \subset \mathbf{b} \\
|\mathbf{b}'|=b-s}}
F(\mathbf{b}') 
\otimes
(N(\cosimp^\bullet \mathbbm{1})\op)^{s+t}  (\mathbf{b}\backslash \mathbf{b}').$$
The $E_1$-page is calculated by using Corollary \ref{cor:cohom_cperm} to identify the cohomology of each complex $(N(\cosimp^\bullet \mathbbm{1})\op)^*  (\mathbf{b}\backslash \mathbf{b}')$ and by applying universal coefficients. In particular, for $|\mathbf{b}'|=b-s$ the cohomology of $(N(\cosimp^\bullet \mathbbm{1})\op)^* (\mathbf{b}\backslash \mathbf{b}')$ is concentrated in degree $s$. 

It follows that the $E_1$-page is concentrated in the line $t=0$. This line, equipped with the differential $d_1$, identifies with the complex $(\kz F)^*(\mathbf{b})$. The spectral sequence degenerates at the $E^2$-page, since there is no space for differentials. The result follows.  
\end{proof}

\subsection{Relating to $\finj\op$-cohomology}
\label{FIop-cohom}
In \cite{MR3654111,MR3603074}, the authors define the $\finj$-homology of a $\finj$-module $F$ as the left derived functors of a right exact functor $H_0^{\finj}$. The functor $H_0^{\finj}$ can be defined as the left adjoint of the functor $\fcatk[\fb] \rightarrow \fcatk[\finj]$ given by extension by zero on morphisms (see Definition \ref{defn:H0_finjop} for a precise definition in the dual case). 

In this section, we consider the dual situation to define the $\finj\op$-cohomology of a $\finj\op$-module $F$. This allows us to rephrase Theorem \ref{thm:norm_vs_kz}  in Corollary \ref{cor:kos_FIop_cohom}, using $\finj\op$-cohomology. Theorem \ref{thm:cohomology_NC}  can also be reinterpreted as the computation of the $\finj\op$-cohomology of the $\finj\op$-module  $\kring \hom_\Oz (-, \mathbf{a})$ considered in Proposition \ref{prop:hom_oz_functor_finjop} (see Theorem \ref{FI-op-cohom}).

By Theorem \ref{thm:barr_beck_finjop}, a $\finj\op$-module $F$ can be considered equivalently as a $\fbcmnd$-comodule in $\fcatk[\fb\op]$, with structure morphism $\psi_F : F \rightarrow \fbcmnd F$. Moreover, by Proposition \ref{prop:fbcmnd_coaug}, one has the natural coaugmentation $\eta_F : F \rightarrow \fbcmnd F$. 

The following constructions should be compared with the general framework that is given in Appendix \ref{sect:homextseq}, working with the category of $\finj\op$-modules, which is abelian and has enough injectives (see below for the latter).  

\begin{defn}
\label{defn:H0_finjop}
(Cf. Example \ref{exam:perp_primitives}.)
For $F \in \ob \fcatk[\finj\op]$, let $H^0_{\finj\op}(F) \in \ob \fcatk[\fb\op]$ be the equalizer of the diagram
\[
\xymatrix{
F 
\ar@<.5ex>[r]^{\eta_F} 
\ar@<-.5ex>[r]_{\psi_F}
&
\fbcmnd F,
}
\]
which defines a functor $H^0_{\finj\op} : \fcatk[\finj\op] \rightarrow \fcatk[\fb\op]$.
\end{defn}

\begin{defn}
\label{defn:zext}
Let $\zext : \fcatk[\fb\op] \rightarrow \fcatk[\finj\op]$ be the exact functor given by extension by zero on morphisms (i.e.,  an $\fb\op$-module is considered as an $\finj\op$-module by specifying that non-isomorphisms of $\finj\op$ act as zero).
\end{defn}

\begin{prop}
\label{prop:H0_left_exact}
The functor $H^0_{\finj\op}  : \fcatk[\finj\op] \rightarrow \fcatk[\fb\op]$ is right adjoint to $\zext : \fcatk[\fb\op] \rightarrow \fcatk[\finj\op]$. In particular, it is left exact.
\end{prop}

\begin{proof}
One checks that, for $F$ an $\finj\op$-module,  $H^0_{\finj\op}F$ is the largest sub $\finj\op$-module that lies in the image of $\zext$,  hence $H^0_{\finj\op}$ is the right adjoint to $\zext$. This implies, in particular, that $H^0_{\finj\op}$ is left exact. 
\end{proof}


It is a standard fact that the category $\fcatk[\finj\op]$ has enough injectives. (This follows from Yoneda's Lemma, which shows that, for $I$ an injective $\kring$-module and $\mathbf{n}\in \ob\finj$, the $\finj\op$-module $\mathrm{Map}({\hom_{\finj} (\mathbf{n}, -)}, I)$ represents the functor $F \mapsto \hom_\kring (F(\mathbf{n}) , I)$.) 

Thus one can define $\finj\op$-cohomology as follows:

\begin{defn}
\label{defn:fiop_cohomo}
For $s\in \nat$, 
\begin{enumerate}
\item 
let $H^s_{\finj \op}$ denote the $s$-th right derived functor of $H^0_{\finj\op}$; 
\item 
$H^s_{\finj\op} (F)\in \fcatk[\fb\op]$ is the $s$th $\finj\op$-cohomology of the $\finj\op$-module $F$.
\end{enumerate}
\end{defn}


Proposition \ref{prop:fbcmnd_coaug} shows that the coaugmentation $\eta$ of $\fbcmnd$ satisfies Hypothesis \ref{hyp:coaugmentation}, hence one can form the cosimplicial object $\cosimp^\bullet F$, as in Proposition \ref{prop:cosimp_comod}, with respect to the structure  $(\fbcmnd, \eta)$ on $\fcatk[\fb\op]$.

\begin{prop}
\label{prop:FIop-cohom}
For $F \in \ob \fcatk[\finj\op]$, there is a natural isomorphism
\[
 H^*_{\finj\op} (F)\cong \pi^* (\cosimp^\bullet F),  
\]
where the right hand side is the cohomotopy of the cosimplicial object $\cosimp^\bullet F $ of $\fcatk[\fb\op]$. 
\end{prop}

\begin{proof}
By definition, $H^0 _{\finj\op}(F) = \pi^0 (\cosimp^\bullet F)$. Moreover, Proposition \ref{prop:cosimp_comod} implies that $F \mapsto \cosimp^\bullet F$ defines an exact functor to cosimplicial objects, so $F \mapsto \pi^* (\cosimp^\bullet F)$ forms a cohomological $\delta$-functor. The usual argument of homological algebra (cf. \cite[Exercise 2.4.5]{MR1269324}, for example), using the effacability property established in  Corollary \ref{cor:delta_perp}, then provides the isomorphism. 
\end{proof}

This allows Theorem \ref{thm:norm_vs_kz} to be restated using $\finj\op$-cohomology:

\begin{cor}
\label{cor:kos_FIop_cohom}
For $F \in \ob \fcatk[\finj\op]$, there is a natural isomorphism
\[
H^*_{\finj\op} (F) \cong H^* (\kz F).
\]
\end{cor}

\begin{proof}
The normalized complex $N (\cosimp^\bullet F)\op $ is canonically isomorphic to $N \cosimp ^\bullet F$, in particular has canonically isomorphic cohomology. Thus the statement follows immediately from Theorem \ref{thm:norm_vs_kz}.
\end{proof}

\begin{rem}
This result is an $\finj\op$-cohomology version of  \cite[Theorem 2]{MR3603074}, which concerns $\finj$-homology.
\end{rem}

\section{Coaugmented cosimplicial objects}
\label{sect:new}
 
The main objective of this paper is to calculate  $\hom_{\fcatk[\finne]} (\kbar^{\otimes a},\kbar^{\otimes b}) $ and, more generally, $\ext^i_{\fcatk[\finne]} (\kbar^{\otimes a},\kbar^{\otimes b}) $, for $a, b \in \nat$  and $ i \in \nat$. As a step in this direction, Theorem \ref{thm:motivate} establishes that, for $i\in \{0,1\}$, these $\ext$-groups are given by the cohomology of the cosimplicial abelian group $\hom_{\fcatk[\Gamma]}  ((t^*)^{\otimes a}, \gcmnd^\bullet (t^*)^{\otimes b} )$ introduced in Proposition \ref{Prop:augmented}. 
 
 In this section we focus principally upon understanding this cosimplicial object.  We give an alternative, highly explicit description of this by using the Dold-Kan equivalence of Theorem \ref{thm:DK_equiv}. More precisely, as $\kring \hom_\Oz (-, \mathbf{a})$ is an $\finj\op$-module (see Proposition \ref{prop:hom_oz_functor_finjop}), we can consider the cosimplicial $\fb\op$-module $\cosimp^\bullet \kring \hom_\Oz (-, \mathbf{a})$ described in (\ref{cosimplicial-obj-1}). The main result of this section is the following:

 \begin{thm}
\label{thm:identify_cosimp_gcmnd_fbcmnd}
For $a \in \nat$, there is an isomorphism of cosimplicial objects in $\fcatk[\fb\op]$:
\[
\hom_{\fcatk[\Gamma]}  ((t^*)^{\otimes a}, \gcmnd^\bullet \twist )
\cong 
[\cosimp^\bullet \kring \hom_\Oz (-, \mathbf{a})]\op
\]
where $-\op$ is the functor introduced in Notation \ref{nota:opp_simp} and 
$\twist$ is the functor in $\fcatk[\Gamma \times \fb\op]$ given by $\twist (\mathbf{b}) := (t^*)^{\otimes b}$ (cf. Definition \ref{defn:twist}).
\end{thm} 
 
Using Theorem \ref{thm:norm_vs_kz} and Corollary \ref{cor:kos_FIop_cohom}, this has the following immediate consequence:

\begin{cor}
\label{cor:relate_finjop_cohom}
For $a \in \nat$, there is an isomorphism of graded $\fb\op$-modules:
\[
H^* ( \hom_{\fcatk[\Gamma]}  ((t^*)^{\otimes a}, \gcmnd^\bullet \twist))
\cong 
H^* (\kz \kring \hom_\Oz (-, \mathbf{a}))
\cong
H^*_{\finj\op} ( \kring \hom_\Oz (-, \mathbf{a}) ) .
\]
\end{cor}

Combining Corollary \ref{cor:relate_finjop_cohom},  Theorem \ref{thm:motivate} and the computation of the cohomology of $\kz \kring \hom_\Oz (-, \mathbf{a})$ that is carried out in Section \ref{sect:cohomology}, leads to the computation of $\hom_{\fcatk[\finne]} (\kbar^{\otimes a},\kbar^{\otimes b})$ and $\ext^1_{\fcatk[\finne]} (\kbar^{\otimes a},\kbar^{\otimes b})$ given in Theorem \ref{thm:complex_Cab}. 

 \subsection{The cosimplicial abelian group $\hom_{\fcatk[\Gamma]}  ((t^*)^{\otimes a}, \gcmnd^\bullet (t^*)^{\otimes b} )$}
 In this section we apply the general constructions given in Appendix \ref{sect:homextseq} for $\mathcal{C}= \fcatk[\Gamma]$ and  $\gcmnd : \fcatk[\Gamma] \rightarrow \fcatk[\Gamma]$ the comonad introduced in Notation \ref{nota:gcmnd}.

By Proposition \ref{prop:stub_cosimp}, for $a, b \in \nat$,  there is an equalizer: 
\begin{equation}
\label{equalizer}
\xymatrix{
\hom_{\fcatk[\Gamma]_{\gcmnd}}  ((t^*)^{\otimes a},(t^*)^{\otimes b} )\ar[r]
&
\hom_{\fcatk[\Gamma]}  ((t^*)^{\otimes a},(t^*)^{\otimes b} )
 \ar@<.5ex>[r]^-{d^0}
 \ar@<-.5ex>[r]_-{d^1}
 &
\hom_{\fcatk[\Gamma]}  ((t^*)^{\otimes a}, \gcmnd (t^*)^{\otimes b} )
}
\end{equation}
where, for $f \in \hom_{\fcatk[\Gamma]}  ((t^*)^{\otimes a},(t^*)^{\otimes b}) $, $d^0 f := (\perp^\Gamma f) \circ \psi_{(t^*)^{\otimes a}} $, $d^1 f := \psi_{ (t^*)^{\otimes b}} \circ f$ and $\fcatk[\Gamma]_{\gcmnd}$ is the category of $\gcmnd$-comodules in $\fcatk[\Gamma]$.

 The comonad $\gcmnd$ provides  the usual augmented simplicial object ${\gcmnd}^{\bullet +1} (t^*)^{\otimes b}$  (cf. Proposition \ref{prop:simplicial}) and hence an augmented simplicial object $
\hom_{\fcatk[\Gamma]}  ((t^*)^{\otimes a}, \gcmnd^{\bullet +1} (t^*)^{\otimes b} )$ (cf. Lemma \ref{lem:hom_aug_simp}).

\begin{prop}
\label{Prop:augmented}
The augmented simplicial object $ 
\hom_{\fcatk[\Gamma]}  ((t^*)^{\otimes a}, \gcmnd^{\bullet +1} (t^*)^{\otimes b} )$ given by Lemma \ref{lem:hom_aug_simp} extends to a cosimplicial abelian group 
\[
\hom_{\fcatk[\Gamma]}  ((t^*)^{\otimes a}, \gcmnd^\bullet (t^*)^{\otimes b} ).
\]
 by using the structure morphisms $d^0$ and $d^1$ and their generalizations. 
\end{prop}

\begin{proof}
This is a direct application of Proposition \ref{prop:coaugmented_cosimp}. It is a cosimplicial abelian group because $\fcatk[\Gamma]$ is abelian and $\gcmnd$ is additive. 
\end{proof}

The interest of  this cosimplicial object is established by the following:

\begin{thm}
\label{thm:motivate}
For $a, b \in \nat$   there are isomorphisms:
\begin{eqnarray*}
\hom_{\fcatk[\finne]} (\kbar^{\otimes a},\kbar^{\otimes b}) 
&\cong & H^0 (\hom_{\fcatk[\Gamma]}  ((t^*)^{\otimes a}, \gcmnd^\bullet (t^*)^{\otimes b} )) 
\\
\ext^1_{\fcatk[\finne]} (\kbar^{\otimes a},\kbar^{\otimes b})
&
\cong 
& 
H^1 (\hom_{\fcatk[\Gamma]}  ((t^*)^{\otimes a}, \gcmnd^\bullet (t^*)^{\otimes b} )). 
\end{eqnarray*}
\end{thm}  

\begin{proof}
The identification of the cohomology groups $H^i$, for $i=0,1$, is given by Theorem \ref{thm:extbar}, using the Barr-Beck theorem (Theorem \ref{thm:barr_beck} and Remark \ref{rem:barr_beck}) to translate from $\gcmnd$-comodules to $\fcatk[\finne]$. The fact that the functor $(t^*)^{\otimes a}$ is projective in $\fcatk[\Gamma]$ (by Proposition \ref{prop:proj_gen_Gamma}) implies that 
\[
\extbar^1 _{\fcatk[\finne]}(\kbar^{\otimes a}, \kbar^{\otimes b} ) 
=
\ext^1 _{\fcatk[\finne]}(\kbar^{\otimes a}, \kbar^{\otimes b} ), 
\]
using the notation of Theorem \ref{thm:extbar}.
\end{proof}

 \begin{rem}
 \label{Rem-noyau}
By the Barr-Beck theorem (Theorem \ref{thm:barr_beck} and Remark \ref{rem:barr_beck}) and Proposition \ref{prop:tstar_kbar}, for $a, b \in \nat$, we have:
$$\hom_{\fcatk[\finne]} (\kbar^{\otimes a},\kbar^{\otimes b})  
\cong
\hom_{\fcatk[\Gamma]_{\gcmnd}}  ((t^*)^{\otimes a},(t^*)^{\otimes b} )$$
 so, by  diagram (\ref{equalizer}), the calculation of $\hom_{\fcatk[\finne]} (\kbar^{\otimes a},\kbar^{\otimes b}) $ reduces to the  calculation of the kernel of $d^0-d^1$. Unfortunately, this kernel is inaccessible to direct computation. 
 
 The above isomorphism shows that $\hom_{\fcatk[\finne]} (\kbar^{\otimes a},\kbar^{\otimes b})$ can be calculated using  morphisms of $\gcmnd$-comodules in $\fcatk[\Gamma]$. This motivates the introduction of the simplicial object $ \hom_{\fcatk[\Gamma]}  ((t^*)^{\otimes a}, \gcmnd^{\bullet +1} (t^*)^{\otimes b} )$.
\end{rem}

\subsection{Structure morphisms for the simplicial object $(\gcmnd)^{\bullet +1} (t^*)^{\otimes b}$}

In this section we give the explicit description of the structure of the augmented simplicial object 
 $
 (\gcmnd)^{\bullet +1} (t^*)^{\otimes b}
 $ obtained from Proposition \ref{prop:simplicial} applied to the comonad $\gcmnd$. 
 
We begin by the cases $b=0$ and $b=1$, exploiting the following basic result, which follows from Theorem \ref{thm:DK_equiv}: 

\begin{lem}
\label{lem:morphisms_tstar}
The identity induces canonical isomorphisms $ 
\hom_{\fcatk[\Gamma]} (\kring, \kring) \cong \kring
$
and 
$ 
\hom_{\fcatk[\Gamma]} (t^*, t^*) \cong \kring.
$  Moreover, 
 $\hom _{\fcatk[\Gamma]} (t^*, \kring) = 0 = \hom _{\fcatk[\Gamma]} (\kring, t^*)$.
\end{lem}

Recall from Section \ref{section:Comparison-FGamma-FFin} the identifications  $\theta^* \overline{\kring} \cong \kring$ and $\theta^* P^\finne_{\mathbf{1}} \cong P^\Gamma_{\mathbf{1}_+}$, together with 
 $\theta^* \kbar \cong t^*$. 
In particular, since $ \kring$ and 
$t^*$ lie in the image of $\theta^*$, they are $\gcmnd$-comodules (see Remark \ref{rem:barr_beck}).

For $b=0$, $(t^*)^{\otimes 0}=\kring$ and $\gcmnd \kring=\kring$. The structure morphisms $\epsilon^\Gamma_{\kring }: \gcmnd \kring \to \kring$ and $\psi: \kring \to \gcmnd \kring$ are the respective identity natural transformations.

For $b=1$, the structure morphisms $\epsilon^\Gamma_{t^*}$ and $\psi$ are identified by the following:

\begin{lem}
\label{lem:gcmnd_t*}
There are isomorphisms
\begin{eqnarray*}
(-)_+^* t^* & \cong & P^\finne _{\mathbf{1}} \\
\gcmnd t^* & \cong & P^\Gamma _{\mathbf{1}_+} \cong t^* \oplus \kring.
\end{eqnarray*}

The $\gcmnd$-comodule structure morphism $\psi : t^* \rightarrow \gcmnd t^* \cong  P^\Gamma _{\mathbf{1}_+} \cong t^* \oplus \kring$ is the canonical inclusion and the counit 
$\epsilon^\Gamma_{t^*} : \gcmnd t^* \rightarrow t^*$ is the canonical projection.

Explicitly, for $(Z,z)$ a finite pointed set, $t^* (Z)$ has generators $[y]-[z]$, for $y \in  Z \backslash \{z\}$ and $\gcmnd t^*  (Z)$ has generators $[y]-[z]$, for $y \in  Z \backslash \{z\}$, and $[z]-[+]$. The morphism $\psi$  sends $[y]-[z]$ to $[y]-[z]$ and $\epsilon^\Gamma_{t^*}$ sends 
$[y]-[z]$ to $[y]-[z]$ and $[z]-[+]$ to $0$.
\end{lem}

\begin{proof}
As in  Proposition \ref{prop:tstar_kbar}, we  consider $t^*$ as a sub-functor of $P^\Gamma_{\mathbf{1}_+}$. Since $(-)^*_+$ is exact (by Proposition \ref{prop:adjoints_theta}), for $X \in \ob \finne$ we have:  
$$
(-)^*_+ t^* (X)= t^* (X_+) \subset (-)^*_+P^\Gamma_{\mathbf{1}_+}(X)=P^\Gamma_{\mathbf{1}_+}(X_+) \cong \kring [X_+]
$$
 and $t^* (X_+)$ identifies as the submodule generated by the elements $[x]-[+]$, for $x \in X$, by the explicit description given in  Lemma \ref{lem:t*_projective}. By the Yoneda Lemma, there is a morphism $P^\finne_{\mathbf{1}} \rightarrow (-)^*_+ t^*$ determined by the element $[1]- [+] \in t^* (\mathbf{1}_+)$. 
 This identifies as the morphism $P^\finne_{\mathbf{1}} (X) \cong \kring[X] \rightarrow t^* (X_+) \subset \kring [X_+]$ that sends a generator $[x]$ to $[x] - [+]$, for $x\in X$;  this gives the  isomorphism $(-)_+^* t^*  \cong  P^\finne _{\mathbf{1}}$. The identification of $\gcmnd t^*$ then follows from the isomorphism $\theta^* P^\finne_{\mathbf{1}} \cong P^\Gamma_{\mathbf{1}_+}$ given by Proposition \ref{prop:adjoints_theta}; the second isomorphism is given by Lemma \ref{lem:t*_projective}.

By Lemma \ref{lem:morphisms_tstar}, since the composite $\epsilon_{\gcmnd} \psi$ is the identity on $t^*$, to identify these structure morphisms it suffices to show that, under the isomorphism $\gcmnd t^* \cong t^* \oplus \kring$, $\epsilon_{\gcmnd}$ is the canonical projection. This follows from the description of $\epsilon_{\gcmnd}$ given in Proposition \ref{prop:comonad-gamma-explicit}, as seen explicitly as follows.

By Lemma \ref{lem:t*_projective}, $t^* (Z) \subset \kring [Z]$ has generators $[y]-[z]$, for $y \in  Z \backslash \{z\}$, thus $t^* (Z_+)$ has generators: $[y]-[+]$, $y $ as above, and $[z]-[+]$. Writing $([y]-[z]) + ([z] -[+]) = [y]-[+]$, one also has that $t^* (Z_+)$ has  generators $[y]- [z]$ and $[z]-[+]$,  corresponding to generators of $t^* (Z) \subset t^* (Z_+)$ and $\kring$ respectively. The element $[z]-[+]$ generates the kernel of $\epsilon_{\gcmnd}$ and $[y]-[z]$ is sent by $\epsilon_{\gcmnd}$ to $[y]-[z] \in t^* (Z)$, corresponding to the projection to $t^*$.
\end{proof}

Consider the beginning of the augmented simplicial object $(\gcmnd)^{\bullet +1} t^*$:
\[
\xymatrix{
\gcmnd \gcmnd t^* 
\ar@<.75ex>[rr]^{\delta_0 = \epsilon_{\gcmnd t^*} }
\ar@<-.75ex>[rr]_{\delta_1 = \gcmnd \epsilon_{t^*}}
&&
\gcmnd t^* 
\ar[rr]^{\delta_0 = \epsilon_{t^*}}
\ar@<-2ex>@/_1pc/[ll]_{\sigma_0 = \Delta}
&&
t^*. 
}
\] 

The structure morphisms are identified by the following, in which we implicitly use Lemma \ref{lem:morphisms_tstar}: 

\begin{prop}
\label{prop:identify_Delta}
\ 
\begin{enumerate}
\item 
There is an isomorphism 
\[
\gcmnd \gcmnd t^* 
\cong 
t^* \oplus \kring_0 \oplus \kring_1,
\]
where $\kring_i$ is the rank one free $\kring$-module  given by $\kring_i := \ker (\delta_i)$ for $i \in \{0, 1 \}$.
\item 
The diagonal $\Delta : \gcmnd t^* \cong t^* \oplus \kring \rightarrow \gcmnd \gcmnd t^* \cong t^* \oplus \kring_0 \oplus \kring_1$ is given by  the canonical inclusion of $t^*$ and the diagonal map 
$
\kring \rightarrow \kring_0 \oplus \kring_1.
$
\item 
The morphism $\gcmnd \psi : \gcmnd t^* \cong t^* \oplus \kring 
\rightarrow \gcmnd \gcmnd t^* \cong t^* \oplus \kring_0 \oplus \kring_1$ is the  inclusion of the direct summand $t^* \oplus \kring_0$. 
\end{enumerate}
\end{prop} 
 
 \begin{proof}
 The isomorphism $\gcmnd t^* \cong t^* \oplus \kring$ of  Lemma \ref{lem:gcmnd_t*} induces an isomorphism  $\gcmnd \gcmnd t^*  \cong t^* \oplus \kring ^{\oplus 2}$. The operators $\delta_0$ and $\delta_1$ give a canonical decomposition, as follows. For $(Z,z)$ a finite pointed set, 
 by Lemma \ref{lem:gcmnd_t*}, $t^* (Z)$ has generators $[y]-[z]$, for $y \in  Z \backslash \{z\}$ and $\gcmnd t^*  (Z)$ has generators $[y]-[z]$, for $y \in  Z \backslash \{z\}$, and $[z]-[+]$. By Proposition \ref{prop:comonad-gamma-explicit}, we have:
  \[
\gcmnd_1 \gcmnd_2  t^* (Z) 
= 
 t^* ((Z_{+_1}) _{+_2}), 
\]
so that $\gcmnd_1 \gcmnd_2  t^* (Z) $ is generated by  $[y]-[+_2]$, for $y \in  Z \backslash \{z\}$, $[y]-[+_2]$, and $[+_1]-[+_2]$. Using the equalities
$$[y]-[+_2]=([y]-[z])+([z]-[+_2]); \qquad [z]-[+_2]=([z]-[+_1])+([+_1]-[+_2])$$
we obtain that $\gcmnd_1 \gcmnd_2  t^* (Z) $ is generated by:
\[
\big\{ ([y]-[z]), ([z] -[+_1]) , ([+_1] -[+_2] ) \  | \ y \in Y\backslash \{ z \} \ \big\}.
\]  
By the description of the counits given in Proposition \ref{prop:comonad-gamma-explicit}, $([z] -[+_1])$ generates the kernel of $\epsilon_{\gcmnd t^*}$ (i.e., $\kring_0$) and $([+_1] -[+_2] )$ generates the kernel of $\gcmnd \epsilon_{t^*}$ (i.e., $\kring_1$).  

By Proposition \ref{prop:comonad-gamma-explicit}, the diagonal $\Delta: \gcmnd t^* \rightarrow \gcmnd_1 \gcmnd_2 t^*$ sends $([y]-[z])$ to $([y]-[z])$ and $([z] - [+])$ to $([z] -[+_2]) =  ([z] -[+_1]) + ([+_1] -[+_2] )$; this identifies $\Delta$. (An alternative argument is to use the fact that the composites $\epsilon_{\gcmnd} \Delta$ and $(\gcmnd \epsilon) \Delta$ are both the identity on $\gcmnd t^*$.)
 
The identification of $\gcmnd \psi$ is similar, by using Lemma \ref{lem:gcmnd_t*}.
\end{proof}
 
For the case $b>1$, we use the fact that the functor $\gcmnd$ is symmetric monoidal, by Corollary \ref{cor:perp_sym}, to obtain the following:

\begin{prop}
\label{prop:calculate_perp}
Let $b\in \nat$.
\begin{enumerate}
\item 
The canonical isomorphism $\gcmnd  t^* \cong t^* \oplus \kring$ induces isomorphisms:
\[
\gcmnd  (t^*)^{\otimes b} 
\cong 
P^\Gamma _{\mathbf{b}_+}
\cong 
\bigoplus_{\mathbf{b}' \subseteq \mathbf{b}} (t^*)^{\otimes b'}
; \qquad
\gcmnd \gcmnd  (t^*)^{\otimes b} 
\cong 
\bigoplus_{\mathbf{b}'' \subseteq \mathbf{b}' \subseteq \mathbf{b}} (t^*)^{\otimes b''}
\]
where $b' = | \mathbf{b}'|$ and $b'' = | \mathbf{b}''|$. 
\item 
\label{prop:calculate_perp-2}
The comodule structure morphism $\psi_{(t^*)^{\otimes b}}:  (t^*)^{\otimes b} \to \gcmnd  (t^*)^{\otimes b} \cong 
\bigoplus_{\mathbf{b}' \subseteq \mathbf{b}} (t^*)^{\otimes b'}$ is the inclusion of the direct summand indexed by $\mathbf{b}'= \mathbf{b}$. 
\item
\label{prop:calculate_perp3}
The counit $\epsilon^\Gamma_{(t^*)^{\otimes b} } : \gcmnd (t^*)^{\otimes b} \cong \bigoplus_{\mathbf{b}' \subseteq \mathbf{b}} (t^*)^{\otimes b'} \rightarrow (t^*)^{\otimes b} $ is the canonical projection. 
\item 
\label{prop:calculate_perp4}
The diagonal $ \Delta : \gcmnd  (t^*)^{\otimes b} 
\cong 
\bigoplus_{\mathbf{b}' \subseteq \mathbf{b}} (t^*)^{\otimes b'} \rightarrow  
\gcmnd \gcmnd  (t^*)^{\otimes b} 
\cong 
\bigoplus_{\mathbf{b}'' \subseteq \tilde{\mathbf{b}}' \subseteq \mathbf{b}} (t^*)^{\otimes b''}$
with component for $ \mathbf{b}' \subseteq \mathbf{b}$  and $\mathbf{b}'' \subseteq \tilde{\mathbf{b}}' \subseteq \mathbf{b}$,    $ (t^*)^{\otimes b'} \rightarrow  
 (t^*)^{\otimes b''}$, the identity if $\mathbf{b}''=\mathbf{b}'$ and zero otherwise. (Here the notation $\tilde{\mathbf{b}}'$ is introduced for the indexing set of $\gcmnd \gcmnd  (t^*)^{\otimes b} $ to avoid confusion.)
\end{enumerate}
\end{prop}

\begin{proof}
Using that  $\gcmnd  $ is symmetric monoidal (see Corollary \ref{cor:perp_sym}), we have
$
\gcmnd  (t^* )^{\otimes b}
\cong 
(\gcmnd  t^* )^{\otimes b}
.
$ 
The decomposition of $\gcmnd  (t^* )^{\otimes b}$ follows by considering the isomorphism 
\[
\gcmnd  (t^* )^{\otimes b}
\cong 
(t^* \oplus \kring) ^{\otimes b}
\] 
that is given by using the decomposition of Lemma \ref{lem:gcmnd_t*}. The tensor factors of the tensor product are indexed by the elements of $\mathbf{b}$:  the direct summand indexed by $\mathbf{b}'$ corresponds to the term in which $t^*$ is a tensor factor for $i \in \mathbf{b}'$ and is $\kring$ otherwise. 

The decomposition of $\gcmnd \gcmnd  (t^*)^{\otimes b}$ is obtained by iterating the previous argument, since 
$$
\gcmnd \gcmnd  (t^*)^{\otimes b} \cong \bigoplus_{\mathbf{b}' \subseteq \mathbf{b}}  \gcmnd (t^*)^{\otimes b'}.
$$

The descriptions of $\psi_{(t^*)^{\otimes b}}$ and $\epsilon^\Gamma_{(t^*)^{\otimes b} } $ follow from the descriptions of $\psi : t^* \rightarrow \gcmnd t^* \cong  P^\Gamma _{\mathbf{1}_+} \cong t^* \oplus \kring$ and $\epsilon^\Gamma_{t^*} : \gcmnd t^* \rightarrow t^*$ given in Lemma \ref{lem:gcmnd_t*} and the previous decomposition. 

Since $\gcmnd$ is monoidal,  $ \Delta : \gcmnd (t^*) ^{\otimes b}  \rightarrow  
\gcmnd \gcmnd (t^*) ^{\otimes b}  $ is induced by the diagonal $\Delta : \gcmnd t^* \rightarrow  
\gcmnd \gcmnd t^*$ identified in Proposition \ref{prop:identify_Delta}. Namely,  this corresponds to 
\[
(t^* \oplus \kring) ^{\otimes b} 
\rightarrow 
(t^* \oplus \kring_0 \oplus \kring_1) ^{\otimes b} 
\]
induced by the canonical inclusion of $t^*$ and the diagonal $\kring \rightarrow \kring_0 \oplus \kring_1$. Developing the tensor products, one obtains the description of the component given in the statement.
\end{proof}

Recall that, for $b \in \nat^*$,  $\xi_b: t^* \to (t^*)^{\otimes b}$ is the standard generator given in Lemma \ref{lem:hom_Oz_a=1}.

\begin{prop}
\label{prop:perpfb}
If $b \in \nat^*$,  the morphism 
$$\gcmnd \xi_b : \gcmnd t^*\cong t^* \oplus \kring \to \gcmnd  (t^*)^{\otimes b} 
\cong 
\bigoplus_{\mathbf{b}' \subseteq \mathbf{b}} (t^*)^{\otimes b'}$$
 sends $\kring$ to the copy of $\kring$ indexed by 
 $\emptyset  \subset \mathbf{b}$ and, for any $\emptyset \neq \mathbf{b}' \subseteq \mathbf{b}$, the component 
 $ t^* \rightarrow (t^*)^{\otimes |\mathbf{b}'|}$ is the standard generator $\xi_{ |\mathbf{b}'|}$. 
\end{prop}

\begin{proof}
By Lemma \ref{lem:gcmnd_t*}, for $(Z,z)$ a finite pointed set, $\gcmnd t^*  (Z)$ has generators $[y]-[z]$, for $y \in  Z \backslash \{z\}$, corresponding to $t^*(Z)$ in the decomposition, and $[z]-[+]$, corresponding to $\kring$ in the decomposition. We have
$$\gcmnd \xi_b ([z]-[+])=([z]-[+])^{\otimes b}$$
which treats the factor $\kring$. For $[y]-[z]$ we have 
$$\gcmnd \xi_b ([y]-[z])=\gcmnd \xi_b (([y]-[+])-([z]-[+]))=\gcmnd \xi_b ([y]-[+])-\gcmnd \xi_b ([z]-[+])$$
$$=([y]-[+])^{\otimes b}-([z]-[+])^{\otimes b}=(([y]-[z])+([z]-[+]))^{\otimes b}-([z]-[+])^{\otimes b}.$$
Developing the tensor product we obtain that the component $t^* \rightarrow (t^*)^{\otimes |\mathbf{b}'|}$ indexed by $\mathbf{b}' \neq \emptyset$ is the map 
\[
([y]-[z])
\mapsto 
 ([y]-[z])^{\otimes |\mathbf{b}'|},
\]
i.e., the standard generator $\xi_{|\mathbf{b}'|}$ of $\hom_{\fcatk[\Gamma]} (t^* , (t^*)^{\otimes |\mathbf{b}'|})$, as required.

\end{proof}

\begin{rem}
\ 
\begin{enumerate}
\item
Above, we gave an explicit description of $\gcmnd (t^*)^{\otimes b}$ and 
$\gcmnd \gcmnd  (t^*)^{\otimes b} $ for $b \in \nat$. The reader's attention is drawn to the fact that, in general, we have no such description of $\gcmnd F$ for $F \in \ob \fcatk[\Gamma]$. This should be compared with Proposition \ref{prop:identify_fbcmnd}, where we give an explicit description of $\fbcmnd G$ and $\fbcmnd \fbcmnd G$ for all $G \in \ob \fcatk[\fb\op]$.
\item
The explicit description of the structure morphisms in simplicial degree $0$ and $1$ for the simplicial object $
 (\gcmnd)^{\bullet +1} (t^*)^{\otimes b}
 $ are used to establish   the isomorphism of augmented simplicial objects, $
(\gcmnd)^{\bullet +1} \twist 
\cong 
\big(
(\fbcmnd)^{\bullet +1} \twist
\big)\op$
in Proposition \ref{prop:compare_aug_simp_comonads}. 
\end{enumerate}
\end{rem}

\subsection{Comparing comonads}

In this section we work with functors in  the category $\fcatk[\Gamma \times \fb\op]$; this  is equivalent to the category of functors from $\Gamma$ to $\fcatk[\fb \op]$ and also to the category of functors from $\fb \op$ to $\fcatk[\Gamma]$. 
Fixing one of the variables, we can consider the following comonads on $\fcatk[\Gamma \times \fb\op]$:
\begin{itemize}
\item 
$\gcmnd: \fcatk[\Gamma \times \fb\op] \to \fcatk[\Gamma \times \fb\op]$ obtained from the comonad introduced in Notation \ref{nota:gcmnd}.
\item 
$\fbcmnd: \fcatk[\Gamma \times \fb\op] \to \fcatk[\Gamma \times \fb\op]$ obtained from the comonad introduced in Notation \ref{nota:fbcmnd}.
\end{itemize}

Here we  establish two results concerning these comonads: in Proposition \ref{prop:comonads_commute} we prove that these two comonads commute and in Proposition \ref{prop:compare_comonads} we construct a natural isomorphism $\gcmnd \twist \cong \fbcmnd \twist$, compatible with the counits, where $\twist$ is the bifunctor  introduced in Definition \ref{defn:twist} that encodes the functors $(t^*)^{\otimes b}$ for $b \in \nat$. 

\begin{prop}
\label{prop:comonads_commute}
The comonads 
$\gcmnd, \fbcmnd: \fcatk[\Gamma \times \fb\op] \to \fcatk[\Gamma \times \fb\op]$
commute up to natural isomorphism. Explicitly, for $F \in \ob \fcatk[\Gamma \times \fb\op]$, there is a natural isomorphism
\[
\gcmnd \fbcmnd F \stackrel{\cong}{\rightarrow} \fbcmnd \gcmnd F. 
\]
This is compatible with the respective counits, in that the following diagrams commute: 
\[
\xymatrix{
\gcmnd \fbcmnd F 
\ar[rr]^\cong 
\ar[dr] _{\epsilon^\Gamma_{\fbcmnd F}} 
&
&
\fbcmnd \gcmnd F
\ar[dl]^{\fbcmnd \epsilon^\Gamma_F}
&
\gcmnd \fbcmnd F 
\ar[rr]^\cong 
\ar[dr]_{\gcmnd \epsilon^\Sigma_F}
&
&
\fbcmnd \gcmnd F
\ar[dl]^{\epsilon^\Sigma_{\gcmnd F}}
\\
&
\fbcmnd F
&
&
&
\gcmnd F . 
}
\]
\end{prop}

\begin{proof}
By Proposition \ref{prop:identify_fbcmnd} part (\ref{item:p1}), for $X \in \ob \Gamma$ and $\mathbf{b} \in \ob \fb$, 
\[
\fbcmnd F (X, \mathbf{b}) 
= 
\bigoplus_{\mathbf{b}' \subseteq \mathbf{b} } F (X, \mathbf{b}').
\]
Hence, by the Definition of $\gcmnd$ (see Notation \ref{nota:gcmnd}), 
\[
\gcmnd \fbcmnd F (X, \mathbf{b})
= 
\fbcmnd F (X_+, \mathbf{b})
= 
 \bigoplus_{\mathbf{b}' \subseteq \mathbf{b} } F (X_+, \mathbf{b}')
\]
and the right hand side is canonically isomorphic to  $\bigoplus_{\mathbf{b}' \subseteq \mathbf{b} } \gcmnd F (X, \mathbf{b}')$ which, in turn, is canonically isomorphic to $ \fbcmnd \gcmnd F (X, \mathbf{b})$. These isomorphisms are natural with respect to $X$ and $\mathbf{b}$. 

The compatibility with the counits is checked by using their explicit  identifications, given in 
Proposition \ref{prop:comonad-gamma-explicit} and  Proposition \ref{prop:identify_fbcmnd} respectively.
\end{proof}

We will apply the comonads $\gcmnd, \fbcmnd$ to the following functor.

\begin{defn}
\label{defn:twist}
Let $\twist$ be the functor in $\fcatk[\Gamma \times \fb\op]$  (considered here as functors from $\fb \op$ to $\fcatk[\Gamma]$)  given by  
\[
\twist (\mathbf{b}) := (t^*)^{\otimes b},
\]
 where $\mathbf{b} \in \ob \fb\op$ and the symmetric group $\sym_b$ acts on the right via place permutations. 
\end{defn}

For $\mathbf{b} \in \ob \fb\op$, by Proposition \ref{prop:identify_fbcmnd} and Proposition \ref{prop:calculate_perp}, there are natural isomorphisms:
$$\fbcmnd \twist (\mathbf{b})
\cong 
\bigoplus_{\mathbf{b}' \subseteq \mathbf{b}} (t^*)^{\otimes b'}; \qquad 
\gcmnd  \twist (\mathbf{b})
\cong 
\gcmnd  (t^*)^{\otimes b} 
\cong 
\bigoplus_{\mathbf{b}' \subseteq \mathbf{b}} (t^*)^{\otimes b'}.$$

This leads to the following:

\begin{prop}
\label{prop:compare_comonads}
There is an isomorphism
$
\overline{\alpha}: 
\gcmnd \twist 
\rightarrow 
\fbcmnd \twist
$
in $\fcatk[\Gamma \times \fb\op]$
defined, for $b \in \nat$, by the isomorphisms in $\fcatk[\Gamma]$ 
$$\overline{\alpha}_\mathbf{b}: \gcmnd  (t^*)^{\otimes b} 
\cong 
\bigoplus_{\mathbf{b}' \subseteq \mathbf{b}} (t^*)^{\otimes b'}
\to 
\fbcmnd \twist (\mathbf{b})
\cong 
\bigoplus_{\mathbf{b}' \subseteq \mathbf{b}} (t^*)^{\otimes b'}$$
given by the identity on each factor $(t^*)^{\otimes b'}.$
Moreover, this isomorphism is compatible with the respective counits, in that the following diagram commutes:
 \[
 \xymatrix{
\gcmnd \twist 
\ar[rr]_{\cong}^{\overline{\alpha}}
\ar[rd]_{\epsilon^\Gamma_{\twist}}
&&
\fbcmnd \twist
\ar[ld]^{\epsilon^\Sigma_{\twist}}
\\
&
\twist.
 }
 \]
\end{prop}

\begin{proof}
For the first part, one checks that the isomorphisms $\overline{\alpha}_\mathbf{b}$ are compatible with the action of the symmetric group $\sym_b$. Here, the action of $\sym_b$ on $\fbcmnd \twist (\mathbf{b})$ is given by Proposition \ref{prop:right_adjoint_fbfi-explicit} (\ref{item:finj_op_structure_uparrow}) and  on 
$\gcmnd  (t^*)^{\otimes b} \cong (\gcmnd  t^*)^{\otimes b}$ is given by place permutations.

The compatibility with the counits follows from the explicit description of the counits given in  Proposition \ref{prop:identify_fbcmnd} (\ref{prop:identify_fbcmnd2}) and Proposition \ref{prop:calculate_perp} (\ref{prop:calculate_perp3}).
\end{proof}

\subsection{An isomorphism of augmented simplicial objects}
The purpose of this section is to construct an isomorphism of augmented simplicial objects $\gcmnd^{\bullet +1} \twist \cong \big( \fbcmnd ^{\bullet +1} \twist\big) \op$ extending the natural isomorphism $\gcmnd \twist \cong \fbcmnd \twist$ given in Proposition \ref{prop:compare_comonads}. Here $\op$ denotes the opposite simplicial structure (see Notation \ref{nota:opp_simp}).

\begin{defn}
\label{defn:overline_alpha_ell}
For $\ell \in \nat$, let $\overline{\alpha} ^\ell : (\gcmnd)^{\ell} \twist \stackrel{\cong}{\rightarrow} (\fbcmnd)^{\ell} \twist$ be the isomorphism in $\fcatk[\Gamma \times \fb\op]$ defined recursively by:
\begin{enumerate}
\item
$\overline{\alpha}^0 = \id_{\twist}$;
\item 
for $\ell >0$, $\overline{\alpha}^{\ell}$ is the composite:
\[
(\gcmnd)^{\ell} \twist
= (\gcmnd)^{\ell -1 } \gcmnd  \twist
\stackrel{(\gcmnd)^{\ell -1} \overline{\alpha}  }{\longrightarrow} 
(\gcmnd)^{\ell -1 } \fbcmnd  \twist
\cong 
 \fbcmnd (\gcmnd)^{\ell -1 }  \twist
 \stackrel{\fbcmnd \overline{\alpha}^{\ell -1}}{\longrightarrow}
 \fbcmnd (\fbcmnd)^{\ell -1 }  \twist =  (\fbcmnd)^{\ell} \twist,
\]
in which the middle isomorphism is given by Proposition \ref{prop:comonads_commute}. 
\end{enumerate}
In particular, $\overline{\alpha}^1 = \overline{\alpha}$. 
\end{defn}

In order to give an explicit description of $\overline{\alpha} ^\ell$ in Proposition \ref{prop:alpha-l-explicit}, we need the following reindexation of iterates of the comonads  $ \fbcmnd$ and $ \gcmnd$.
\begin{rem}
\label{rem:reindexation}
Recall that by Remark \ref{rem:reindexation2} (\ref{iteration-Sigma}) we have 
\begin{equation}
\label{iteration-Sigma-2}
 \fbcmnd_0 \fbcmnd_1 \ldots \fbcmnd_\ell G (\mathbf{b}) 
\cong
\bigoplus_{\mathbf{b}=(\ldots ((\mathbf{b}^{(\ell+1)}_\Sigma \amalg \mathbf{b}^{\Sigma}_\ell) \amalg \mathbf{b}^{\Sigma}_{\ell-1})) \ldots \amalg \mathbf{b}^{\Sigma}_1) \amalg \mathbf{b}^{\Sigma}_0 }
G (\mathbf{b}^{(\ell+1)}_\Sigma).  
\end{equation}
On the other hand, the decomposition given in Proposition \ref{prop:calculate_perp} can be reindexed as follows:
$$\gcmnd_0 \gcmnd_1  \twist (\mathbf{b}) 
\cong 
\gcmnd_0 (\bigoplus_{\mathbf{b}= \mathbf{b}^{(1)}_\Gamma \amalg \mathbf{b}^{\Gamma}_1 } (t^*)^{\otimes b^{(1)}_\Gamma })
\cong
\bigoplus_{\mathbf{b}= (\mathbf{b}^{(2)}_\Gamma   \amalg \mathbf{b}^{\Gamma}_0) \amalg \mathbf{b}^{\Gamma}_1 } \twist ( b^{(2)}_\Gamma ).
$$
With respect to the decomposition given in Proposition \ref{prop:calculate_perp} we have  $\mathbf{b}^{(2)}_\Gamma  := \mathbf{b}''$, $\mathbf{b}^{\Gamma}_0:=\mathbf{b}'\backslash \mathbf{b}''$ and $\mathbf{b}^{\Gamma}_1:=\mathbf{b}\backslash \mathbf{b}'$. 
More generally, we have:
\begin{equation}
\label{iteration-Gamma}
  \gcmnd_0 \gcmnd_1 \ldots \gcmnd_\ell  \twist (\mathbf{b}) 
\cong
\bigoplus_{\mathbf{b}=(\ldots ((\mathbf{b}^{(\ell+1)}_\Gamma  \amalg \mathbf{b}^{\Gamma}_0) \amalg \mathbf{b}^{\Gamma}_{1})) \ldots \amalg \mathbf{b}^{\Gamma}_{\ell-1}) \amalg \mathbf{b}^{\Gamma}_\ell }
\twist (\mathbf{b}^{(\ell+1)}_\Gamma).  
\end{equation}
\end{rem}

Note the reversal of the order of the indices in the decomposition of $\mathbf{b}$ between (\ref{iteration-Sigma-2}) and  (\ref{iteration-Gamma}). This  explains why it is  the   opposite structure for the augmented simplicial object $ \fbcmnd ^{\bullet +1} \twist$ that arises in Proposition \ref{prop:compare_aug_simp_comonads} and justifies why we consider  $ \fbcmnd_\ell \fbcmnd_{\ell-1} \ldots \fbcmnd_0 \twist $ in the following Proposition.

\begin{prop}
\label{prop:alpha-l-explicit}
For $\ell \in \nat$, the isomorphism $\overline{\alpha} ^\ell : (\gcmnd)^{\ell} \twist \stackrel{\cong}{\rightarrow} (\fbcmnd)^{\ell} \twist$ in $\fcatk[\Gamma \times \fb\op]$
is defined, for $b \in \nat$, by the isomorphisms in $\fcatk[\Gamma]$ 
$$
\xymatrix{
\gcmnd_0 \ldots \gcmnd_{\ell-1}  \twist (\mathbf{b})
 \ar[rr]^{\overline{\alpha}_\mathbf{b}^\ell} 
&& 
 \fbcmnd_{\ell-1}  \ldots \fbcmnd_0  \twist (\mathbf{b}) 
  \ar[d]^\cong
  \\
  \bigoplus_{\mathbf{b}=((\mathbf{b}^{(\ell)}_\Gamma \amalg \mathbf{b}^{\Gamma}_0)\amalg \ldots  \amalg \mathbf{b}^{\Gamma}_{\ell-1}) }
\twist ( \mathbf{b}^{(\ell)}_\Gamma)
\ar[u]^\cong
&& 
\bigoplus_{\mathbf{b}=((\mathbf{b}^{(\ell)}_\Sigma \amalg \mathbf{b}^{\Sigma}_0)  \ldots \amalg  \mathbf{b}^{\Sigma}_{\ell-1} )}
\twist (\mathbf{b}^{(\ell)}_\Sigma)
}
$$
sending the component indexed by $\mathbf{b}=(\ldots ((\mathbf{b}^{(\ell)}_\Gamma \amalg \mathbf{b}^{\Gamma}_0) \amalg \mathbf{b}^{\Gamma}_{1})) \ldots \amalg \mathbf{b}^{\Gamma}_{\ell-2}) \amalg \mathbf{b}^{\Gamma}_{\ell-1}$ to the component with $\mathbf{b}^{(\ell)}_\Sigma=\mathbf{b}^{(\ell)}_\Gamma$ and $\mathbf{b}^{\Sigma}_{i}=\mathbf{b}^{\Gamma}_{i}$ for $i \in \{0, \ldots, \ell-1\}$, where the vertical isomorphisms are given by (\ref{iteration-Gamma}).
\end{prop}

To prove this Proposition we need to make explicit the natural isomorphism given in Proposition \ref{prop:comonads_commute} for $F=\twist $. First note that there are isomorphisms
$$
 \gcmnd_0 (\fbcmnd_{1} \twist)(\mathbf{b}) \cong \gcmnd_0 (\bigoplus_{\mathbf{b}= \mathbf{b}^{(1)}_\Sigma \amalg \mathbf{b}^{\Sigma}_1 } (t^*)^{\otimes b^{(1)}_\Sigma }) \cong 
\bigoplus_{\mathbf{b}= (\mathbf{b}^{(2)}_\Gamma   \amalg \mathbf{b}^{\Gamma}_0) \amalg \mathbf{b}^{\Sigma}_1 } (t^*)^{\otimes b^{(2)}_\Gamma }
$$
and
$$
 \fbcmnd_{1} (\gcmnd_0  \twist)(\mathbf{b}) \cong 
\bigoplus_{\mathbf{b}= \mathbf{b}^{(1)}_\Sigma \amalg \mathbf{b}^{\Sigma}_1 } (\gcmnd_0 \twist)( \mathbf{b}^{(1)}_\Sigma  ) \cong 
\bigoplus_{\mathbf{b}= (\mathbf{b}^{(2)}_\Gamma   \amalg \mathbf{b}^{\Gamma}_0) \amalg \mathbf{b}^{\Sigma}_1 } (t^*)^{\otimes b^{(2)}_\Gamma }.
$$

The proof of the following is left to the reader:

\begin{lem}
\label{lm:commute-explicit}
The natural isomorphism
$
\gcmnd_0 \fbcmnd_1 \twist \stackrel{\cong}{\rightarrow} \fbcmnd_1 \gcmnd_0 \twist
$
in $\fcatk[\Gamma \times \fb\op]$ of Proposition \ref{prop:comonads_commute}  is given, for $b \in \nat$, by the isomorphisms in $\fcatk[\Gamma]$: 
$$ 
\gcmnd_0 (\fbcmnd_{1} \twist)(\mathbf{b}) \cong \bigoplus_{\mathbf{b}= (\mathbf{b}^{(2)}_\Gamma   \amalg \mathbf{b}^{\Gamma}_0) \amalg \mathbf{b}^{\Sigma}_1 } (t^*)^{\otimes b^{(2)}_\Gamma } \rightarrow 
 \fbcmnd_{1} (\gcmnd_0  \twist)(\mathbf{b}) \cong \bigoplus_{\mathbf{b}= (\mathbf{b}^{(2)}_\Gamma   \amalg \mathbf{b}^{\Gamma}_0) \amalg \mathbf{b}^{\Sigma}_1 } (t^*)^{\otimes b^{(2)}_\Gamma }
 $$
 determined by the identity on each component indexed by $\mathbf{b}= (\mathbf{b}^{(2)}_\Gamma   \amalg \mathbf{b}^{\Gamma}_0) \amalg \mathbf{b}^{\Sigma}_1 $.
\end{lem}

\begin{proof}[Proof of Proposition \ref{prop:alpha-l-explicit}]
We prove the result by induction on $\ell$. For $\ell=1$, the result is true by Proposition \ref{prop:compare_comonads}. To prove the inductive step, consider the component of $\gcmnd_0 \ldots \gcmnd_{\ell-1}  \twist (\mathbf{b})$ indexed by $(\ldots(\mathbf{b}^{(\ell)}_\Gamma \amalg \mathbf{b}^{\Gamma}_0)\amalg \ldots  \amalg \mathbf{b}^{\Gamma}_{\ell-2}) \amalg \mathbf{b}^{\Gamma}_{\ell-1}$. By Proposition \ref{prop:compare_comonads}, $(\gcmnd)^{\ell -1} \overline{\alpha} $ sends this component to the component indexed by
$$(\ldots (\mathbf{b}^{(\ell)}_\Gamma \amalg \mathbf{b}^{\Gamma}_0)\amalg \ldots  \amalg \mathbf{b}^{\Gamma}_{\ell-2}) \amalg \mathbf{b}^{\Sigma}_{\ell-1}.$$
The result follows from Lemma \ref{lm:commute-explicit} and the inductive hypothesis.
\end{proof}

Proposition \ref{prop:alpha-l-explicit} and Lemma \ref{lm:commute-explicit} yield the following useful result:

\begin{lem}
\label{lem:compose_overline_alpha}
For $\ell, m\in \nat$, there is a commutative diagram of isomorphisms:
\[
\xymatrix{
(\gcmnd)^{\ell +m } \twist
\ar@{=}[r]
\ar[d]_{\overline{\alpha}^{\ell + m}}
&
(\gcmnd)^{\ell}(\gcmnd)^{m } \twist
\ar[r]^{(\gcmnd)^{\ell} \overline{\alpha}^m}
&
(\gcmnd)^{\ell}(\fbcmnd)^{m } \twist
\ar[d]^{\cong}
\\
(\fbcmnd)^{\ell +m } \twist
\ar@{=}[r]
&
(\fbcmnd)^{m}(\fbcmnd)^{\ell } \twist
&
(\fbcmnd)^{m}(\gcmnd)^{\ell } \twist,
\ar[l]^{(\fbcmnd)^{m} \overline{\alpha}^\ell}
}
\] 
in which the right hand vertical morphism is induced by the interchange isomorphism of Proposition \ref{prop:comonads_commute}. 
\end{lem}

By Proposition \ref{prop:simplicial}, we have the  augmented simplicial objects  $(\gcmnd)^{\bullet +1} \twist$ and $(\fbcmnd) ^{\bullet +1} \twist$ in $\fcatk[\Gamma \times \fb\op]$ associated to the comonads $\gcmnd$ and $\fbcmnd$ respectively. In the following, $\op$ denotes the opposite augmented simplicial structure, as in Notation \ref{nota:opp_simp}.

\begin{prop}
\label{prop:compare_aug_simp_comonads}
The isomorphisms $\overline{\alpha}^{\ell+1} : (\gcmnd)^{\ell+1} \twist \stackrel{\cong}{\rightarrow} (\fbcmnd)^{\ell+1} \twist$ for $\ell \geq -1$ give  an  isomorphism of  augmented simplicial objects in $\fcatk[\Gamma \times \fb\op]$ 
\[
(\gcmnd)^{\bullet +1} \twist 
\cong 
\big(
(\fbcmnd)^{\bullet +1} \twist
\big)\op
\]
extending the identity on $\twist$ in degree $-1$.
\end{prop}

\begin{proof}
By Proposition \ref{prop:compare_comonads} and the construction of the morphisms $\overline{\alpha}^{\ell+1}$, these give  an isomorphism of the underlying graded objects. It remains to check compatibility with the respective augmented simplicial structures.  

We begin by proving the compatibility with the face maps by induction on the simplicial degree.

The compatibility in simplicial degrees $-1$ and $0$ is an immediate consequence of Proposition \ref{prop:compare_comonads}. For the compatibility in simplicial degrees $0$ and $1$, by definition of the face maps given in Proposition \ref{prop:simplicial}, we have to show that the following two diagrams commute:
\[
\xymatrix{
\gcmnd \gcmnd \twist
\ar[r]_\cong^{\overline{\alpha}^2} 
\ar[d]_{\delta_0^\Gamma=\epsilon_{\gcmnd \twist}}
&
\fbcmnd \fbcmnd \twist
\ar[d]^{\widetilde{\delta_0}^\Sigma=\delta_1^\Sigma=\fbcmnd \epsilon_{\twist}}
&
&
&
\gcmnd \gcmnd \twist
\ar[r]_\cong^{\overline{\alpha}^2} 
\ar[d]_{\delta_1^\Gamma=\gcmnd \epsilon_\twist}
&
\fbcmnd \fbcmnd \twist
\ar[d]^{\widetilde{\delta_1}^\Sigma=\delta_0^\Sigma=\epsilon_{\fbcmnd \twist}}
\\
\gcmnd \twist
\ar[r]^\cong_{\overline{\alpha}} 
&\fbcmnd \twist
&
&
&
\gcmnd \twist
\ar[r]^\cong_{\overline{\alpha}} 
&\fbcmnd \twist,
}
\]
where Notation \ref{nota:opp_simp} is used for the opposite simplicial structure maps.

For the left hand square, using the definition of $\overline{\alpha}^2$, 
this follows from the commutativity of the following:
\[
\xymatrix{
\gcmnd \gcmnd \twist
\ar[d]_{\epsilon_{\gcmnd \twist}}
\ar[r]^{\gcmnd \overline{\alpha}}
&
\gcmnd \fbcmnd \twist
\ar[d]|{\epsilon_{\fbcmnd \twist}}
\ar[r]^{\cong}
&
\fbcmnd  \gcmnd  \twist
\ar[d]|{\fbcmnd \epsilon_{\twist}}
\ar[r]^{\fbcmnd \overline{\alpha}}
&
\fbcmnd  \fbcmnd \twist
\ar[ld]^{\fbcmnd \epsilon_{\twist}}
\\
\gcmnd \twist
\ar[r]_{\overline{\alpha}}
&
\fbcmnd \twist
\ar@{=}[r]
&
\fbcmnd \twist.
}
\]
Here the left hand square commutes by naturality of $\epsilon$; the middle square is commutative by  Proposition \ref{prop:comonads_commute}; the right hand triangle commutes by applying $\fbcmnd$ to the commutative triangle of Proposition \ref{prop:compare_comonads}.

For the second square, one proceeds similarly, by establishing the commutativity of the diagram:
\begin{equation}
\label{preuve:DC}
\xymatrix{
\gcmnd \gcmnd \twist
\ar[rd]_{\gcmnd \epsilon_{ \twist}}
\ar[r]^{\gcmnd \overline{\alpha}}
&
\gcmnd \fbcmnd \twist
\ar[d]|{\gcmnd \epsilon_{\twist}}
\ar[r]^{\cong}
&
\fbcmnd  \gcmnd  \twist
\ar[d]|{ \epsilon_{\gcmnd \twist}}
\ar[r]^{\fbcmnd \overline{\alpha}}
&
\fbcmnd  \fbcmnd \twist
\ar[d]^{ \epsilon_{\fbcmnd\twist}}
\\
&
\gcmnd \twist 
\ar@{=}[r]
&
\gcmnd \twist 
\ar[r]_{\overline{\alpha}}
&
\fbcmnd\twist.
}
\end{equation}

Assume that the compatibility in simplicial degrees $i-1$ and $i$ for $i\leq n-1$ is satisfied for the face maps. We have to show that the following diagrams commute, for $i\in \{0, \ldots, n\}$:
\[
\xymatrix{
(\gcmnd)^{n+1} \twist
\ar[r]_\cong^{\overline{\alpha}^{n+1}} 
\ar[d]_{\delta_i^\Gamma=(\gcmnd)^i  \epsilon_{(\gcmnd)^{n-i} \twist}}
&
(\fbcmnd)^{n+1}  \twist
\ar[d]^{\widetilde{\delta_i}^\Sigma=\delta_{n-i}^\Sigma=(\fbcmnd)^{n-i} \epsilon_{(\fbcmnd)^{i} \twist}}
\\
(\gcmnd)^{n} \twist
\ar[r]^\cong_{\overline{\alpha}^n} 
&(\fbcmnd)^{n}  \twist.
}
\]
For $i=n$, the proof of the commutativity of the diagram is similar to that for the diagram (\ref{preuve:DC}).

For $i \neq n$, using Lemma \ref{lem:compose_overline_alpha}, this follows from the commutative diagram
\\ 
\scalebox{0.8}{
$
\xymatrix{
 (\gcmnd)^{i} \gcmnd (\gcmnd)^{n-i} \twist
\ar[d]|{(\gcmnd)^i  \epsilon^\Gamma_{(\gcmnd)^{n-i} \twist}}
\ar[r]_{(\gcmnd)^{i} \gcmnd \overline{\alpha}^{n-i}}
&
 (\gcmnd)^{i} \gcmnd (\fbcmnd)^{n-i} \twist
\ar[d]|{(\gcmnd)^{i}  \epsilon^\Gamma_{(\fbcmnd)^{n-i} \twist}}
\ar[r]^{\cong}
&
 (\gcmnd)^{i} (\fbcmnd)^{n-i}  \gcmnd \twist
\ar[d]|{ (\gcmnd)^{i} (\fbcmnd)^{n-i} \epsilon_{\twist}}
\ar[r]^{\cong}
&
(\fbcmnd)^{n-i}   (\gcmnd)^{i} \gcmnd \twist
\ar[d]|{(\fbcmnd)^{n-i}  (\gcmnd)^{i} \epsilon_{\twist}}
\ar[r]_{(\fbcmnd)^{n-i} \overline{\alpha}^{i+1}}
&
(\fbcmnd)^{n-i}   (\fbcmnd)^{i+1}  \twist
\ar[d]|{(\fbcmnd)^{n-i} \epsilon_{(\fbcmnd)^{i} \twist}}
\\
 (\gcmnd)^{i} (\gcmnd)^{n-i} \twist
 \ar[r]_{(\gcmnd)^{i} \overline{\alpha}^{n-i}}
 &
  (\gcmnd)^{i} (\fbcmnd)^{n-i} \twist
  \ar@{=}[r]
   &
  (\gcmnd)^{i} (\fbcmnd)^{n-i} \twist
  \ar[r]_{\cong}
  &
(\fbcmnd)^{n-i}    (\gcmnd)^{i} \twist
\ar[r]_{(\fbcmnd)^{n-i} \overline{\alpha}^{i}}
&
(\fbcmnd)^{n-i}   (\fbcmnd)^{i}  \twist,
}
$
}
\\
where the left hand square commutes by naturality of $\epsilon$, the second square commutes by iterating Proposition \ref{prop:comonads_commute}, the third square commutes by naturality of the isomorphism $\gcmnd \fbcmnd F \stackrel{\cong}{\rightarrow} \fbcmnd \gcmnd F$ and the last square is seen to commute by applying $(\fbcmnd)^{n-i} $ to the commutative diagram obtained by the inductive  hypothesis in simplicial degrees $i-1$ and $i$.

We now consider the degeneracies; these are induced by the respective coproducts $\Delta$. 
For the compatibility in simplicial degrees $0$ and $1$, we require to show that  the diagram 
\begin{equation}
\label{proof:DC-1}
\xymatrix{
\gcmnd_0 \twist 
\ar[d]_{\Delta}
\ar[r]_\cong^{\overline{\alpha}}
&
\fbcmnd_0 \twist 
\ar[d]^{\Delta}
\\
\gcmnd_0 \gcmnd_1 \twist 
\ar[r]^\cong_{\overline{\alpha}^2}
&
\fbcmnd_1 \fbcmnd_0 \twist
}
\end{equation}
commutes.

 Using the reindexation given in  Remark \ref{rem:reindexation}, by Proposition \ref{prop:compare_comonads}, for $b \in \nat$, the map 
$$
\overline{\alpha}_\mathbf{b}: 
\gcmnd_0  \twist (\mathbf{b})
\cong 
\bigoplus_{\mathbf{b}= \mathbf{b}^{(1)}_\Gamma \amalg \mathbf{b}^{\Gamma}_0 } 
\twist ( \mathbf{b}^{(1)}_\Gamma)
\to 
\fbcmnd_0 \twist (\mathbf{b})
\cong 
\bigoplus_{\mathbf{b}= \mathbf{b}^{(1)}_\Sigma \amalg \mathbf{b}^{\Sigma}_0 } 
\twist (\mathbf{b}^{(1)}_\Sigma)
$$
sends the component indexed by $\mathbf{b}= \mathbf{b}^{(1)}_\Gamma \amalg \mathbf{b}^{\Gamma}_0$ to the component with $\mathbf{b}^{(1)}_\Sigma =\mathbf{b}^{(1)}_\Gamma$ and $\mathbf{b}^{\Sigma}_0 = \mathbf{b}^{\Gamma}_0$.

Proposition \ref{prop:identify_fbcmnd}  identifies $\Delta_\mathbf{b} : \fbcmnd_0 \twist (\mathbf{b})\rightarrow \fbcmnd_1 \fbcmnd_0 \twist (\mathbf{b})$ as the morphism 
\[
\bigoplus_{\mathbf{b}= \mathbf{b}^{(1)}_\Sigma \amalg \mathbf{b}^{\Sigma}_0} 
\twist (\mathbf{b}^{(1)}_\Sigma )
\rightarrow 
\bigoplus _{\mathbf{b}= \mathbf{b}^{(2)}_\Sigma \amalg \mathbf{b}^{\Sigma}_0 \amalg \mathbf{b}^{\Sigma}_1} 
\twist ( \mathbf{b}^{(2)}_\Sigma),
\]
where the component indexed by the pair of decompositions $(\mathbf{b}= \mathbf{b}^{(1)}_\Sigma \amalg \mathbf{b}^{\Sigma}_1, \mathbf{b}= \mathbf{b}^{(2)}_\Sigma \amalg \mathbf{b}^{\Sigma}_0 \amalg \mathbf{b}^{\Sigma}_1)$ is zero unless $\mathbf{b}^{(1)}_\Sigma= \mathbf{b}^{(2)}_\Sigma$, when it is the identity morphism $\twist (\mathbf{b}^{(1)}_\Sigma) \to \twist (\mathbf{b}^{(2)}_\Sigma)$.

By Proposition \ref{prop:calculate_perp} (\ref{prop:calculate_perp4}), the diagonal $ \Delta_\mathbf{b} : \gcmnd_0  \twist (\mathbf{b}) 
\rightarrow  
\gcmnd_0 \gcmnd_1 \twist (\mathbf{b})
$
identifies as the morphism
\[
\bigoplus_{\mathbf{b}= \mathbf{b}^{(1)}_\Gamma \amalg \mathbf{b}^{\Gamma}_0 } 
\twist ( \mathbf{b}^{(1)}_\Gamma)
\rightarrow 
\bigoplus _{\mathbf{b}= \mathbf{b}^{(2)}_\Gamma \amalg \mathbf{b}^{\Gamma}_0 \amalg \mathbf{b}^{\Gamma}_1} 
\twist (\mathbf{b}^{(2)}_\Gamma),
\]
where the component indexed by the pair of decompositions $(\mathbf{b}= \mathbf{b}^{(1)}_\Gamma \amalg \mathbf{b}^{\Gamma}_0, \mathbf{b}= \mathbf{b}^{(2)}_\Gamma \amalg \mathbf{b}^{\Gamma}_0 \amalg \mathbf{b}^{\Gamma}_1)$ is zero unless $\mathbf{b}^{(1)}_\Gamma= \mathbf{b}^{(2)}_\Gamma$, when it is the identity morphism $\twist (\mathbf{b}^{(1)}_\Gamma) \to  \twist (\mathbf{b}^{(2)}_\Gamma) $.  

Proposition \ref{prop:alpha-l-explicit} identifies $\overline{\alpha}^2_\mathbf{b} : \gcmnd_0 \gcmnd_1 \twist (\mathbf{b}) \to \fbcmnd_1 \fbcmnd_0 \twist(\mathbf{b}) $
as the morphism
$$
\bigoplus_{\mathbf{b}=(\mathbf{b}^{(2)}_\Gamma \amalg \mathbf{b}^{\Gamma}_0)\amalg \mathbf{b}^{\Gamma}_{1} }
\twist (\mathbf{b}^{(2)}_\Gamma) 
\to 
\bigoplus_{\mathbf{b}=((\mathbf{b}^{(2)}_\Sigma \amalg \mathbf{b}^{\Sigma}_0)  \amalg  \mathbf{b}^{\Sigma}_{1} }
\twist (\mathbf{b}^{(2)}_\Sigma) 
$$
sending the component indexed by $\mathbf{b}=(\mathbf{b}^{(2)}_\Gamma \amalg \mathbf{b}^{\Gamma}_0)\amalg \mathbf{b}^{\Gamma}_{1}$ to the component with $\mathbf{b}^{(2)}_\Sigma=\mathbf{b}^{(2)}_\Gamma$ and $\mathbf{b}^{\Sigma}_{i}=\mathbf{b}^{\Gamma}_{i}$ for $i \in \{0, 1\}$. The commutativity of the diagram (\ref{proof:DC-1}) follows.

The compatibility for degeneracy maps in higher simplicial degree is proved by induction, using 
 the naturality of $\Delta$, Proposition \ref{prop:comonads_commute} and a similar commutative diagram as for the face maps.
\end{proof}

\subsection{Proof of Theorem \ref{thm:identify_cosimp_gcmnd_fbcmnd}}

We begin by constructing the cosimplicial object $[\cosimp^\bullet \kring \hom_\Oz (-, \mathbf{a})]\op$, appearing in the statement of Theorem \ref{thm:identify_cosimp_gcmnd_fbcmnd}, as a particular case of (\ref{cosimplicial-obj-1}). For this we introduce the $\finj\op$-module $\kring \hom_\Oz (-, \mathbf{a})$.

\begin{prop}
\label{prop:hom_oz_functor_finjop}
For $a \in \nat$, the association $\mathbf{b} \mapsto \kring \hom_\Oz (\mathbf{b}, \mathbf{a})$ defines a functor  in $\fcatk[\finj\op]$, where a generator $f\in \hom_\Oz (\mathbf{b}, \mathbf{a})$ is sent by a morphism $i : \mathbf{b}' \hookrightarrow \mathbf{b}$ of $\finj$ to the generator $f \circ i $ given by restriction along $i$, if this is a morphism of $\Oz$, and to zero otherwise.  
\end{prop}

\begin{proof}
The case $a=0$ follows immediately from the fact that $\hom_\Oz (\mathbf{b}, \mathbf{0})$ is empty unless $\mathbf{b} = \mathbf{0}$. 

For $a>1$, $f\circ i $ belongs to $\Oz$ if and only if it is surjective. Since non-surjectivity is preserved under restriction, it is clear that the above defines a functor on $\finj\op$.
\end{proof}

\begin{rem}
We have $ \kring \hom_\Oz (-, \mathbf{0})=\mathbbm{1}$ where $\mathbbm{1}$ is the $\finj\op$-module introduced in Notation \ref{notation-1}.
\end{rem}

The association $(\mathbf{a},\mathbf{b}) \mapsto \hom_{\fcatk[\Gamma]} ((t^*)^{\otimes a},  (t^*)^{\otimes b})$, with the place permutation action on $(t^*)^{\otimes a}$ and $(t^*)^{\otimes b}$, defines a $\fb \times \fb\op$-module. Proposition \ref{prop:explicit_DK_morphism} can be rephrased as follows: 

\begin{prop}
\label{prop:iso-FB-mod}
The maps $\dk$ induce an isomorphism of $\fb \times \fb\op$-modules:
$$\kring \dk :\  \kring \hom_\Oz (-, -) \xrightarrow{\cong} \hom_{\fcatk[\Gamma]} ({t^*}^{\otimes},  \twist).$$ 
\end{prop}

By Proposition \ref{prop:cosimp_comod}, the augmented simplicial object  underlying $\cosimp^\bullet \kring \hom_\Oz (-, \mathbf{a})$ identifies as  $
(\fbcmnd)^{\bullet+1} \kring \hom_{\Oz} ( - , \mathbf{a} )$. In the following Proposition, we deduce from Proposition \ref{prop:compare_aug_simp_comonads} that the underlying augmented simplicial objects of the cosimplicial objects appearing in Theorem \ref{thm:identify_cosimp_gcmnd_fbcmnd} are isomorphic.

\begin{prop}
\label{cor:iso_aug_simp}
For $a \in \nat$, there is an isomorphism of augmented simplicial objects in 
$\fcatk[\fb \op]$:
\begin{eqnarray*}
\hom_{\fcatk[\Gamma]} ((t^*) ^{\otimes a}, (\gcmnd)^{\bullet+1} \twist  ) 
&\cong &
\big(
(\fbcmnd)^{\bullet+1} \kring \hom_{\Oz} ( - , \mathbf{a} )\big)\op.
\end{eqnarray*} 
\end{prop}

\begin{proof}
By Proposition \ref{prop:compare_aug_simp_comonads}, there is an isomorphism of augmented simplicial objects in $\fcatk[\fb \op]$:
\[
\hom_{\fcatk[\Gamma]} ((t^*) ^{\otimes a}, (\gcmnd)^{\bullet+1} \twist  ) 
\cong 
 \hom_{\fcatk[\Gamma]} ((t^*)^{\otimes a}, \big((\fbcmnd)^{\bullet+1} \twist\big)\op).
\]

Now, for any $G \in \fcatk[\Gamma \times \fb \op]$, since $\hom_{\fcatk[\Gamma]} ((t^*) ^{\otimes a}, -)$ is an additive functor, there is a natural isomorphism
\[
\hom_{\fcatk[\Gamma]} ((t^*) ^{\otimes a}, \fbcmnd G)
\cong 
\fbcmnd \hom_{\fcatk[\Gamma]} ((t^*) ^{\otimes a}, G),
\]
since the $\Gamma$-module and $\fb\op$-structures commute. This extends to give  an isomorphism of augmented simplicial objects:
\[
\hom_{\fcatk[\Gamma]} ((t^*)^{\otimes a}, \big((\fbcmnd)^{\bullet+1} \twist\big)\op)
\cong 
\big( (\fbcmnd)^{\bullet+1} \hom_{\fcatk[\Gamma]} ((t^*)^{\otimes a},  \twist)\big)\op.
\]

The isomorphism of the statement follows from Proposition \ref{prop:iso-FB-mod}.
\end{proof}

By Proposition \ref{Prop:augmented}, the augmented simplicial structure on 
  $\hom_{\fcatk[\Gamma]} ((t^*) ^{\otimes a},  {\gcmnd}^{\bullet} \twist  )$ described in Proposition \ref{cor:iso_aug_simp} extends to a cosimplicial structure, where the additional structure arises from the two morphisms $d^0, d^1$ of $\fcatk[\fb\op]$ given in (\ref{equalizer}). Composing with the isomorphisms  $\overline{\alpha}^{\ell+1}  : (\gcmnd)^{\ell+1} \twist \stackrel{\cong}{\rightarrow} (\fbcmnd)^{\ell+1} \twist$, gives  a cosimplicial structure  on  $\hom_{\fcatk[\Gamma]} ((t^*) ^{\otimes a},  {\fbcmnd}^{\bullet} \twist  )$. For example, in cosimplicial degrees $0$ and $1$ this cosimplicial structure gives the two morphisms of $\fcatk[\fb\op]$:
\begin{equation*}
\xymatrix{
\hom_{\fcatk[\Gamma]}  ((t^*)^{\otimes a},(t^*)^{\otimes b} )
 \ar@<.5ex>[r]^- {d^0}
 \ar@<-.5ex>[r]_-{d^1}
 &
\hom_{\fcatk[\Gamma]}  ((t^*)^{\otimes a}, \fbcmnd(\twist)(\mathbf{b}) )
}
\end{equation*}
where, for $f \in \hom_{\fcatk[\Gamma]}  ((t^*)^{\otimes a},(t^*)^{\otimes b}) $, $d^0 f := \overline{\alpha} \circ (\gcmnd  f) \circ \psi_{(t^*)^{\otimes a}} $ and $d^1 f := \overline{\alpha} \circ \psi_{(t^*)^{\otimes b}} \circ  f$. 

We begin by proving that the cosimplicial objects $\hom_{\fcatk[\Gamma]}  ((t^*)^{\otimes a}, \fbcmnd^\bullet \twist )$ and $[\cosimp^\bullet \kring \hom_\Oz (-, \mathbf{a})]\op$ are isomorphic in cosimplicial degrees $0$ and $1$ using the maps $\dk_{\mathbf{a},\mathbf{b}}$ introduced in Proposition \ref{prop:explicit_DK_morphism}.

\begin{prop}
\label{prop:iso-cosimp}
For $a \in \nat$, we have the following commutative diagrams in $\fcatk[\fb\op]$:
$$\xymatrix{
\hom_{\fcatk[\Gamma]}  ((t^*)^{\otimes a}, \twist) \ar[rr]^-{d^1} & &\hom_{\fcatk[\Gamma]}  ((t^*)^{\otimes a},  \fbcmnd \twist) \cong \fbcmnd \hom_{\fcatk[\Gamma]}  ((t^*)^{\otimes a}, \twist)\\
 \kring\hom_\Oz(-,\mathbf{a}) \ar[u]^{\kring\dk_{\mathbf{a},-}}_{\cong} \ar[rr]_-{d^0=\eta_{\kring\hom_\Oz(-,\mathbf{a}) }}& &\fbcmnd  \kring\hom_\Oz(-,\mathbf{a})  \ar[u]_{\fbcmnd \kring\dk_{\mathbf{a},-}}^{\cong};
}$$

$$\xymatrix{
\hom_{\fcatk[\Gamma]}  ((t^*)^{\otimes a}, \twist) \ar[rr]^-{d^0} & &\hom_{\fcatk[\Gamma]}  ((t^*)^{\otimes a},  \fbcmnd \twist) \cong \fbcmnd \hom_{\fcatk[\Gamma]}  ((t^*)^{\otimes a}, \twist)\\
\kring\hom_\Oz(-,\mathbf{a}) \ar[u]^{\kring\dk_{\mathbf{a},-}}_{\cong} \ar[rr]_-{d^1=\psi_{\kring\hom_\Oz(-,\mathbf{a}) }}& &\fbcmnd  \kring\hom_\Oz(-,\mathbf{a})  \ar[u]_{\fbcmnd \kring\dk_{\mathbf{a},-}}^{\cong}.
}$$

\end{prop}
\begin{proof}
It suffices to prove that the diagrams are commutative for $\mathbf{b} \in \ob \fb\op$. 

The case $a=0$ is exceptional and can be treated directly; the details are left to the reader.

Let us treat the case $a=1$. In this case, for $\mathbf{b}=\mathbf{0}$ there is nothing to prove, so we may assume that $b>0$. By Definition \ref{defn:dk_map},  $\kring \dk_{\mathbf{1}, \mathbf{b}}$ sends the unique surjection $\mathbf{b} \twoheadrightarrow \mathbf{1}$ to the standard generator $\xi_b: t^* \rightarrow (t^*) ^{\otimes b}$.
Using Proposition \ref{prop:calculate_perp} (\ref{prop:calculate_perp-2}) and Proposition \ref{prop:compare_comonads}
$$d^1(\xi_b)=\overline{\alpha}_\mathbf{b} \circ \psi_{(t^*)^{\otimes b}} \circ  \xi_b: t^* \to \fbcmnd (t^*) ^{\otimes b} \cong  \bigoplus_{\mathbf{b}' \subseteq \mathbf{b}} (t^*)^{\otimes b'}$$
is equal to 
$$t^* \xrightarrow{\xi_b} (t^*) ^{\otimes b}\xrightarrow{i_b}\bigoplus_{\mathbf{b}' \subseteq \mathbf{b}} (t^*)^{\otimes b'}$$
where $i_b$ is the inclusion to the factor indexed by $\mathbf{b}' = \mathbf{b}$. By Proposition  \ref{prop:fbcmnd_coaug} and the functoriality of $\fbcmnd$, this is equal to the other composition in the first commutative diagram.

For the second commutative diagram, using Lemma \ref{lem:gcmnd_t*} and Proposition \ref{prop:perpfb} the map
$$d^0(\xi_b)=\overline{\alpha} \circ (\gcmnd  \xi_b) \circ \psi_{t^*}: t^* \to \fbcmnd (t^*) ^{\otimes b} \cong  \bigoplus_{\mathbf{b}' \subseteq \mathbf{b}} (t^*)^{\otimes b'}$$
 has component 
 $ t^* \rightarrow (t^*)^{\otimes |\mathbf{b}'|}$, the standard generator $\xi_{ |\mathbf{b}'|}$ for any $\emptyset \neq \mathbf{b}' \subseteq \mathbf{b}$. By Proposition \ref{prop:fi_comod} and the functoriality of $\fbcmnd$ this is equal to the other composition in the second commutative diagram.

 The case $a > 1$ can be deduced from the case $a=1$ using the $\sym_b \op$-equivariant isomorphism given in Lemma \ref{lem:mor_Oz} as follows. 
 
As in the Lemma, for $a, b  \in \nat^*$ there is a $\sym_b\op$-equivariant isomorphism 
\[
\bigoplus_{\substack { (b_i)\in (\nat^*)^{a} \\ \sum_{i=1}^a b_i = b }}
\big( \bigotimes_{i=1}^a  \hom_{\fcatk[\Gamma]} (t^*, (t^*)^{\otimes b_i})\big)
 \uparrow_{\prod_{i=1}^a \sym_{b_i}}^{\sym_b} 
\stackrel{\cong}{\rightarrow }
\hom_{\fcatk[\Gamma]} ((t^*)^{\otimes a }, (t^*)^{\otimes b} ).
\]
 This is compatible with the isomorphism of Lemma \ref{lem:mor_Oz} via the isomorphism of Proposition \ref{prop:explicit_DK_morphism}.
 
Hence, for $a>1$, it suffices to check the commutativity of the diagrams restricted to the following submodules
that are compatible via $\kring  \dk_{\mathbf{a}, \mathbf{b}}$: 
\begin{eqnarray}
\label{eqn:restrict_tensor_fcatkGamma}
\nonumber
\bigotimes_{i=1}^a  \kring \hom_{\Oz} (\mathbf{b_i},\mathbf{1})
&\subset &
\kring \hom_{\Oz} (\mathbf{b}, \mathbf{a}) 
\\
\bigotimes_{i=1}^a  \hom_{\fcatk[\Gamma]} (t^*, (t^*)^{\otimes b_i})
&\subset&
\hom_{\fcatk[\Gamma]} ((t^*)^{\otimes a }, (t^*)^{\otimes b} )
\end{eqnarray}
 for each sequence $(b_i)$. Here, the inclusion in (\ref{eqn:restrict_tensor_fcatkGamma}) is given by the tensor product of morphisms. 

The upper horizontal morphisms $d^0$, $d^1$ in the statement of the Proposition are defined using $\gcmnd$, which is symmetric monoidal, by Corollary \ref{cor:perp_sym}. Using this, their restriction to the submodule in (\ref{eqn:restrict_tensor_fcatkGamma}) can be calculated directly from the case $a=1$. The result follows by direct verification. 
 \end{proof}
 
\begin{proof}[Proof of Theorem \ref{thm:identify_cosimp_gcmnd_fbcmnd}]
We require to establish   the isomorphism of cosimplicial $\fb\op$-modules:
\[
\hom_{\fcatk[\Gamma]}  ((t^*)^{\otimes a}, \gcmnd^\bullet \twist )
\cong 
\cosimp^\bullet \kring \hom_\Oz (-, \mathbf{a})  \op.
\]
Proposition \ref{cor:iso_aug_simp} has already established that the underlying augmented simplicial objects are isomorphic. It remains to check that the first and last coface operators correspond. 

In cosimplicial degree $0$ and $1$ the isomorphism is given by Proposition \ref{prop:iso-cosimp}. The higher cosimplicial degrees are deduced from this case, proceeding as in the passage to higher simplicial degrees in the proof of Proposition \ref{prop:compare_aug_simp_comonads}. 
\end{proof}

\begin{rem}
The result of Theorem \ref{thm:identify_cosimp_gcmnd_fbcmnd} is stronger than stated, since it is also equivariant with respect to automorphisms of $\mathbf{a}$. Namely, there is an isomorphism of cosimplicial  $\fb \times \fb\op$-modules:
\[
\hom_{\fcatk[\Gamma]}  (\twist , \gcmnd^\bullet \twist )
\cong 
\cosimp^\bullet \kring \hom_\Oz (-, -)  \op.
\]
Here, it is important to recall that $\twist$ is considered as an object of $\fcatk[\Gamma \times \fb\op]$; this allows the two variances to be distinguished correctly.
\end{rem}

\section{The cohomology of the Koszul complex of $\kring \hom_\Oz (-, \mathbf{a}) $}
\label{sect:cohomology}

The purpose of this section is to  calculate, for $a \in \nat$, the cohomology of the Koszul complex of the $\finj\op$-module $\kring \hom_\Oz (-, \mathbf{a})$ introduced  in Proposition \ref{prop:hom_oz_functor_finjop}.
Evaluated upon $\mathbf{b}$, this calculation   is achieved in Theorem \ref{thm:cohomology_NC}, with information on the action of the symmetric groups given in Theorem \ref{thm:cohom_equivariance}. 

By Corollary \ref{cor:relate_finjop_cohom} and Theorem \ref{thm:motivate}, the cohomology of    this Koszul complex gives access to the computation of $\hom_{\fcatk[\finne]} (\kbar^{\otimes a},\kbar^{\otimes b})$ and $\ext^1_{\fcatk[\finne]} (\kbar^{\otimes a},\kbar^{\otimes b})$ (see Theorem \ref{thm:complex_Cab}) which are the main objects of interest of this paper.

For typographical simplicity we introduce the following notation:

\begin{nota}
\label{nota:Kz_Cba}
For $a \in \nat$, let $\cbb (- , \mathbf{a})$ denote the complex $\kz \kring \hom_\Oz (-, \mathbf{a})$ of $\fb\op$-modules,  so that 
$
\cbb (\mathbf{b}, \mathbf{a}) 
= 
\kz \kring \hom_\Oz (-, \mathbf{a}) (\mathbf{b})
$. Explicitly, for $t \in \nat$:
 $$\cbb (\mathbf{b} , \mathbf{a})^t =\bigoplus_{\substack{ \mathbf{b}^{(t)}\subset \mathbf{b} 
\\
|\mathbf{b}^{(t)}| = b-t
}}
\kring \hom _\Oz (\mathbf{b}^{(t)}, \mathbf{a} )
\otimes \orient (\mathbf{b} \backslash \mathbf{b}^{(t)} ).
 $$
 where $\orient$ is given in Notation \ref{nota:orient} and the differential $d: \cbb (\mathbf{b} , \mathbf{a})^t \to \cbb (\mathbf{b} , \mathbf{a})^{t+1}$ is the morphism of $\kring$-modules given on  $f \otimes \alpha$, where $f \in \hom _\Oz (\mathbf{b}^{(t)}, \mathbf{a} )$ for $\mathbf{b}^{(t)}\subset \mathbf{b}$ such that $|\mathbf{b}^{(t)}| = b-t$
 and $\alpha \in \orient (\mathbf{b}\backslash \mathbf{b}^{(t)})$, by 
\[
d \big ([f] \otimes \alpha  \big) 
= 
\sum_{\substack{y \in \mathbf{b}^{(t)} \\ f \circ \iota_y \in \hom _\Oz(\mathbf{b}^{(t)} \backslash \{ y\}, \mathbf{a})}} 
[f \circ \iota_y]
\otimes (y \wedge \alpha), 
\]
where $\iota_y : \mathbf{b}^{(t)} \backslash \{ y\}
\hookrightarrow \mathbf{b}^{(t)}$ is the inclusion. 
\end{nota}

\begin{rem}
By Proposition \ref{Koszul-exact}, the construction $\kz$ is functorial. This implies that $\cbb( -, -)$ is a complex  of $\fb \times \fb\op$-modules, since the symmetric group $\sym_a$ acts on $\kring \hom_\Oz(-, \mathbf{a})$ by automorphisms of $\finj\op$-modules. 
\end{rem}

 First we record the following exceptional cases:

\begin{lem}
\label{lem:case_a=b}
For $a \in \nat$, 
\begin{enumerate}
\item
$\cbb (\mathbf{b}, \mathbf{a}) =0$ for $b <a$;
\item 
$\cbb (\mathbf{b}, \mathbf{0})\cong \orient (\mathbf{b})$, concentrated in cohomological degree $b$, which thus coincides with the cohomology;
\item 
 $\mathbb{C}(\mathbf{a}, \mathbf{a}) \cong \kring[\sym_a]$, concentrated in cohomological degree zero, which thus coincides with the cohomology. 
\end{enumerate}
\end{lem}

This allows us to concentrate on the case $b> a>0$ henceforth. 

\subsection{The case $a=1$}
This forms one of the initial steps of a double induction in the proof of Theorem \ref{thm:cohomology_NC}.

\begin{prop}
\label{prop:case_a=1_b>1}
For $a=1$ and  $b \in \nat^*$,
\[
H^* (\mathbb{C}(\mathbf{b}, \mathbf{1})) \cong 
\left\{
\begin{array}{ll}
\orient (\mathbf{b}) & *= b-1 \\
0 & \mbox{ otherwise.}
\end{array}
\right.
\]  
\end{prop}

\begin{proof}
In the case $b=1$, the result follows from Lemma \ref{lem:case_a=b}, hence suppose that $b>1$.

There is a short exact sequence of $\finj\op$-modules 
\[
0
\rightarrow 
\mathbbm{1}
\rightarrow 
\kring 
\rightarrow 
\kring \hom_\Oz (-, \mathbf{1}) 
\rightarrow 
0,
\]
where $\kring$ denotes the constant $\finj\op$-module and $\mathbbm{1}$ is as in Notation \ref{notation-1}. By Proposition \ref{Koszul-exact} this gives rise to a short exact sequence of cochain complexes in $\fcatk[\fb\op]$

\[
0
\rightarrow 
\kz \mathbbm{1}
\rightarrow 
\kz \kring 
\rightarrow 
\kz \kring \hom_\Oz (-, \mathbf{1}) 
\rightarrow 
0
\]
which induces a long exact sequence in cohomology. 

Since $\kz \mathbbm{1}=\cbb (-, \mathbf{0})$, by Lemma \ref{lem:case_a=b}, for $b \in \nat^*$, $(\kz \mathbbm{1})(\mathbf{b})$ has cohomology concentrated in cohomological degree $b$, where it is isomorphic to $\orient (\mathbf{b})$. 

We prove that $\kz \kring$ has zero cohomology.   This is proved as for classical Koszul complexes by exhibiting a chain null-homotopy as follows. Firstly, $(\kz \kring )^t(\mathbf{b})$ identifies as 
\[
\bigoplus_{\substack{\mathbf{b}'' \subset \mathbf{b} \\ |\mathbf{b}''| = t}}
\orient (\mathbf{b}''),
\]
noting that this includes the module $\kring$ in degree $t=b$.

Define a chain null-homotopy $h: (\kz \kring )^{t}(\mathbf{b}) \to (\kz \kring )^{t-1}(\mathbf{b})$ using the generators introduced in Notation \ref{nota:orient} by 
\[
\omega (\mathbf{b}'') \mapsto 
\left\{
\begin{array}{ll}
0 & 1 \not \in \mathbf{b}''\\
\omega (\mathbf{b}'' \backslash \{ 1 \})& 1 \in \mathbf{b}'',
\end{array}
\right.
\]
where $\mathbf{b}'' \subset \mathbf{b}$ such that  $|\mathbf{b}''| = t$.

Note that, for $\mathbf{b}'' \subset \mathbf{b}$ such that $1 \in \mathbf{b}''$ and $x \in \mathbf{b} \setminus \mathbf{b}''$, we have
$$
h(1 \wedge x \wedge \omega(\mathbf{b}''  \setminus \{1\}))=x \wedge \omega(\mathbf{b}''  \setminus \{1\}).
$$

The following calculations show that  $dh + hd = \mathrm{id}$, as required:
\begin{enumerate}
\item 
If $1 \not\in \mathbf{b}''$, $d \circ h(\omega( \mathbf{b}''))=d(0)=0$ and 
$$
h \circ d(\omega( \mathbf{b}''))=h(\sum_{\substack{ x \in \mathbf{b} \setminus\mathbf{b}'' }} x \wedge \omega( \mathbf{b}''))=\sum_{\substack{ x \in \mathbf{b} \setminus\mathbf{b}'' }} h(x \wedge \omega( \mathbf{b}''))=h(1 \wedge\omega( \mathbf{b}''))= h(\omega ( \mathbf{b}'' \cup\{1 \}))= \omega (\mathbf{b}'' ). 
$$
\item 
If $1 \in \mathbf{b}''$,
\begin{eqnarray*}
d \circ h(\omega( \mathbf{b}''))&=&d(\omega(\mathbf{b}'' \setminus \{1\}))
= \sum_{\substack{ x \in (\mathbf{b} \setminus\mathbf{b}'') \cup \{1\} }} x \wedge \omega(\mathbf{b}'' \setminus \{1\})\\
&=& 1 \wedge \omega(\mathbf{b}'' \setminus \{1\}) + \sum_{\substack{ x \in \mathbf{b} \setminus\mathbf{b}'' }} x \wedge \omega(\mathbf{b}'' \setminus \{1\})= \omega(\mathbf{b}'') + \sum_{\substack{ x \in \mathbf{b} \setminus\mathbf{b}'' }} x \wedge \omega(\mathbf{b}'' \setminus \{1\})
\end{eqnarray*}
and 
\begin{eqnarray*}
h \circ d(\omega( \mathbf{b}''))&=& h(\sum_{\substack{ x \in \mathbf{b} \setminus\mathbf{b}'' }} x \wedge \omega(\mathbf{b}''))=\sum_{\substack{ x \in \mathbf{b} \setminus\mathbf{b}'' }} h(x \wedge \omega(\mathbf{b}''))\\
&=&\sum_{\substack{ x \in \mathbf{b} \setminus\mathbf{b}'' }} h(x \wedge 1 \wedge  \omega(\mathbf{b}'' \setminus \{1\})) = \sum_{\substack{ x \in \mathbf{b} \setminus\mathbf{b}'' }} -h(1 \wedge x \wedge  \omega(\mathbf{b}'' \setminus \{1\}))\\
&=&\sum_{\substack{ x \in \mathbf{b} \setminus\mathbf{b}'' }} -x \wedge  \omega(\mathbf{b}'' \setminus \{1\}).
\end{eqnarray*}
\end{enumerate} 
\end{proof}

\subsection{The top cohomology of $\mathbb{C}(\mathbf{b}, \mathbf{a})$}

The other ingredient to the double induction used in the proof of Theorem \ref{thm:cohomology_NC} is the identification of the cohomology in  the top degree $b-a$:

\begin{prop}
\label{prop:top_degree}
If $b>a>0$, then there is a $\sym_a \times \sym_b\op$-equivariant surjection of chain complexes 
\begin{eqnarray}
\label{eqn:cbb_surjection}
\cbb (\mathbf{b}, \mathbf{a}) 
\twoheadrightarrow 
\orient (\mathbf{a}) \otimes \orient (\mathbf{b}), 
\end{eqnarray}
where $\orient (\mathbf{a}) \otimes \orient (\mathbf{b})$ is placed in cohomological degree $b-a$. 

This induces an isomorphism of $\sym_a \times \sym_b\op$-modules:
$$
H^{b-a} (\mathbb{C}(\mathbf{b}, \mathbf{a})) \cong \orient (\mathbf{a}) \otimes\orient (\mathbf{b}) .
$$
In particular, for $\mathbf{b}' \subset \mathbf{b}$ such that $|\mathbf{b}'|=a-1$, $(u,v) \in  (\mathbf{b}\setminus \mathbf{b}')^2 $, $f_1 \in \hom _\Oz (\mathbf{b}' \cup \{u\}, \mathbf{a} )$,  $f_2 \in \hom _\Oz (\mathbf{b}' \cup \{v\}, \mathbf{a} )$ such that $f_1| _{\mathbf{b}'}=f_2|_{ \mathbf{b}'} $ and $f_1(u)=f_2(v)$ then the cohomology classes of the generators 
$[f_1] \otimes \omega_1 $ and $[f_2] \otimes \omega_2$ (where $\omega_1$, $\omega_2$ denote the respective canonical orientation classes) are equal, up to a possible sign $\pm 1$.  
\end{prop}

\begin{proof}
The result is proved by analysing the tail of the complex $\mathbb{C}(\mathbf{b}, \mathbf{a})$:
\[
\xymatrix{
\mathbb{C}(\mathbf{b}, \mathbf{a})^{b-a-1} 
\ar[r]^-d
\ar@{=}[d]
& 
\mathbb{C}(\mathbf{b}, \mathbf{a}) ^{b-a}
\ar@{=}[d]
\\
\bigoplus_{\substack{ \tilde{\mathbf{b}}\subset \mathbf{b} 
\\
|\tilde{\mathbf{b}}| = a+1
}}
\kring \hom _\Oz (\tilde{\mathbf{b}}, \mathbf{a} ) 
\otimes 
\orient (\mathbf{b} \backslash \tilde{\mathbf{b}} ) 
\ar[r]
&
\bigoplus_{\substack{ \mathbf{b}'\subset \mathbf{b} 
\\
|\mathbf{b}'| = a
}}
\kring \hom _\Oz (\mathbf{b}', \mathbf{a} )
\otimes \orient (\mathbf{b} \backslash \mathbf{b}' ), 
}
\]

The surjection (\ref{eqn:cbb_surjection}) is defined by the $\kring$-module morphism:
\[
\psi: \kring \hom _\Oz (\mathbf{b}', \mathbf{a} )
\otimes \orient (\mathbf{b} \backslash \mathbf{b}' )
\rightarrow 
\orient (\mathbf{a}) 
\otimes 
\orient (\mathbf{b}) 
\]
that sends a generator $[f] \otimes \omega (\mathbf{b}\backslash \mathbf{b}')$,
 where $f$ is a bijection $f: \mathbf{b}' \stackrel{\cong}{\rightarrow} \mathbf{a}$, to 
 $$
\omega (\mathbf{a}) \otimes \big( f^{-1}(1) \wedge \ldots \wedge f^{-1} (a) \wedge \omega (\mathbf{b}\backslash \mathbf{b}')\big ) \in \orient (\mathbf{a})\otimes   \orient (\mathbf{b}) .
 $$

To see that this sends boundaries to zero, consider a generator $[g] \otimes \omega (\mathbf{b}\backslash \tilde {\mathbf{b}})$, where $g: \tilde{\mathbf{b}} \twoheadrightarrow \mathbf{a}$ with $|\tilde{\mathbf{b}} | = a+1$. For such a $g$, there exists a unique $i \in \mathbf{a}$ such that  $|g^{-1} (i)|=2$, with  all other fibres of cardinal one; write $g^{-1} (i) = \{u, v \}$. Then:
\[
d ([g]  \otimes \omega (\mathbf{b}\backslash \tilde {\mathbf{b}}))
= 
 [g|_{\tilde{\mathbf{b}} \backslash \{v \}}] \otimes (v \wedge   \omega (\mathbf{b}\backslash \tilde {\mathbf{b}}))
+ 
 [g|_{\tilde{\mathbf{b}} \backslash \{ u \}}] \otimes (u \wedge   \omega (\mathbf{b}\backslash \tilde {\mathbf{b}})).
\]
Now the images under $\psi$ of the two terms of the right hand side differ only by the transposition of the terms $u$, $v$ in the wedge product in $\orient (\mathbf{b})$. The sum of these images is therefore zero in $\orient (\mathbf{a})\otimes \orient (\mathbf{b})$. 

In particular, for $\mathbf{b}' \subset \mathbf{b}$ such that $|\mathbf{b}'|=a-1$ and $u, v, f_1, f_2$ as in the statement, taking $\tilde{\mathbf{b}}=\mathbf{b}'\cup\{u,v\}$ and $g \in  \hom _\Oz (\tilde{\mathbf{b}}, \mathbf{a} )$ such that $g(u)=g(v)$ and $g|_{\tilde{\mathbf{b}} \backslash \{u,v \}}=f_1|_{\mathbf{b}'}=f_2|_{\mathbf{b}'}$ we have
$$
d ([g]  \otimes \omega (\mathbf{b}\backslash \tilde {\mathbf{b}}))
= 
 [f_1] \otimes   (v \wedge   \omega (\mathbf{b}' \cup \{u\})) 
+ 
 [f_2] \otimes (u \wedge   \omega (\mathbf{b}' \cup \{v\}))=0.
$$
Thus the cohomology classes of the generators 
$[f_1] \otimes \omega_1 $ and $[f_2] \otimes \omega_2$ are equal, up to a possible sign $\pm 1$.  

One deduces that the cohomology $H^{b-a} (\cbb (\mathbf{b}, \mathbf{a}))$ is generated by the class of the generator $[\id_{\mathbf{a}}] \otimes \omega (\mathbf{b}\backslash \mathbf{a})$, where $\mathbf{a}$ is considered as a subset of $\mathbf{b}$ by the canonical inclusion. This element is sent to the generator $\omega (\mathbf{a}) \otimes \omega (\mathbf{b})$ of $\orient (\mathbf{a}) \otimes \orient (\mathbf{b})$ by the surjection. It follows that   the surjection induces an isomorphism in cohomology.

Finally, by construction, the surjection of chain complexes is $\sym_a \times \sym_b\op$-equivariant, hence the induced isomorphism in cohomology is one of $\sym_a \times\sym_b\op$-modules.
\end{proof}

\subsection{The cohomology of $\mathbb{C}(\mathbf{b}, \mathbf{a})$}
\label{subsect:thm_cohom}

We begin by considering the case $\mathbf{b}=\mathbf{a+1}$.

\begin{prop}
\label{cor:cohomology_b=a+1}
For $a \in \nat$, 
\[
H^*(\mathbb{C}(\mathbf{a+1}, \mathbf{a}))
\cong 
\left\{
\begin{array}{ll}
\kring ^{\oplus \epfn (a+1, a)} 
& *=0 \\
\kring
& *=1 \\
0 & *>1, 
\end{array}
\right.
\]
where 
\[
\epfn (a+1, a) = 
\left\{
\begin{array}{ll}
0 & a \in \{0,1\} \\
(a-2) \frac{(a+1)!}{2} +1  & a >1.
\end{array}
\right.
\]
In particular, the cohomology is $\kring$-free, concentrated in degrees $0$ and $1$.
\end{prop}

\begin{proof}
The case $a=0$ is given by Lemma \ref{lem:case_a=b}.

For $a>0$, by Proposition \ref{prop:top_degree},  $H^1  (\mathbb{C}(\mathbf{a+1}, \mathbf{a})) \cong \kring$ 
and the complex $\mathbb{C} (\mathbf{a+1}, \mathbf{a})$ is concentrated in degrees $0$ and $1$, with terms that are free $\kring$ modules with:
\begin{eqnarray*}
\krank \mathbb{C}(\mathbf{a+1}, \mathbf{a})^0 &=& a! \binom{a+1}{2} = \frac{a}{2} (a+1)! 
\\
\krank \mathbb{C}(\mathbf{a+1}, \mathbf{a})^1 &= &a! (a+1) = (a+1)!,
\end{eqnarray*}
by direct calculation. The result follows by observing that the image of the differential is a free $\kring$-module of rank $ \krank \mathbb{C}(\mathbf{a+1}, \mathbf{a})^1 - 1$ and hence deducing that the kernel of the differential is a free $\kring$-module of rank $\krank \mathbb{C}(\mathbf{a+1}, \mathbf{a})^0 -  \krank \mathbb{C}(\mathbf{a+1}, \mathbf{a})^1 +1$. (This can be viewed as an Euler-Poincar\'e characteristic argument.)
\end{proof}

In order to compute the cohomology of $\mathbb{C}(\mathbf{b}, \mathbf{a})$ for $\mathbf{b} \geq \mathbf{a+2}$ we will consider several subcomplexes of $\mathbb{C}(\mathbf{b}, \mathbf{a})$. To construct these subcomplexes, we consider particular families of subsets of $\hom_\Oz (\mathbf{b}', \mathbf{a})$, for $\mathbf{b}' \subseteq \mathbf{b}$, which we will refer to as families of subsets of  $\hom_\Oz (-, \mathbf{a})$, leaving the restriction to subsets of $\mathbf{b}$ implicit.

\begin{defn}
\label{def:subcx_generation}
 Let  $ X (\mathbf{b}') \subseteq \hom_\Oz (\mathbf{b}', \mathbf{a})$, for $\mathbf{b}' \subseteq \mathbf{b}$, be  a family of subsets of $\hom_\Oz (-, \mathbf{a})$, denoted simply by $X$. We say that $X$ is stable under restriction when, if $f \in X(\mathbf{b}')$ and $f|_{\mathbf{b}''}$ is surjective for $\mathbf{b}'' \subset \mathbf{b}'$, then $f|_{\mathbf{b}''} \in X(\mathbf{b}'')$. 
\end{defn}

\begin{nota}
For $b>0$, denote the element $ b \in \mathbf{b}$ by $x$; this choice ensures that $\mathbf{b} \backslash \{ x \} = \mathbf{b-1}$.
\end{nota}

\begin{exam}
\label{exemple-famille}
The following families of subsets of $\hom_\Oz (- , \mathbf{a})$ are stable under restriction.
\begin{enumerate}
\item The family $X^1$ such that 
\[
X^1(\mathbf{b}') =
\left\{
\begin{array}{ll}
\hom_\Oz (\mathbf{b}', \mathbf{a}) & \textrm{if}\  x \not\in \mathbf{b}',\\
\emptyset & \textrm{if} \ x\in \mathbf{b}'.
\end{array}
\right.
\]
\item For $y \in \mathbf{a}$, the family $\widetilde{X}_y^2$ such that
\[
\widetilde{X}_y^2(\mathbf{b}') =
\left\{
\begin{array}{ll}
\hom_\Oz (\mathbf{b}', \mathbf{a}) & \textrm{if}\  x \not\in \mathbf{b}',\\
\{ f \in \hom_\Oz (\mathbf{b}', \mathbf{a})| f(x)=y \} & \textrm{if} \ x\in \mathbf{b}'.
\end{array}
\right.
\]
\item 
For $y \in \mathbf{a}$, the family $\widetilde{X}_{[1,y]}^2$ such that
\[
\widetilde{X}_{[1,y]}^2(\mathbf{b}') =
\left\{
\begin{array}{ll}
\hom_\Oz (\mathbf{b}', \mathbf{a}) & \textrm{if}\  x \not\in \mathbf{b}',\\
\{ f \in \hom_\Oz (\mathbf{b}', \mathbf{a})| 1 \leq f(x) \leq y \} & \textrm{if} \ x\in \mathbf{b}'.
\end{array}
\right.
\]
\item 
For $y \in \mathbf{a}$, the family $X_{y}^2$ such that
\[
X_{y}^2(\mathbf{b}') =
\left\{
\begin{array}{ll}
\emptyset & \textrm{if}\  x \not\in \mathbf{b}',\\
\{ f \in \hom_\Oz (\mathbf{b}', \mathbf{a})| f^{-1}(y)=\{x\} \}& \textrm{if} \ x\in \mathbf{b}'.
\end{array}
\right.
\]
\end{enumerate}
By construction, there are  inclusions of families of subsets
$$\begin{array}{ccc}
X^2_y \subset& \widetilde{X^2_y} &\subset  \hom_\Oz (-, \mathbf{a})\\
 & \cup& \\
  & X^1
\end{array}$$
and
$$\widetilde{X}^2_1=\widetilde{X}^2_{[1,1]}\subset \widetilde{X}^2_{[1,2]}\subset \ldots \subset  \widetilde{X}^2_{[1,a-1]}\subset \widetilde{X}^2_{[1,a]}=\hom_\Oz (-, \mathbf{a}).$$
\end{exam}

A family of subsets of $\hom_\Oz (-, \mathbf{a})$ stable under restriction gives rise to a subcomplex of $\mathbb{C}(\mathbf{b}, \mathbf{a})$:

\begin{lem}
\label{lem:subcx_generation}
If  $ X$ is a family  of subsets of $\hom_\Oz (- , \mathbf{a})$ stable under restriction, then the free $\kring$-submodules, for $t \in \nat$,  
$$
\bigoplus_{\substack{ \mathbf{b}^{(t)}\subset \mathbf{b} 
\\
|\mathbf{b}^{(t)}| = b-t
}}
\kring X (\mathbf{b}^{(t)})
\otimes \orient (\mathbf{b} \backslash \mathbf{b}^{(t)} )
\subset 
\cbb (\mathbf{b}, \mathbf{a})^t
 $$
form a subcomplex of $\cbb (\mathbf{b}, \mathbf{a})$, called the subcomplex {\em generated} by $X$.
\end{lem}

\begin{proof}
This is clear from the definition of the differential of the complex $\cbb (\mathbf{b}, \mathbf{a})$ that is  recalled in Notation \ref{nota:Kz_Cba}.
\end{proof}

Applying this lemma to the families introduced in Example \ref{exemple-famille} we obtain the following subcomplexes of $\mathbb{C}:= \mathbb{C} (\mathbf{b}, \mathbf{a})$.

\begin{defn}
\label{defn:subcomplexes}
Suppose that $b>a>0$ and set $x:=b \in \mathbf{b}$. 
\begin{enumerate}
\item 
Let $\mathbb{D}$ be the subcomplex of $\mathbb{C}$ generated by $X^1$. 
\item 
For $y \in \mathbf{a}$, let
\begin{enumerate}
\item 
$\widetilde{\mathbb{C}_y}$ be the subcomplex of $\mathbb{C}$ generated by $\widetilde{X}_y^2$;
 \item 
$ \filt_{y} \cbb$  be the subcomplex of $\mathbb{C}$ generated by $\widetilde{X}_{[1,y]}^2$;
 \item 
$\mathbb{C}_y$  be the subcomplex of $\mathbb{C}$ generated by $X^2_y$. 
\end{enumerate} 
\end{enumerate}
\end{defn}

The inclusions of families of subsets give the following inclusions of complexes
$$\begin{array}{ccc}
\mathbb{C}_y \subset& \widetilde{\mathbb{C}_y} &\subset  \mathbb{C}\\
 & \cup& \\
  & \mathbb{D}
\end{array}$$
and the increasing filtration of subcomplexes
\begin{eqnarray}
\label{filtration}
\widetilde {\mathbb{C}_1}= 
\filt_1 \cbb
\subset 
\filt_2 \cbb 
\subset 
\ldots 
\subset 
\filt_{a-1} \cbb
\subset 
\filt_{a} \cbb = \cbb.
\end{eqnarray}

These complexes are related by the following Lemma.

\begin{lem}
\label{lem:decompose_C}
For $a>0$ and $y \in \mathbf{a}$, there is a short exact sequence of complexes:
\begin{eqnarray}
\label{eqn:ses_D_filt}
0
\rightarrow 
\mathbb{D} 
\rightarrow 
\filt_y \cbb \oplus \widetilde{\cbb_{y+1}}
\rightarrow 
\filt_{y+1} \cbb
\rightarrow 
0.
\end{eqnarray}
\end{lem}

\begin{proof}
The inclusions $\mathbb{D} \hookrightarrow \filt_y \cbb$ and $\mathbb{D} \hookrightarrow \widetilde{\cbb_{y+1}}$  induce short exact of complexes:
$$0 \to \mathbb{D} \to \filt_y \cbb \to  \filt_y \cbb/ \mathbb{D}  \to 0; \qquad 0 \to \mathbb{D} \to \widetilde{\cbb_{y+1}} \to  \widetilde{\cbb_{y+1}} /\mathbb{D} \to  0.$$
Considering the pushout of $\mathbb{D} \hookrightarrow \filt_y \cbb$ and $\mathbb{D} \hookrightarrow \widetilde{\cbb_{y+1}}$ in the category of chain complexes we obtain a short exact sequence of complexes 
$$
0
\rightarrow 
\mathbb{D} 
\rightarrow 
\filt_y \cbb \oplus \widetilde{\cbb_{y+1}}
\rightarrow 
\filt_y \cbb \cup_{\mathbb{D}} \widetilde{\cbb_{y+1}}
\rightarrow 
0.
$$
For $t \in \nat$, consider the following $\kring$-modules
$$A^t_y:=
\bigoplus_{\substack{ \mathbf{b}^{(t)}\subset \mathbf{b} 
\\
|\mathbf{b}^{(t)}| = b-t
\\
x \in \mathbf{b}^{(t)}
}}
\kring \{ f \in \hom_\Oz (\mathbf{b}^{(t)}, \mathbf{a})| 1 \leq f(x) \leq y \}
\otimes \orient (\mathbf{b} \backslash \mathbf{b}^{(t)} )
 $$
$$B^t_{y+1}:=
\bigoplus_{\substack{ \mathbf{b}^{(t)}\subset \mathbf{b} 
\\
|\mathbf{b}^{(t)}| = b-t
\\
x \in \mathbf{b}^{(t)}
}}
\kring \{ f \in \hom_\Oz (\mathbf{b}^{(t)}, \mathbf{a})|  f(x)= y+1 \}
\otimes \orient (\mathbf{b} \backslash \mathbf{b}^{(t)} ).
 $$
We have the following isomorphisms of $\kring$-modules
$$ (\filt_y \cbb/ \mathbb{D})^t \cong A^t_y; \qquad (\widetilde{\cbb_{y+1}} /\mathbb{D})^t\cong B^t_y;$$
$$  A^{t}_{y+1} \cong A^t_y \oplus B^t_{y+1}; \qquad (\filt_y \cbb)^t \cong \mathbb{D}^t \oplus A^t_y; \qquad(\widetilde{\cbb_{y+1}})^t\cong \mathbb{D}^t \oplus B^t_{y+1}.$$
We deduce that 
$$(\filt_y \cbb \cup_{\mathbb{D}} \widetilde{\cbb_{y+1}})^t\cong \mathbb{D}^t \oplus A^t_y\oplus B^t_{y+1}\cong \mathbb{D}^t \oplus A^{t}_{y+1} \cong (\filt_{y+1} \cbb)^t.$$
Using the definition of the differential of $\cbb$ recalled in Notation \ref{nota:Kz_Cba} we see that this isomorphism of $\kring$-modules gives an isomorphism of complexes.
\end{proof}

The complexes $\mathbb{D}$ and $\mathbb{C}_y$, for $y \in \mathbf{a}$, are identified in terms of complexes of the form $\cbb (\mathbf{b}', \mathbf{a}')$ as follows, where $[1]$ denotes the shift of cohomological degree (i.e., $\mathbb{C}[1]^t:= \mathbb{C}^{t-1}$ by Notation \ref{nota:shift}):

\begin{lem}
\label{lem:identify_subcomplexes}
For $a, b \in \nat^*$, we have the following isomorphisms of complexes
\begin{enumerate}
\item 
$\mathbb{D} \cong \mathbb{C}(\mathbf{b-1}, \mathbf{a})[1]$; 
\item 
for $y \in \mathbf{a}$, $\mathbb{C}_y \cong  \mathbb{C}(\mathbf{b-1}, \mathbf{a-1})$.  
\end{enumerate}
\end{lem}

\begin{proof}
The isomorphism $\mathbb{D}^t \cong \mathbb{C}(\mathbf{b-1}, \mathbf{a})^{t-1}$ is induced by the obvious map
 noting that 
$$\mathbf{b}^{(t)}\subset \mathbf{b} \textrm{ such that } \ 
|\mathbf{b}^{(t)}| = b-t
 \textrm{ and } x \not\in \mathbf{b}^{(t)}  \iff \mathbf{b}^{(t)}\subset \mathbf{b-1} \textrm{ such that } \ 
|\mathbf{b}^{(t)}| = (b-1)-(t-1)
.$$

The isomorphism $\mathbb{C}_y^t \cong  \mathbb{C}(\mathbf{b-1}, \mathbf{a-1})^t$ is induced by the map $\mathbb{C}_y^t \to  \mathbb{C}(\mathbf{b-1}, \mathbf{a-1})^t$ forgetting the fibre above $y$ together with the bijection of ordered sets 
$\mathbf{a} \setminus\{y \}\cong \mathbf{a-1}$ 

This proves that the underlying graded $\kring$-modules are isomorphic. The choice $x=b$ ensures that the signs appearing in the differentials are compatible with these isomorphisms.
\end{proof}

Crucially, this also allows us to understand the cohomology of $\widetilde{\mathbb{C}_y}$, by the following result:

\begin{prop}
\label{prop:equivalence_cohom}
If $b >a>0$, then 
$j: \mathbb{C}_y \hookrightarrow \widetilde{\mathbb{C}_y}$ is a quasi-isomorphism.
\end{prop}

\begin{proof}
It suffices to check that the cokernel $\widehat{\mathbb{C}_y}$ of $\mathbb{C}_y \hookrightarrow \widetilde{\mathbb{C}_y}$ has zero cohomology.

Using the notation introduced in the proof of Lemma \ref{lem:decompose_C}, we have: $\widetilde{\cbb_{y}}^t\cong \mathbb{D}^t \oplus B^t_{y}$ and $\mathbb{C}_y^t$ is a submodule of $B^t_{y}$ so that
$$
\widehat{\mathbb{C}_y}^t \cong \mathbb{D}^t \oplus B^t_{y}/ \mathbb{C}_y^t
$$
and $\mathbb{D}$  is a subcomplex of $\widehat{\mathbb{C}_y}$. So $\mathbb{D} \hookrightarrow \widehat{\mathbb{C}_y}$ induces a short exact sequence of complexes 
\begin{eqnarray}
\label{ses-Prop6.15}
0
\rightarrow 
\mathbb{D}
\rightarrow 
\widehat{\mathbb{C}_y}
\rightarrow 
\widehat{\mathbb{C}_y}/\mathbb{D}
\rightarrow 
0
\end{eqnarray}
where $(\widehat{\mathbb{C}_y}/\mathbb{D})^t \cong B^t_{y}/ \mathbb{C}_y^t$.

We have the following isomorphisms
\begin{eqnarray*}
(\widehat{\mathbb{C}_y}/\mathbb{D}
)^t
& \cong &
 B^t_{y}/ \mathbb{C}_y^t
\\
&\cong&
\bigoplus_{\substack{ \mathbf{b}^{(t)}\subset \mathbf{b} 
\\
|\mathbf{b}^{(t)}| = b-t
\\
x \in \mathbf{b}^{(t)}
}}
\kring \{ f \in \hom_\Oz (\mathbf{b}^{(t)}, \mathbf{a})|  f(x)= y \}/\kring \{ f \in \hom_\Oz (\mathbf{b}^{(t)}, \mathbf{a})|  f^{-1}(y)= \{ x\} \}
\otimes \orient (\mathbf{b} \backslash \mathbf{b}^{(t)} )
\\
& \cong & \bigoplus_{\substack{ \mathbf{b}^{(t)}\subset \mathbf{b} 
\\
|\mathbf{b}^{(t)}| = b-t
\\
x \in \mathbf{b}^{(t)}
}}
\kring \{ f \in \hom_\Oz (\mathbf{b}^{(t)}, \mathbf{a})|  f(x)= y \mathrm{\  and\  } |f^{-1}(y)| \geq 2\}
\otimes \orient (\mathbf{b} \backslash \mathbf{b}^{(t)} )
\\
&
\cong & 
\bigoplus_{\substack{ \mathbf{b}^{(t+1)}\subset \mathbf{b} 
\\
|\mathbf{b}^{(t+1)}| = b-(t+1)
\\
x \not\in \mathbf{b}^{(t+1)}
}}
\kring \{ f \in \hom_\Oz (\mathbf{b}^{(t+1)}, \mathbf{a}) \}
\otimes \orient (\mathbf{b} \backslash \mathbf{b}^{(t+1)} )
\\
&\cong &
\mathbb{D}^{t+1},
\end{eqnarray*}
where the fourth isomorphism is obtained taking the restriction given by removing the element $x=b$ from $\mathbf{b}^{(t)} $. These isomorphisms of $\kring$-modules are compatible with the differential of the complexes, so that we have an isomorphism of complexes:  $\widehat{\mathbb{C}_y}/\mathbb{D} \cong \mathbb{D}[-1]$.

The connecting morphism of the long exact sequence associated to (\ref{ses-Prop6.15})  is induced by the restriction given by removing the element $x=b$ from $\mathbf{b}'$. This operation corresponds to the isomorphism between $\widehat{\mathbb{C}_y}/\mathbb{D}$ and $\mathbb{D}[-1]$; so the connecting morphism is  an isomorphism, which implies the acyclicity of  $\widehat{\mathbb{C}_y}$, as required. 
\end{proof}

\begin{thm}
\label{thm:cohomology_NC}
If  $b >a >0 $, then $H^* (\mathbb{C}(\mathbf{b}, \mathbf{a})) $ 
is the free graded $\kring$-module 
\[
H^* (\mathbb{C}(\mathbf{b}, \mathbf{a})) \cong 
\left\{
\begin{array}{ll}
\kring & *= b-a\\
0 & \mbox{ $0<* < b-a$ or $*> b-a$} \\
\kring ^{\oplus \epfn (b,a)} & *= 0 
\end{array}
\right.
\]
where $\epfn (b,a) =   \chi (\mathbb{C(\mathbf{b}, \mathbf{a})}) - (-1)^{b-a}$, for  $\chi$  the Euler-Poincar\'e characteristic 
of the complex $\mathbb{C(\mathbf{b}, \mathbf{a})}$:
\[
\chi (\mathbb{C(\mathbf{b}, \mathbf{a})})
= 
\sum_{t=0}^{b-a}
(-1)^t 
\binom{b}{t}
|\hom_\Oz (\mathbf{b-t}, \mathbf{a}) |.
\]
\end{thm}

\begin{proof}
We consider the proposition $\mathcal{P}(b,a)$: the cohomology of $\mathbb{C}= \mathbb{C}(\mathbf{b}, \mathbf{a})$ is free as a $\kring$-module concentrated in degrees $0$ and $b-a$, with $H^{b-a}(\mathbb{C})$ of rank one.

We prove that  $\mathcal{P}(b,a)$ is true for all  $b >a \geq 1$ by double induction upon $1 \leq a$ and upon $1 \leq b-a$.

The initial case  $\mathcal{P}(b,1)$ is treated by Proposition  \ref{prop:case_a=1_b>1} and 
the initial case $\mathcal{P}(a+1,a)$ is covered by Proposition \ref{cor:cohomology_b=a+1}.
The inductive step is to show that, given a pair $(b,a)$ with $a>1$ and $b-a >1$ and  such that  $\mathcal{P}(b-1,a)$ and  $\mathcal{P}(b-1,a-1)$ are true then  $\mathcal{P}(b,a)$ holds.

For this, we will prove, by induction on $y$,  the following stronger result: if $\mathcal{P}(b-1,a)$ and  $\mathcal{P}(b-1,a-1)$ are true  then for all $y \in \{1, \ldots, a\}$ the cohomology of $\filt_{y} \cbb$ is free as a $\kring$-module concentrated in degrees $0$ and $b-a$, with $H^{b-a}(\filt_{y} \cbb)$ of rank one. The required result is then obtained for $y=a$.

For $y=1$, by 
Lemma \ref{lem:identify_subcomplexes} and Proposition \ref{prop:equivalence_cohom} we have 
$$H^* (\filt_1 \cbb) = H^*(\widetilde {\mathbb{C}_1}) \cong H^*( {\mathbb{C}_1}) \cong H^*(\mathbb{C} (\mathbf{b-1}, \mathbf{a-1}) )$$
so the result follows by the assumption on $\mathcal{P}(b-1, a-1)$. 

To prove the inductive step we consider the long exact sequence in cohomology associated to the short exact sequence (\ref{eqn:ses_D_filt}) given in Lemma \ref{lem:decompose_C}.
\begin{eqnarray}
\label{les-preuve}
\ldots \rightarrow H^t(\mathbb{D}) \rightarrow H^t(\filt_y \cbb) \oplus H^t( \widetilde{\cbb_{y+1}})
\rightarrow 
H^t(\filt_{y+1} \cbb)  \rightarrow H^{t+1}(\mathbb{D}) \rightarrow \ldots
\end{eqnarray}

Lemma \ref{lem:identify_subcomplexes} together with the inductive hypothesis for $\mathcal{P}(b-1, a)$ implies that $H^* (\mathbb{D})$ is concentrated in cohomological degrees $1$ and $b-a$, with  $H^{b-a} (\mathbb{D}) \cong \kring$ and $H^{1} (\mathbb{D})$ a free $\kring$-module.
 Moreover, as for the case $y=1$, the cohomology of $\widetilde {\mathbb{C}_{y+1}}$ is concentrated in cohomological degrees zero and $b-a$, with $H^{b-a} (\widetilde{\mathbb{C}_{y+1}}) \cong \kring$ and $H^{0} (\widetilde{\mathbb{C}_{y+1}})$ a free $\kring$-module. So using the inductive hypothesis on $H^{*}(\filt_{y} \cbb)$, we obtain that the exact sequence (\ref{les-preuve}) reduces to the following:
\begin{itemize}
\item $
H^t(\filt_{y+1} \cbb)=0$ for $1\leq t\leq b-a-2$ and $t>b-a$;
\item 
the  exact sequence:
\begin{eqnarray}
\label{ses1-preuve}
0 \rightarrow H^0(\filt_y \cbb) \oplus H^0( \widetilde{\cbb_{y+1}})
\rightarrow 
H^0(\filt_{y+1} \cbb)  \rightarrow H^{1}(\mathbb{D}) \rightarrow 0
\end{eqnarray}
in which $H^{1}(\mathbb{D})$ and $H^{0} (\widetilde{\mathbb{C}_{y+1}})$ are free $\kring$-modules; 
\item 
the exact sequence, in which $d$ is induced by the inclusions $\mathbb{D} \subset \filt_y \cbb$ and $\mathbb{D} \subset\widetilde{\cbb_{y+1}}$
\begin{eqnarray}
\label{ses2-preuve}
\\
\nonumber
0 \rightarrow H^{b-a-1}(\filt_{y+1} \cbb) \rightarrow H^{b-a}(\mathbb{D}) \xrightarrow{d} H^{b-a}(\filt_y \cbb) \oplus H^{b-a}( \widetilde{\cbb_{y+1}}) \rightarrow H^{b-a}(\filt_{y+1} \cbb)\rightarrow 0.
\end{eqnarray}
\end{itemize}

Using the exact sequence (\ref{ses1-preuve}): by the inductive hypothesis $H^0(\filt_y \cbb)$ is a free $\kring$-module; since $H^{1}(\mathbb{D})$ is a free $\kring$-module,  the short exact sequence splits as $\kring$-modules giving that $H^0 (\filt_{y+1}\cbb) $ is $\kring$-free, as required. 

Consider the exact sequence (\ref{ses2-preuve}): recall that $H^{b-a}(\mathbb{D}) \cong H^{b-a}( \widetilde{\cbb_{y+1}}) \cong \kring$. We prove that  the inclusion $i: \mathbb{D} \hookrightarrow \widetilde {\mathbb{C}_{y+1}}$ induces an isomorphism $H^{b-a}(\mathbb{D}) \xrightarrow{\cong} H^{b-a}( \widetilde{\cbb_{y+1}})$. 
 
The cocycle $[\id_\mathbf{a}] \otimes \omega (\mathbf{b}\backslash \mathbf{a}) \in \mathbb{D}^{b-a}$ generates  $H^{b-a} (\mathbb{D})$. We will see that the cohomology class of $i([\id_\mathbf{a}] \otimes \omega (\mathbf{b}\backslash \mathbf{a})) \in \widetilde{\cbb_{y+1}}^{b-a}$ generates $H^{b-a}( \widetilde{\cbb_{y+1}})$.

 Consider $\mathbf{b}'=\mathbf{a}\setminus \{y+1\}$ and  $f \in \hom _\Oz (\mathbf{b}' \cup \{x\}, \mathbf{a} )$ given by $f|_{\mathbf{b}'}=\id_\mathbf{a}|_{ \mathbf{b}'}$ and $f(x)=y+1$, the cocycle $[f] \otimes \omega (\mathbf{b} \backslash (\{x \} \cup \mathbf{b}')) \in \mathbb{C}_{y+1}^{b-a}$ generates $H^{b-a}( {\mathbb{C}_{y+1}})$. By Proposition \ref{prop:equivalence_cohom}, $j([f] \otimes \omega (\mathbf{b} \backslash (\{x \} \cup \mathbf{b}'))) \in \widetilde{\mathbb{C}_{y+1}}^{b-a}$ generates $H^{b-a}(\widetilde{\mathbb{C}_{y+1}})$.

 Adapting the second part of Proposition \ref{prop:top_degree} to the complex $\widetilde {\mathbb{C}_{y+1}}$ we obtain that the cohomology classes of $i([\id_\mathbf{a}] \otimes \omega (\mathbf{b}\backslash \mathbf{a}))$ and $j([f] \otimes \omega (\mathbf{b} \backslash (\{x \} \cup \mathbf{b}')))$ are equal (up to sign) (where $j: \mathbb{C}_{y+1} \hookrightarrow \widetilde {\mathbb{C}_{y+1}}$). We deduce that $i([\id_\mathbf{a}] \otimes \omega (\mathbf{b}\backslash \mathbf{a}))$ generates $H^{b-a}( \widetilde{\cbb_{y+1}})$.

It follows from the exact sequence that $H^{b-a-1}(\filt_{y+1} \cbb)=0$ and $H^{b-a}(\filt_y \cbb) \cong H^{b-a}(\filt_{y+1} \cbb)$ which is of rank one by the recursive hypothesis.

 It remains to identify the rank of the free $\kring$-module $H^0 (\mathbb{C})$. 
By the first part of the proof, $\mathbb{C}(\mathbf{b}, \mathbf{a})$ is a complex of free $\kring$-modules of finite rank, concentrated in degrees between $0$ and $b-a$. So we have
$$
\chi (\mathbb{C(\mathbf{b}, \mathbf{a})})
=
\sum_{t=0}^{b-a}(-1)^t 
\krank \mathbb{C(\mathbf{b}, \mathbf{a})}^t=
\sum_{t=0}^{b-a}(-1)^t \krank
H^t(\mathbb{C(\mathbf{b}, \mathbf{a})}).$$
By the explicit form of the complex $\mathbb{C}(\mathbf{b}, \mathbf{a})$ given in Notation \ref{nota:Kz_Cba} we have 
\[
\chi (\mathbb{C(\mathbf{b}, \mathbf{a})})
= 
\sum_{t=0}^{b-a}
(-1)^t 
\binom{b}{t}
|\hom_\Oz (\mathbf{b-t}, \mathbf{a}) |.
\]
Since $H^* (\mathbb{C})$ is concentrated in degrees $0$ and 
$b-a$, we obtain
$$\krank H^0(\mathbb{C(\mathbf{b}, \mathbf{a})})
=\chi (\mathbb{C(\mathbf{b}, \mathbf{a})})-(-1)^{b-a} \krank H^{b-a}(\mathbb{C(\mathbf{b}, \mathbf{a})})$$
and $\krank H^{b-a}(\mathbb{C(\mathbf{b}, \mathbf{a})})=1$ by Proposition \ref{prop:top_degree}.
\end{proof}

\begin{rem}
\label{rem-extra-case}
By Lemma \ref{lem:case_a=b},  for $a, b \in \nat$, $H^* (\mathbb{C}(\mathbf{b}, \mathbf{a}))=0$ if $b<a$, $H^* (\mathbb{C}(\mathbf{b}, \mathbf{0}))$ is $\kring$, concentrated in degree $b$ and $H^* (\mathbb{C}(\mathbf{a}, \mathbf{a}))$ is $ \kring[\sym_a]$ concentrated in degree $0$.
\end{rem}

\begin{rem}
The long exact sequence in cohomology given by \cite[Theorem 1]{MR3603074} suggests an alternative approach to proving Theorem \ref{thm:cohomology_NC}, based upon a similar strategy. Namely, choosing the element $x$ in the above proof corresponds to using the shift functor of $\finj\op$-modules in \cite{MR3603074}.
 
To implement this strategy using \cite[Theorem 1]{MR3603074} requires showing that the relevant connecting morphism is surjective in cohomology.
\end{rem}

By Lemma  \ref{lem:case_a=b} and Theorem \ref{thm:cohomology_NC}, for $(b,a) \in  \nat \times  \nat$,  $H^0(\mathbb{C(\mathbf{b}, \mathbf{a})})$ is a free finitely-generated $\kring$-module. So we have a function $\epfn: \nat \times  \nat \to \nat$ given by
$$\epfn(b,a)=\krank H^0(\mathbb{C(\mathbf{b}, \mathbf{a})}).$$
This function satisfies the following Proposition.

\begin{prop}
\label{cor:inductive_rho}
The function $\epfn: \nat \times  \nat \to \nat$  is determined by:
\begin{eqnarray}
\label{cor:inductive_rho-eq-1}
\epfn (b,0) &=& \left\{
\begin{array}{ll}
1 &  b=0 \\
0 & \mbox{otherwise} .
\end{array}
\right.
\\
\label{cor:inductive_rho-eq0}
\epfn (b,a) &=& 0  \qquad \mbox{for $b<a$} \\
\label{cor:inductive_rho-eq1}
\epfn (a,a) &=& a! \\
\label{cor:inductive_rho-eq2}
\epfn (b,1) &=& 0 \qquad \mbox{ for $b>1$} \\
\label{cor:inductive_rho-eq3}
\epfn (a+1, a) &=& (a-2) \frac{(a+1)!}{2} +1  \qquad \mbox{ for $a >1$}\\
\label{cor:inductive_rho-eq4}
\epfn (b,a) &=& a \epfn (b-1, a-1) + (a-1)\epfn (b-1,a) \qquad \mbox{ for $b-1>a>1$}.
\end{eqnarray}
\end{prop}

\begin{proof}
The equalities (\ref{cor:inductive_rho-eq-1}),  (\ref{cor:inductive_rho-eq0}) and (\ref{cor:inductive_rho-eq1}) are consequences of Lemma \ref{lem:case_a=b};  (\ref{cor:inductive_rho-eq2}) follows from Proposition \ref{prop:case_a=1_b>1} and (\ref{cor:inductive_rho-eq3}) from Proposition \ref{cor:cohomology_b=a+1}.

To prove the recursive expression (\ref{cor:inductive_rho-eq4}), observe that the short exact sequence  (\ref{ses1-preuve}) given in the proof of Theorem \ref{thm:cohomology_NC} implies the following equality for all $y \in \{1, \ldots, a-1 \}$:
$$
\krank H^0(\filt_{y+1} \cbb)= \krank  H^0(\filt_y \cbb)+ \krank  H^0( \widetilde{\cbb_{y+1}})+ \krank H^{1}(\mathbb{D}).
$$
Taking the sum over $y$, for $1 \leq y \leq a-1$, of these equalities, we obtain:
$$\epfn (b,a)=\krank H^0(\filt_{a} \cbb)= \krank  H^0(\filt_1 \cbb)+ \sum_{y=1}^{a-1}\krank  H^0( \widetilde{\cbb_{y+1}})+ \sum_{y=1}^{a-1}\krank H^{1}(\mathbb{D})$$
$$=  \sum_{y=1}^{a}\krank  H^0( \widetilde{\cbb_{y}})+ \sum_{y=1}^{a-1}\krank H^{1}(\mathbb{D}).$$
By Lemma \ref{lem:identify_subcomplexes} we have: 
$\krank H^{1}(\mathbb{D})=\epfn (b-1,a)$, and by Lemma  \ref{lem:identify_subcomplexes}  and Proposition \ref{prop:equivalence_cohom} we have: $\krank  H^0( \widetilde{\cbb_{y}})= \krank  H^0({\cbb_{y}})=\epfn(b-1,a-1)$.
\end{proof}

\begin{cor}
\label{cor2:inductive_rho}
The function $\epfn: \nat \times  \nat \to \nat$  satisfies the following equalities:
\begin{eqnarray}
\label{cor:inductive_rho-eq5}
\epfn (b,2) &=& 1 \qquad \mbox{ for $b>2$}\\
\label{cor:inductive_rho-eq6}
\epfn (b,3) &=& 2^{b}-3 \qquad \mbox{ for $b>3$}.
\end{eqnarray}
\end{cor}

\begin{proof}
To prove (\ref{cor:inductive_rho-eq5}), note that $\epfn (3,2)=1$ by (\ref{cor:inductive_rho-eq3}) and, for $b>3$, we have $b-1>2>1$;  thus, by (\ref{cor:inductive_rho-eq4}) and (\ref{cor:inductive_rho-eq2}), we have
$$\epfn(b,2)=2\epfn(b-1,1)+ \epfn(b-1,2)=\epfn(b-1,2)$$
so $\epfn(b,2)=\epfn (3,2)=1$.

To prove (\ref{cor:inductive_rho-eq6}), note that $\epfn (4,3)=13$ by (\ref{cor:inductive_rho-eq3}) and for $b>4$, we have $b-1>3>1$; thus, by (\ref{cor:inductive_rho-eq4}) and (\ref{cor:inductive_rho-eq5}), we have
$$\epfn(b,3)=3\epfn(b-1,2)+ 2  \epfn(b-1,3)=3+ 2\epfn(b-1,3).$$
Taking the linear combination of these equalities for $b\geq y \geq 5$ we obtain:
$$\sum_{y=5}^{b} 2^{b-y} \epfn(y,3)=3\sum_{y=5}^{b} 2^{b-y}+ 2 \sum_{y=5}^{b} 2^{b-y} \epfn(y-1,3)= 3\sum_{y=0}^{b-5} 2^{y}+ \sum_{y=4}^{b-1} 2^{b-y} \epfn(y,3)$$
So $\epfn(b,3)=3(2^{b-4}-1)+2^{b-4} \epfn(4,3)=  2^{b}-3$.
\end{proof}

\subsection{The cohomology as a representation}
\label{subsect:cohom_rep}

Recall that $H^* (\mathbb{C}(\mathbf{b}, \mathbf{a}) )$ takes values in graded $\sym_a \times \sym_b\op$-modules. Proposition \ref{prop:top_degree} identifies
this representation in the top degree, $* = b-a$. 

Working in the Grothendieck group of finitely-generated $\sym_a \times \sym_b\op$-modules, Theorem \ref{thm:cohomology_NC} (and its proof) give further information, beyond the dimension of $H^0 (\mathbb{C} (\mathbf{b}, \mathbf{a}))$. In the following, the class of a representation $M$ in the Grothendieck group is denoted by $[M]$.

\begin{thm}
\label{thm:cohom_equivariance}
Suppose that $b > a \in \nat^*$, then there is an equality in the Grothendieck group of finitely-generated $\sym_a \times \sym_b\op$-modules:
\begin{eqnarray*}
[H^0(\mathbb{C} (\mathbf{b}, \mathbf{a}))]
&=& 
(-1)^{b-a+1} [\orient (\mathbf{a}) \otimes \orient (\mathbf{b}) ] 
+ 
\sum_{t=0}^{b-a}(-1)^t[\mathbb{C}(\mathbf{b}, \mathbf{a})^t]
\\
&=&
(-1)^{b-a+1} [\orient (\mathbf{a}) \otimes \orient (\mathbf{b})] 
\\
&&
+ 
\sum_{t=0}^{b-a}(-1)^t
\big[\big(
\kring \hom_\Oz (\mathbf{b - t}, \mathbf{a}) 
\otimes 
\orient (\mathbf{b}\backslash (\mathbf{b-t}))
\big)\uparrow _{\sym_{b-t}\op \times \sym_t \op}^{\sym_b\op}
\big].
\end{eqnarray*}

In particular, if $\kring = \rat$, then this identifies the representation $H^0(\mathbb{C} (\mathbf{b}, \mathbf{a}))$ up to isomorphism.
\end{thm}

\begin{proof}
Passage to the Grothendieck group allows the Euler-Poincar\'e characteristic argument used in the proof of Theorem \ref{thm:cohomology_NC} to be applied to give the first stated equality. The second expression then follows from Lemma 
\ref{lem:Cbb_action}.

The category of  $\rat[\sym_a \times \sym_b \op]$-modules is semi-simple, so this suffices to identify the representation up to isomorphism.
\end{proof}

\begin{rem}
For $a < b \in \nat$,  $0 \leq t \leq b-a$ and $\kring =\rat$, the permutation $\sym_a \times \sym_{b-t}\op$-modules
$
\rat \hom_\Oz (\mathbf{b-t}, \mathbf{a})
$
arising in the statement of Theorem \ref{thm:cohom_equivariance} can be identified by standard methods of representation theory, so $H^0 (\mathbb{C} (\mathbf{b}, \mathbf{a}))$ can be calculated effectively in this case.
\end{rem}

\begin{exam}
\label{exam:H0_equivariance}
Consider the case $a \in \nat^*$ and $b =a+1$, so that:
\begin{eqnarray*}
\mathbb{C}(\mathbf{a+1}, \mathbf{a})^0&=& \kring \hom_\Oz (\mathbf{a+1},\mathbf{a}) 
\\
 \mathbb{C}(\mathbf{a+1}, \mathbf{a})^1&\cong& \kring \sym_{a+1},
 \end{eqnarray*}
 where the action in degree zero is the natural one  and, in degree one, $\sym_a$ acts on the left by restriction along $\sym_a \subset \sym_{a+1}$.
 
Theorem \ref{thm:cohom_equivariance} gives the equality in the Grothendieck group of $\sym_a \times \sym_{a+1}\op$-modules:
\[
[H^0(\mathbb{C}(\mathbf{a+1}, \mathbf{a}))]
=[\orient(\mathbf{a}) \otimes \orient(\mathbf{a+1})] 
+ 
[\kring \hom_\Oz (\mathbf{a+1},\mathbf{a}) ]
-
[ \kring \sym_{a+1}].
\]
(The case $a=0$ is deduced from Lemma \ref{lem:case_a=b}.)
Consider the case $a=2$ and restrict to the $\sym_2$-action (for the inclusion of the first factor in  $ \sym_2 \times \sym_3\op$). Both $\kring \sym_3$ and $\kring \hom_\Oz (\mathbf{3}, \mathbf{2})$ are free $\kring \sym_2$-modules of rank three; this gives the equality in the Grothendieck group of $\sym_2$-modules:
\[
[H^0(\mathbb{C}(\mathbf{3}, \mathbf{2}))]
= [\sgnrep_{\sym_2}],
\]
by identifying $\orient (\mathbf{2})$ as the signature representation. 

For $\kring=\rat$, this shows that $H^0(\mathbb{C}(\mathbf{3}, \mathbf{2}))$ is not a permutation representation as a $\sym_2$-module and hence cannot be a permutation representation as a $\sym_2 \times \sym_3\op$-module. 
\end{exam}

\section{Computation of $\hom_\finne$ and $\ext^1_\finne$}
\label{sect:homext}

In this section we combine our previous results to prove our main result, Theorem \ref{thm:complex_Cab}, calculating the groups that interest us. From this, we deduce properties of the objects $\kbar^{\otimes a} \in \ob \fcatk[\finne]$. Recall that the function $\epfn: \nat \times  \nat \to \nat$  is defined before Proposition \ref{cor:inductive_rho}.

\begin{thm}
\label{thm:complex_Cab}
For $a,b \in \nat$, there are isomorphisms
\begin{eqnarray*}
\hom_{\fcatk[\finne]}(\kbar^{\otimes a}, \kbar^{\otimes b} ) 
& \cong & \kring^{\oplus \epfn (b,a)} 
\\
\ext^1 _{\fcatk[\finne]}(\kbar^{\otimes a}, \kbar^{\otimes b} )
&\cong & 
\left\{ 
\begin{array}{ll}
\kring & b = a+1 \\
0 & \mbox{otherwise.}
\end{array}
\right.
\end{eqnarray*}
\end{thm}

\begin{proof}[Proof of Theorem \ref{thm:complex_Cab}]
By Theorem \ref{thm:motivate}, for $a, b \in \nat$ there are isomorphisms:
\begin{eqnarray*}
\hom_{\fcatk[\finne]} (\kbar^{\otimes a},\kbar^{\otimes b}) 
&\cong & H^0 (\hom_{\fcatk[\Gamma]}  ((t^*)^{\otimes a}, \gcmnd^\bullet (t^*)^{\otimes b} )) 
\\
\ext^1_{\fcatk[\finne]} (\kbar^{\otimes a},\kbar^{\otimes b})
&
\cong 
& 
H^1 (\hom_{\fcatk[\Gamma]}  ((t^*)^{\otimes a}, \gcmnd^\bullet (t^*)^{\otimes b} )). 
\end{eqnarray*}

By Theorem \ref{thm:identify_cosimp_gcmnd_fbcmnd} (or its Corollary \ref{cor:relate_finjop_cohom}),  there is an isomorphism in graded $\fb\op$-modules:
\[
H^* ( \hom_{\fcatk[\Gamma]}  ((t^*)^{\otimes a}, \gcmnd^\bullet \twist))
\cong 
H^* (\kz \kring \hom_\Oz (-, \mathbf{a}))
\]
By Lemma \ref{lem:case_a=b}, for $b<a$, $\kz \kring \hom_\Oz (-, \mathbf{a}) (\mathbf{b})=0$, $\kz \kring \hom_\Oz (-, \mathbf{a}) (\mathbf{a})=\kring[\sym_a]$, concentrated in cohomological degree zero and $\kz \kring \hom_\Oz (-, \mathbf{0}) (\mathbf{b})=\kring$  concentrated in cohomological degree $b$. For $b>a >0$, the calculation of these cohomology modules is given by Theorem \ref{thm:cohomology_NC}.
\end{proof}

Theorem \ref{thm:complex_Cab} gives the following:

\begin{cor}
\label{cor:non-proj_inj}
For $a \in \nat$,
\begin{enumerate}
\item 
$\overline{\kring} = \kbar^{\otimes 0}$ is injective if and only if $\kring$ is an injective $\kring$-module;
\item 
$\kbar^{\otimes a}$ is not injective if $a>0$; 
\item 
$\kbar^{\otimes a}$ is not projective.
\end{enumerate}
\end{cor}

\begin{proof}
The argument employed in the proof of
Proposition \ref{prop:overline_k_injective} implies that $\overline{\kring} = \kbar^{\otimes 0}$ is injective if and only if $\kring$ is an injective $\kring$-module.

The non-vanishing of 
$$\ext^1_{\fcatk[\finne]} (\kbar^{\otimes a} , \kbar^{\otimes a+1} )\cong \kring $$ 
for any $a\in \nat$, given by Theorem \ref{thm:complex_Cab}, implies the remaining non-projectivity and non-injectivity statements.
\end{proof}

For the next Corollary we use the following Notation:

\begin{nota}
\label{nota:Rba}
For $a \in \nat^*$ and $b \in \nat^*$, let 
$$
R_{b,a} : 
\kring \hom_\Oz (\mathbf{b}, \mathbf{a}) 
\rightarrow 
 \bigoplus _{\substack{\mathbf{b'}\subsetneq \mathbf{b} \\
 |\mathbf{b'}| = |\mathbf{b}|-1 }}
 \kring \hom_\Oz (\mathbf{b'}, \mathbf{a}) 
$$ 
denote the map induced by restriction along subsets $\mathbf{b'} \subsetneq \mathbf{b}$ with $ |\mathbf{b'}| = |\mathbf{b}|-1$. 
\end{nota}

\begin{cor}
\label{cor:Rba}
For $a \in \nat^*$ and $b \in \nat^*$, 
we have 
$$
\hom_{\fcatk[\finne]} (\kbar^{\otimes a},\kbar^{\otimes b}) \cong H^0 (\kz \kring \hom_\Oz (-, \mathbf{a}) (\mathbf{b})) \cong \ker(R_{b,a}).
$$
In particular, this allows us to consider $\hom_{\fcatk[\finne]} (\kbar^{\otimes a},\kbar^{\otimes b}) $ as a submodule of $\kring \hom_\Oz (\mathbf{b}, \mathbf{a})$, equipped with its natural   $\sym_a \times \sym_b\op$-action.
\end{cor}
\begin{proof}
For the first isomorphism, combine Theorems \ref{thm:motivate} and \ref{thm:identify_cosimp_gcmnd_fbcmnd} as in the beginning of the proof of Theorem  \ref{thm:complex_Cab}. For the second, use Definition \ref{defn:Cbb}.
\end{proof}

As in Section \ref{subsect:cohom_rep}, the statement of Theorem \ref{thm:complex_Cab} can be refined by taking into account the natural $ \sym_a \times \sym_b\op $-action. 

\begin{prop}
\label{prop:hom_equivariance}
For $a<b \in \nat^*$, there is an equality in the Grothendieck group of $\sym_a \times \sym_b\op$-modules:
\begin{eqnarray*}
[\hom_{\fcatk[\finne]}(\kbar^{\otimes a}, \kbar^{\otimes b} )]
&=&
(-1)^{b-a+1} [\sgnrep_{\sym_a} \otimes \sgnrep_{\sym_b}] 
\ + 
\\
&&
\sum_{t=0}^{b-a}(-1)^t
\big[\big(
\kring \hom_\Oz (\mathbf{b - t}, \mathbf{a}) 
\otimes 
\sgnrep_{\sym_t} 
\big)\uparrow _{\sym_{b-t}\op \times \sym_t \op}^{\sym_b\op}
\big].
\end{eqnarray*}
\end{prop}

\appendix 

 \section{Normalized (co)chain complexes}
 \label{sect:app_norm}

This short appendix fixes our conventions on the normalized cochain complex associated to a cosimplicial object in an abelian category $\cala$. We start by reviewing the slightly more familiar simplicial case.  

First recall (see \cite[Chapter 8]{MR1269324}, for example) that there is an {\em unnormalized} chain complex $C(A_\bullet)$ associated to a simplicial object $A_\bullet$ in $\cala$ with  $C(A_\bullet)_n = A_n$ and differential given by the alternating sum of the face maps. 

The normalized chain complex $N (A _\bullet)$ is defined (see \cite[Definition 8.3.6]{MR1269324}) by 
\[
N_n (A_\bullet) := \bigcap _{i=0} ^{n-1} \ker (\partial_i : A_n \rightarrow A_{n-1}),
\]
where $\partial_i$ denote the face maps, and the differential is given by $d:= (-1)^n \partial_n$. The  inclusion $N (A_\bullet) \rightarrow C(A_\bullet)$ induces an isomorphism in homology.

The chain complex $C(A_\bullet)$ has an acyclic subcomplex $D(A_\bullet) \subset C(A_\bullet)$, where $D(A_\bullet)$ is the subcomplex generated by the degenerate simplices. There is an isomorphism:
 \[
N (A_\bullet) 
\stackrel{\cong}{\rightarrow}
C(A_\bullet) / D (A_\bullet), 
\]
and hence the right hand side gives an alternative definition of the normalized complex.

In the cosimplicial setting, for $C^\bullet$ a cosimplicial object in $\cala$, the associated unnormalized cochain complex $(C^*, d)$ has differential $d$ given by the alternating sum of the coface maps.

\begin{defn}
\label{defn:normalized1}
For $C^\bullet$ a cosimplicial object in $\cala$, the  normalized cochain complex $(\tilde{N} (C^\bullet), \tilde{d})$,  is given by 
\[
\tilde{ N }(C^\bullet)^n:= \mathrm{Coker} \Big( 
 \bigoplus _{i=0}^{n-1}
 C^{n-1}\stackrel{d^i}{\rightarrow} C^n
 \Big )
\]
with differential $\tilde{d} := (-1)^{n+1} d^{n+1} : \tilde{N} (C^\bullet)^n
\rightarrow
\tilde{N} (C^\bullet)^{n+1}$.
\end{defn}

It can be more convenient to work with the isomorphic complex $N (C^\bullet)$ that is given as a subcomplex of the cochain complex $(C^*, d)$; this is analogous to working with the quotient complex $C(A_\bullet) / D (A_\bullet)$ in the simplicial case.

\begin{defn}
\label{defn:normalized}
For $C^\bullet$ a cosimplicial object in $\cala$, let $(N(C^\bullet), d)$ denote the subcomplex of $(C^* , d)$ given by:
\[
N (C^\bullet)^n
:= 
\ker 
\Big ( 
C^n
\stackrel{s^j}{\rightarrow} 
\bigoplus_{j=0}^{n-1} 
C^{n-1}
\Big ).
\]
\end{defn}

As in the simplicial case, one has an isomorphism between the respective constructions of the normalized cochain complex:

\begin{prop}
\label{prop:iso_normalized_cosimp}
The inclusion $(N(C^\bullet), d) \hookrightarrow (C^* ,d)$ is a cohomology equivalence and induces an isomorphism of cochain complexes 
$
(N(C^\bullet), d)
\stackrel{\cong}{\rightarrow} 
(\tilde{N} (C^\bullet), \tilde{d}).
$
\end{prop}

\section{Cosimplicial objects from comonads}
\label{cosimplicial}

In this appendix, we associate to a comonad two natural coaugmented cosimplicial objects of a different nature. The first construction is a cobar-type cosimplicial resolution, extending the canonical augmented simplicial object recalled in Proposition \ref{prop:simplicial}, and the second one encodes the notion of morphisms of comodules.

 \begin{nota}
 For a category $\calc$, let $(\perp, \Delta, \epsilon)$ be a comonad, where $\Delta : \perp \rightarrow \perp\perp $ and $\epsilon : \perp \rightarrow \id$ are the structure natural transformations.  
\begin{enumerate}
\item  
 The structure morphism of a $\perp$-comodule $M$ is denoted $\psi_M : M \rightarrow \perp M$ and this $\perp$-comodule  by $(M, \psi_M)$.
\item 
The category of $\perp$-comodules in $\calc$ is denoted by $\calc_\perp$ and 
the morphisms  in $\calc_\perp$   by $\hom_\perp (-,- ) $.
\item 
Morphisms in $\calc$ are denoted simply by $\hom (-,-)$ to that there is a forgetful map
$ \hom_\perp (-,- )\rightarrow \hom (-,-)$.
\end{enumerate}
\end{nota}

We record the standard fact:

\begin{prop}
\label{prop:comodules_abelian}
Let $\calc$ be an abelian category equipped with a comonad $(\perp, \Delta, \epsilon)$ such that $\perp$ is exact. Then the category $\calc_\perp$ of $\perp$-comodules in $\calc$ has a unique abelian category structure for which the forgetful functor $\calc_\perp \rightarrow \calc$ is exact.
\end{prop}

Recall the  natural augmented simplicial object $\perp^{\bullet +1} X$ associated to any  $X \in \ob \calc$. 

\begin{prop}\cite[8.6.4]{MR1269324} 
\label{prop:simplicial}
For a  comonad $(\perp, \Delta, \epsilon)$ on the category $\calc$ and $X \in \ob \calc$, there is a natural augmented simplicial object $\perp^{\bullet +1} X$, which is $\perp^{n+1} X$ in simplicial degree $n$ and has face operators induced by $\epsilon$ and degeneracies by $\Delta$: 
\begin{eqnarray*}
\delta_i := \perp^i \epsilon_{\perp^{n-i}X}& \ :\ & \perp^{n+1}X \rightarrow \perp^n X \\
\sigma_i :=  \perp^i \Delta_{\perp^{n-i} X}& \ :\ &  \perp^{n+1}X \rightarrow \perp^{n+2} X. 
\end{eqnarray*}

The augmentation is given by $\epsilon _X : \perp X \rightarrow X$. 
\end{prop}

The augmented simplicial object $\perp^{\bullet +1} X$ of Proposition \ref{prop:simplicial} has the form :
\[
\xymatrix{
\ar@{.>}[rr]
&&
\perp \perp \perp X
\ar@<-2ex>@/_1pc/@{.>}[ll]
\ar@<-1ex>[rr]
\ar[rr]
\ar@<1ex>[rr]
&&
\perp \perp X
\ar@<-2ex>@/_1pc/[ll]
\ar@<-1ex>@/_1pc/[ll]
\ar@<.75ex>[rr]|{ \epsilon_{\perp X}}
\ar@<-.75ex>[rr]|{ \perp \epsilon_X}
&&
\perp X
\ar@<-1ex>@/_1pc/[ll]|{\Delta_X}
\ar[rr]|{ \epsilon_X}
&&
X
}
\]

\subsection{A cobar-type cosimplicial construction} 
\label{subsect:cobar}

In this section we show that, when $X \in \ob \calc$ is a $\perp$-comodule and the comonad $\perp$ is equipped with a coaugmentation $\eta: \id \rightarrow \perp$ satisfying Hypothesis \ref{hyp:coaugmentation}, then the  augmented simplicial object $\perp^{\bullet +1} X$  given in Proposition \ref{prop:simplicial} can be extended  naturally to a  coaugmented cosimplicial object. This is a cobar-type cosimplicial construction for $\perp$-comodules.

For this,  we recall that there is an augmented simplicial object underlying any cosimplicial object. (See Notation \ref{nota:ordinals} for the notation used.)

\begin{prop}
\label{prop:embed_augsimp_cosimp}
There is an embedding of categories $\aord\op \hookrightarrow \ord$ given by 
$[n] \mapsto [n+1]$  on objects,  for $-1 \leq n \in \zed$ and 
$\delta_i \mapsto  s_i$, 
$\sigma_j  \mapsto  d_{j+1}$. 
 In particular, for any category $\calc$, restriction induces a functor 
$\ord \calc \rightarrow \aord \op \calc$ from the category $\ord \calc$ of  cosimplicial objects in $\calc$  to the category  $\aord \op \calc$ of  augmented simplicial objects in $\calc$.
\end{prop}

\begin{proof}
Analogous to the proof  that Connes' cyclic category is isomorphic to its opposite given in \cite[Proposition 6.1.11]{MR1600246}.
\end{proof}

To define the coaugmented cosimplicial object of Proposition \ref{prop:cosimp_comod} we require that the comonad $\perp$ satisfies the following:

\begin{hyp}
\label{hyp:coaugmentation}
Suppose that there exists a natural transformation $\eta : \id \rightarrow \perp$ such that 
\begin{enumerate}
\item 
$\epsilon \circ \eta = 1_\calc$;
\item 
the following diagram commutes:
\[
\xymatrix{
\id 
\ar[r]^{\eta}
\ar[d]_{\eta}
&
\perp
\ar[d]^{\Delta}
\\
\perp 
\ar[r]_{\eta_{\perp}}
&
\perp \perp .
}
\]
\end{enumerate}
\end{hyp}

The augmented simplicial object of Proposition \ref{prop:simplicial} extends to give the following cosimplicial object in $\calc$.

\begin{prop}
\label{prop:cosimp_comod}
Let $(\perp, \Delta, \epsilon)$ be a comonad on the category $\calc$, equipped with a coaugmentation $\eta$ satisfying Hypothesis \ref{hyp:coaugmentation}. 
For $(X,\psi_X)  \in \ob \calc_\perp$, there is a natural cosimplicial object $\cosimp^\bullet X$ in $\calc$ with  $\cosimp^n X = \perp ^n X$ in cosimplicial degree $n$ and structure morphisms
\[
\xymatrix{
X
\ar@<.75ex>[rr]|{d^0= \eta_X}
\ar@<-.75ex>[rr]| {d^1= \psi_X}
&&
\perp X
\ar@/_1pc/@<-.5ex>[ll]|{s^0 = \epsilon_X}
\ar@<1.5ex>[rr]|{d^0= \eta_{\perp X}}
\ar[rr]|{d^1 = \Delta_X}
\ar@<-1.5ex>[rr]| {d^2= \perp \psi_X
}
&&
\perp \perp X
\ar@/_1pc/@<-2.5ex>[ll]|{s^0 = \epsilon_{\perp X}}
\ar@/_1pc/@<-1ex>[ll]|{s^1 = \perp \epsilon_X}
\ar@{.>}[rr]
&&
\ldots
\ar@{.>}@<-1ex>@/_1pc/[ll]
}
\]
given explicitly as follows: 
$d^i_{\perp^l X} \in \hom_\calc (\perp^l X, \perp^{l+1} X)$ for $0 \leq i \leq l+1$ and $s^j_{\perp^{l+1} X} \in \hom_\calc (\perp^{l+1} X , \perp ^l X)$ for $0 \leq j \leq l$, 
\begin{eqnarray*}
d^i_{\perp^l X} 
&=& 
\left\{
\begin{array}{ll}
\eta_{\perp^l X}  & i =0 \\
\perp^{i-1} \Delta_{\perp^{l-i} X}& 0< i < l+1\\
\perp^l \psi_X & i = l+1; 
\end{array}
\right.
\\
s^j_{\perp^{l+1} X} 
&= &
\perp^j \epsilon_{\perp^{l-j} X}.
\end{eqnarray*}

The underlying augmented simplicial object given by Proposition \ref{prop:embed_augsimp_cosimp} is the augmented simplicial object $\perp^{\bullet +1}X$ of Proposition \ref{prop:simplicial}.

If $\calc$ is abelian and $\perp$ is exact, then $\cosimp^\bullet$ is an exact functor $\calc_\perp  \rightarrow \ord \calc$. 
\end{prop}

\begin{proof}
The construction in the statement is natural with respect to the $\perp$-comodule $X$. Moreover, by inspection, the underlying augmented simplicial object is $\perp^{\bullet +1} X$. Hence, to verify the cosimplicial identities it suffices to check those involving $\eta$ or $\psi$. Moreover, it is straightforward to deduce these from those for the displayed part of the structure. 

The required identities follow from: 
\begin{eqnarray*}
\perp \psi_X \circ \psi_X &= &\Delta_X \circ \psi_X, \mbox{ by coassociativity of $\psi_X$}
\\
\eta_{\perp X} \circ \psi_X &=& \perp \psi_X \circ \eta_X, \mbox{ by naturality of $\eta$}
\\
\Delta_X \circ \eta_X &=& \eta_{\perp X} \circ \eta_X, \mbox{ by Hypothesis \ref{hyp:coaugmentation} (2)}
\\
\epsilon_X \circ \psi_X &=& 1_X, \mbox{ by the counital property of $\psi_X$}
\\
\epsilon_X \circ \eta_X &=& 1_X, \mbox{ by Hypothesis \ref{hyp:coaugmentation} (1)}
\\
\epsilon_{\perp X} \circ \perp \psi_X &=& \psi_X \circ \epsilon_X, \mbox{ by naturality of $\epsilon$}
\\
\perp \epsilon_X \circ \eta_{\perp X}
&=&
\eta_X \circ \epsilon_X, \mbox{ by naturality of $\eta$.}
 \end{eqnarray*}
\end{proof}

The above construction defines a cohomology theory for $\perp$-comodules in $\calc$:

\begin{defn}
\label{defn:perp_cohom}
Let $(\perp, \Delta, \epsilon)$ be a comonad on the  abelian category $\calc$ such that $\perp $ is exact,  equipped with a coaugmentation $\eta$ satisfying Hypothesis \ref{hyp:coaugmentation}. For $(X,\psi_X) \in \ob \calc_\perp$ and $n \in \nat$, let $H^n_\perp (X)$ be the $n$-th cohomology  of the complex associated to $\cosimp^\bullet X$. 
\end{defn}

\begin{exam}
\label{exam:perp_primitives}
Using the notation of Definition \ref{defn:perp_cohom}, one has the identification 
\[
H^0_\perp (X) = \mathrm{equalizer} \big(
\xymatrix{
X 
\ar@<.75ex>[r]|-{\eta_X} 
\ar@<-.75ex>[r]|-{\psi_X}
&
\perp X
}
\big). 
\]
This can be considered as the primitives of the $\perp$-comodule $X$ (with respect to $\eta$). 
\end{exam}

In order to justify considering the above as a cohomology theory for $\perp$-comodules, it is important to understand its behaviour on cofree $\perp$-comodules; namely, for $Y \in \ob \calc$, one considers the cofree comodule $(\perp Y, \Delta_Y)$  in $\calc_\perp$. By Proposition \ref{prop:cosimp_comod}, we have the cosimplicial object $\cosimp^\bullet (\perp Y)$.

\begin{prop}
\label{prop:cobar_perp}
Let $(\perp, \Delta, \epsilon)$ be a comonad on the category $\calc$, equipped with a coaugmentation $\eta$ satisfying Hypothesis \ref{hyp:coaugmentation}. For $Y \in \ob \calc$, the cosimplicial object $\cosimp^\bullet (\perp Y)$ associated to $(\perp Y, \Delta_Y) \in \ob \calc_\perp$ is coaugmented via 
\[
Y \stackrel{\eta_Y} {\rightarrow} \perp Y = \cosimp ^0 (\perp Y).
\] 
Moreover, the counit $\eta_Y : \perp Y \rightarrow Y$ equips $\cosimp ^\bullet (\perp Y)$ with an extra codegeneracy. 

If $\calc$ is abelian, the extra codegeneracy provides a chain nulhomotopy, so that the complex associated to the coaugmented  cosimplicial object $Y \rightarrow \cosimp^\bullet( \perp Y)$ is acyclic. 
\end{prop}

\begin{proof}
The second condition of Hypothesis  \ref{hyp:coaugmentation} ensures that $Y \stackrel{\eta_Y}{\rightarrow} \perp Y$ equalizes $d^0$ and $d^1$, hence defines a coaugmentation of the cosimplicial object $\cosimp^\bullet (\perp Y)$.

The counit $ \perp Y \stackrel{\eta_Y}{\rightarrow} Y$ therefore provides an extra codegeneracy of the coaugmented cosimplicial object. Since $\epsilon_Y \circ \eta_Y$ is the identity on $Y$ by the first condition of Hypothesis  \ref{hyp:coaugmentation},  the extra codegeneracy defines a chain nulhomotopy.
\end{proof}

The formal properties of $H^*_\perp$ are summarized by:

\begin{cor}
\label{cor:delta_perp}
Let $(\perp, \Delta, \epsilon)$ be a comonad on the  abelian category $\calc$ such that $\perp $ is exact,  equipped with a coaugmentation $\eta$ satisfying Hypothesis \ref{hyp:coaugmentation}.

Then $H^* _\perp$ is a $\delta$-functor on $\calc$ such that, for $Y \in \ob \calc$:
\[
H^n_\perp  (\perp Y)
\cong 
\left\{
\begin{array}{ll}
Y & n=0 \\
0 & n>0.
\end{array}
\right.
\] 
\end{cor}

\begin{proof}
By Proposition \ref{prop:cosimp_comod}, $\cosimp^\bullet $ is an exact functor 
 from $\calc_\perp$ to $\ord \calc$. In particular, it sends a short exact sequence of $\perp$-comodules to a short exact sequence of $\ord \calc$. Taking cohomotopy therefore gives a long exact sequence in cohomology, thus giving the $\delta$-functor structure. 
 
Proposition \ref{prop:cobar_perp} then shows that this $\delta$-functor is effacable (see  \cite[Exercise 2.4.5]{MR1269324}, for example, for this notion) by cofree $\perp$-comodules. 
\end{proof}

\subsection{The cosimplicial object $\hom (M, \perp^\bullet N)$}
\label{sect:homextseq}

In Proposition \ref{prop:coaugmented_cosimp} we show that, given a comonad $\perp$ on a category $\calc$, the natural inclusion of morphisms of $\perp$-comodules  extends to a natural coaugmented cosimplicial object. 

In the abelian setting, when $\perp$ is an additive functor, this gives rise to an associated complex. Theorem \ref{thm:extbar} gives an interpretation of the first two cohomology groups.

 \begin{defn}
 \label{defn:stub_cosimp}
 For $M, N \in \ob \calc_\perp$, define the natural transformations:
 \[
 \xymatrix{
 \hom (M,N) 
 \ar@<.5ex>[r]^{d^0}
 \ar@<-.5ex>[r]_{d^1}
 &
 \hom (M, \perp N) 
 \ar@<-1ex>@/_1pc/[l]|{s^0}
 }
 \]
 by $d^0 \gamma := (\perp \gamma)\psi_M $; $d^1 \gamma := \psi_N \gamma$ and $s^0 \zeta:= \epsilon_N \zeta$ for $\gamma \in \hom (M,N) $ and $\zeta \in \hom (M, \perp N)$ corresponding to composites in the (non-commutative) diagram:
 \[
 \xymatrix{
 M \ar[r]^{\psi_M} 
 \ar[d]_\gamma
 &
 \perp M 
 \ar[d]^{\perp \gamma}
 \\
 N 
 \ar[r]
_{\psi_N} 
&
\perp N
\ar[r]_{\epsilon_N}
&
N. 
 }
 \] 
 \end{defn}
 
\begin{prop}
\label{prop:stub_cosimp}
Using the notation of Definition \ref{defn:stub_cosimp}, 
\begin{enumerate}
\item 
$
\xymatrix{
\hom_\perp (M,N)
\ar[r]
&
\hom (M,N)
 \ar@<.5ex>[r]^{d^0}
 \ar@<-.5ex>[r]_{d^1}
 &
 \hom  (M,\perp N)
}
$
is an equalizer, where $\hom_\perp (M,N) \rightarrow \hom (M,N)$ is the canonical inclusion; 
\item 
there are equalities 
$
s^0 d^0 = \mathrm{id} = s^0 d^1.
$
\end{enumerate} 
\end{prop} 
 
\begin{proof}
The first statement is by definition of a morphism of $\perp$-comodules, corresponding to the commutativity of the square in Definition \ref{defn:stub_cosimp}. The second is a straightforward verification, using the fact that $\epsilon_N  \psi_N$ is the identity on $N$ (likewise for $M$) and the naturality of $\epsilon$.
\end{proof} 

The following is clear:

\begin{lem}
\label{lem:hom_aug_simp}
For $M, N \in \ob \calc_\perp$, the natural augmented simplicial object $\perp ^{\bullet +1} N$ in $\calc$ given by Proposition \ref{prop:simplicial} induces a natural augmented simplicial structure:  
 $$\hom (M, \perp ^{\bullet +1} N).$$
\end{lem}

The following result shows that this extends (as in Proposition \ref{prop:embed_augsimp_cosimp}) to a cosimplicial object:

\begin{prop}
\label{prop:coaugmented_cosimp}
For $M, N \in \ob \calc_\perp$, there is a natural cosimplicial object 
$\hom (M, \perp^\bullet N)$ given by $\hom(M, \perp^t N)$ in cosimplicial degree $t$, such that:
\begin{enumerate}
\item 
the underlying augmented simplicial object is that of Lemma \ref{lem:hom_aug_simp}; 
\item 
the remaining coface maps $d^0, d^{t+1} :\hom(M, \perp^t N) \rightrightarrows 
 \hom(M, \perp^{t+1} N)$ are given by 
 \begin{eqnarray*}
 d^0 \gamma &=& (\perp \gamma) \psi_M \\
 d^{t+1} \gamma &=& (\perp^t \psi_N)\gamma 
\end{eqnarray*}
for $\gamma \in \hom(M, \perp^t N)$;
\item 
the canonical inclusion $\hom_\perp (M,N) \rightarrow \hom (M,N)$ defines a coaugmentation of the cosimplicial object $\hom (M, \perp^\bullet N)$.
\end{enumerate}
If $\calc$ is an abelian category and $\perp $ is an additive functor, then this gives a coaugmented cosimplicial abelian group.
\end{prop}

\begin{proof}
Proposition \ref{prop:stub_cosimp} establishes the cosimplicial structure up to cosimplicial degree one together with the coaugmentation.

Because of the nature of the construction,  it is  sufficient to verify the cosimplicial identities for 
\begin{eqnarray*}
d^i d^j &:& \hom (M,N) \rightarrow \hom (M, \perp\perp  N) \\
s^i d^j &:& \hom (M, \perp N) \rightarrow \hom (M, \perp N).
\end{eqnarray*}

These are consequences of the axioms for a comonad and comodules over this comonad. For example, $d^0 d^0 \gamma$ is, by definition, the composite
$$
\big(\perp ( (\perp \gamma) \psi_M)\big) \psi _M
= 
(\perp \perp \gamma ) (\perp \psi_M) \psi_M
$$
whereas $d^1 d^0 \gamma$ is the composite 
$
\Delta _N (\perp \gamma) \psi_M.
$

That these coincide follows from the commutativity of the following diagram:
\[
\xymatrix{
M \ar[r]^{\psi_M} 
\ar[d]_{\psi_M}
&
\perp M 
\ar[d]|{\Delta_M} 
\ar[r]^{\perp \gamma} 
&
\perp N 
\ar[d]^{\Delta_N}
\\
\perp M 
\ar[r]_{\perp \psi _M}
&
\perp \perp M 
\ar[r]_{\perp \perp \gamma}
&
\perp \perp N,
}
\]
where the left hand square commutes by coassociativity  for $M$ and the right hand square by naturality of $\Delta$. 
The remaining verifications are similar and are left to the reader. 

Finally, if $\calc$ is abelian and $\perp$ is additive, then each $\hom(M, \perp^t N)$ is an abelian group, the functor $\perp$ induces a group morphism $\hom (-,-) \rightarrow \hom(\perp -, \perp -)$ so that, together with biadditivity of composition, the above construction yields a coaugmented cosimplicial abelian group. 
\end{proof} 

\begin{exam}
\label{exam:cosimp_t=1}
For $t=1$, the morphisms $d^i$ correspond to the three possible composites in the (non-commutative) diagram 
\[
\xymatrix{
M 
\ar[r]^{\psi_M} 
\ar[d]_\zeta 
&
\perp M 
\ar[d]^{\perp \zeta}
\\
\perp N
\ar@<.3ex>[r] ^{\Delta_N} 
\ar@<-.3ex>[r]_{\perp \psi_N} 
&
\perp \perp N
}
\]
and $s^0, s^1$  are induced respectively by the composite with  $\epsilon _{\perp N} ,  \perp \epsilon_N
 : \perp \perp N \rightrightarrows \perp N$. 
\end{exam}

\begin{cor}
\label{cor:aug_cochain}
Suppose that $\calc$ is abelian and $\perp $ is additive, then the complex associated to the coaugmented cosimplicial object of Proposition \ref{prop:coaugmented_cosimp} for $M, N \in \ob \calc_\perp$ yields a natural coaugmented cochain complex in abelian groups:
\[
0
\rightarrow 
\hom_\perp (M, N) 
\rightarrow 
\hom (M, N) 
\rightarrow 
\hom (M, \perp N)
\rightarrow 
\hom (M, \perp\perp N) 
\rightarrow
\ldots ,
\]
with $\hom (M, \perp^t N)$ placed in cohomological degree $t$. 
\end{cor}

In the rest of this section we assume:

\begin{hyp}
The category $\calc$ is abelian and $\perp$ is exact.
\end{hyp}

\begin{rem}
Proposition \ref{prop:comodules_abelian} implies that the category $\calc_\perp$ of $\perp$-comodules in $\calc$ is an abelian category and the forgetful functor $\calc_\perp \rightarrow \calc$ is exact. 
\end{rem}

In order to give an interpretation of the first and second cohomology groups of the cochain complex obtained in  Corollary \ref{cor:aug_cochain} we need the following notation.

\begin{nota}
Denote by
\begin{enumerate}
\item 
$\ext^1 (-,-)$ the group defined via classes of extensions in $\calc$; 
\item  
$\ext^1_\perp (-, -)$ the group of classes of extensions in $\calc_\perp$. 
\end{enumerate} 
\end{nota}

The following is clear:

\begin{prop}
\label{prop:Ext_perp_restrict}
The forgetful functor induces a natural transformation
\[
\ext^1_\perp (-, -)
\rightarrow 
\ext^1 (-, -).
\] 
\end{prop}

\begin{nota}
Write $\extbar^1_\perp (-,-)$ for the kernel of the 
transformation $\ext^1_\perp (-, -)
\rightarrow 
\ext^1 (-, -)$.
\end{nota}

\begin{rem}
If the underlying object of the $\perp$-comodule $M$ is projective in $\calc$, then $\extbar^1_\perp (M, -) \cong \ext^1_\perp (M, -)$.
\end{rem}

\begin{thm}
\label{thm:extbar}
For $M, N \in \ob \calc_\perp$,  the first two cohomology groups of the  cochain complex $(\hom (M, \perp^\bullet N),d)$  associated to the cosimplicial object of Proposition \ref{prop:coaugmented_cosimp}  identify as:
\begin{eqnarray*}
H^0 (\hom (M, \perp^\bullet N)) & \cong & \hom_\perp (M,N) \\
H^1 (\hom (M, \perp^\bullet N) ) &\cong & \extbar^1_\perp (M,N).
\end{eqnarray*}
\end{thm}

\begin{proof}
By Proposition \ref{prop:coaugmented_cosimp}, 
the beginning of the unnormalized cochain complex associated to $\hom (M, \perp^\bullet N)$ is of the form 
\begin{eqnarray}
\label{eq:cx_H0H1}
\hom (M, N) 
\stackrel{\mathfrak{d}_0}{\rightarrow}
\hom (M, \perp N)
\stackrel{\mathfrak{d}_1}{\rightarrow}
\hom (M, \perp \perp N), 
\end{eqnarray}
with $\hom (M,N)$ placed in cohomological degree zero, 
where 
\begin{eqnarray*}
\mathfrak{d}_0 \gamma &:=& (\perp \gamma) \psi_M - \psi_N \gamma ;
\\
\mathfrak{d}_1  \zeta &:=& \big( (\perp \zeta) \psi_M  + (\perp \psi_N) \zeta \big ) - \Delta_N \zeta  .
\end{eqnarray*}

The complex $(\hom (M, \perp^\bullet N),d)$ has the same cohomology as its 
 normalized complex (as in Definition \ref{defn:normalized}, this is defined as a subcomplex); in degree zero these coincide, whereas in degree one $\hom (M, \perp N)$ is replaced by  the kernel 
$$
\overline{\hom (M, \perp N) }
:= 
\ker \{
\hom (M, \perp N) \stackrel{\hom (M, \epsilon_N) } {\longrightarrow } \hom (M, N)\}.$$ 
Thus, to calculate the cohomology in degrees $0$ and $1$, in the complex (\ref{eq:cx_H0H1}), one can replace $\hom (M, \perp N) $ by $\overline{\hom (M, \perp N) }$.

The identification of the kernel of $\mathfrak{d}_0$ (and hence of $H^0 ( \hom (M, \perp^\bullet N))$)  follows from Proposition \ref{prop:stub_cosimp}. It remains to show that  $H^1( \hom (M, \perp^\bullet N))$ is isomorphic to $\extbar^1_\perp (M, N)$.

An element of $\extbar^1_\perp (M, N)$ represents an equivalence class of extensions of $\perp$-comodules of the form 
\[
0 
\rightarrow 
N 
\rightarrow 
E 
\rightarrow 
M 
\rightarrow 
0
\]
such that the underlying short exact sequence splits in $\calc$, so that $E \cong M \oplus N$. Fix a choice of this splitting. 

It follows that the structure morphism $ \psi_E : E \rightarrow \perp E$ is determined by $\psi_N$, $\psi_M$ and a morphism $\zeta \in \hom (M, \perp N)$. 
 Indeed, such a $\zeta$ gives a comodule structure on $E$ (with respect to the given splitting in $\calc$ as above) if and only if $\epsilon_N \zeta=0$ and $\mathfrak{d}_1 \zeta =0$, since the first condition is equivalent to the counit axiom and the second to the coassociativity axiom.

Hence, the above construction defines a surjection:
\[
\ker \big\{ \overline{\hom (M, \perp N)}
\stackrel{\mathfrak{d}_1}{\rightarrow}
\hom (M, \perp \perp N) \}
\rightarrow 
\extbar^1 _\perp (M, N).
\]
It remains to show that  the kernel of this surjection is the image of $\mathfrak{d}_0$.  

The morphism $\zeta$ maps to $0$ in $\extbar^1 _\perp (M, N)$ if and only if there exists a retract 
of $N \hookrightarrow E$ in $\perp$-comodules. With respect to the given isomorphism $E \cong M \oplus N$ in $\calc$, this is determined by a morphism $\gamma : M \rightarrow N$ in $\calc$.

It is straightforward to check that $\gamma$ induces a morphism of comodules $E_\zeta \rightarrow N$ (where $E_\zeta $ denotes $E$ equipped with the $\perp$-comodule structure corresponding to $\zeta$) if and only if $\zeta = - \mathfrak{d}_0 \gamma$.
\end{proof}

In the rest of this section we study the extra structure obtained 
when the category $\calc$ is tensor abelian with respect to $\otimes$ and the functor $\perp$ is symmetric monoidal.

The following is standard:

\begin{lem}
\label{lem:symm_mon_perp}
The tensor product induces a symmetric monoidal structure on $\calc_\perp$; the comodule structure morphism of $M \otimes N$ is 
\[
M \otimes N \stackrel{\psi_M \otimes \psi_N} {\rightarrow } (\perp M) \otimes (\perp N)
\cong \perp (M\otimes N).
\]
\end{lem}

The behaviour of (relative) extensions under tensor product with a comodule is described by the following:

\begin{prop}
\label{prop:A_otimes_extbar}
Let $A, M,N$ be $\perp$-comodules. Then 
\begin{enumerate}
\item 
the functor $A \otimes - : \calc \rightarrow \calc $ induces a morphism of abelian groups 
\[
\extbar^1 _\perp (M, N) 
\rightarrow 
\extbar^1 _\perp (A \otimes M, A \otimes N);
\]
\item 
under this morphism, the extension represented by $\zeta : M \rightarrow \perp N$ is sent to the extension represented by 
\[
\psi _A \otimes \zeta : A \otimes M \rightarrow (\perp A) \otimes (\perp N) \cong \perp (A \otimes N). 
\] 
\end{enumerate}
\end{prop}

\begin{proof}
The result follows from the definitions and by analysis of (the proof of) Theorem \ref{thm:extbar}.
\end{proof}

\providecommand{\bysame}{\leavevmode\hbox to3em{\hrulefill}\thinspace}
\providecommand{\MR}{\relax\ifhmode\unskip\space\fi MR }
\providecommand{\MRhref}[2]{%
  \href{http://www.ams.org/mathscinet-getitem?mr=#1}{#2}
}
\providecommand{\href}[2]{#2}

\end{document}